\documentclass{article}           
\usepackage[a4paper]{geometry}
\usepackage[utf8]{inputenc}

\usepackage[square,numbers]{natbib}

\usepackage{amssymb}
\usepackage{amsmath}
\usepackage{pifont}
\usepackage{xcolor}
\usepackage{amsthm} 
\usepackage{enumerate}

\usepackage{algorithmic}

\usepackage[ruled,titlenumbered,vlined]{algorithm2e}
\usepackage{float}

\SetKwInput{Input}{Input}
\SetKwInput{Output}{Output}
\SetKw{KwOf}{of}

\floatstyle{plain}
\newfloat{myalgo}{tb}

\newenvironment{Algorithm}[2][tb]%
{\begin{myalgo}[#1]
\centering
\begin{minipage}{#2}
\begin{algorithm}[H]}{\end{algorithm}
\end{minipage}
\end{myalgo}}

\usepackage{tabularx}
\usepackage{color}

\usepackage{url}
\usepackage[plainpages=false,pdfpagelabels,colorlinks=true,citecolor=blue,hypertexnames=false]{hyperref}  

\usepackage{pgf}
\usepackage{tikz}

\tikzstyle{rond}=[draw,fill,minimum size=1mm] 
\tikzstyle{every child}=[level distance=2.5mm]
\tikzstyle{level 1}=[sibling distance=4mm]
\tikzstyle{level 2}=[sibling distance=2.8mm]
\tikzstyle{every node}=[circle,inner sep=0,minimum size=0mm] 
\tikzstyle{mat}=[ampersand replacement=\&,rectangle,matrix anchor =west ]

\newcommand{\nrp}{node[rond]{}}

\definecolor{mygreen}{rgb}{0.10,0.70,0.20}  																	

\newcounter{example}
\newenvironment{example}{\refstepcounter{example} 
\smallskip
\noindent {\em Example~\theexample.} }   
{\smallskip}

\newcommand{\eset}{\varnothing}

\def\bd{\boldsymbol}
\def\bs#1{{\boldsymbol #1}}   
\def\bc#1{\boldsymbol{{\mathcal #1}}}

\def\N{{\mathbb N}}
\def\Id{{\mathbf{Id}}} 
\def\SP{{\bc S\!\bc P}} 
\def\val{\operatorname{val}}

\def\Seq{\textsc{Seq}}
\def\Cyc{\textsc{Cyc}} 
\def\Set{\textsc{Set}}

\def\PSet{\textsc{PSet}}

\def\Sh{\operatorname{str}}
\def\ds{\displaystyle } 

\def\M{{\mathsf M}}

\newcommand{\jac}[1]{\bs \partial \bc H/\bs \partial \bc Y({\mathcal Z}, \bc Y^{[#1]})}
\newcommand{\bjac}[1]{\bs \partial \bc H/\bs \partial \bc Y(\bc Z, \bc Y^{[#1]})}

\newtheorem{theorem}{Theorem}[section]
\newtheorem{lemma}[theorem]{Lemma}
\newtheorem{proposition}[theorem]{Proposition}
\newtheorem{corollary}[theorem]{Corollary}
\newtheorem{property}[theorem]{Property}
\theoremstyle{definition}
\newtheorem{definition}[theorem]{Definition}

\title{Algorithms for Combinatorial Structures:\\ Well-Founded Systems and Newton Iterations\footnote{Partial support was provided by the grant ANR 2010 MAGNUM BLAN 0204 and by the Microsoft-Research-Inria Joint Centre.}}
\author{Carine Pivoteau\footnote{Universit\'e Paris-Est, LIGM (CNRS UMR 8049), Marne-la-Vall\'ee, France.}\and Bruno Salvy\footnote{Inria}\and Mich{\`e}le Soria\footnote{Université Pierre et Marie Curie}
}

\begin{document}
\date{\today}
\maketitle

\begin{quote}
\hfill \begin{minipage}{4.5cm}\em This article is dedicated to the 
memory of Philippe Flajolet\end{minipage}
\end{quote}

\begin{abstract}
We consider systems of recursively defined combinatorial structures. We give algorithms checking that these systems are well founded, computing generating series and providing numerical values.  
Our framework is an articulation of the constructible classes of Flajolet \& Sedgewick with Joyal's species theory. 
We extend the implicit species theorem to structures of size zero. A quadratic iterative Newton method is shown to solve well-founded systems combinatorially. From there, truncations of the corresponding generating series are obtained in quasi-optimal complexity. This iteration transfers to a numerical scheme that converges unconditionally to the values of the generating series inside their disk of convergence. These results provide important subroutines in random generation. Finally, the approach is extended to combinatorial differential systems. 
\end{abstract}

\tableofcontents               

\bigskip      

\section*{Introduction}\label{sec:intro}
Generating series play a central role in enumerative combinatorics. They obey functional equations derived from decompositions of combinatorial structures. These equations offer a route of choice to the enumeration sequences of these structures: they let one compute the first terms of these sequences, they sometimes lead to a closed formula for the $n$th term, and often to its asymptotic behavior. Reference books on this topic include Stanley's \emph{Enumerative Combinatorics}~\cite{Stanley1986,Stanley1999}, the treatise on species theory by Bergeron, Labelle and Leroux~\cite{BeLaLe98} and the recent \emph{Analytic Combinatorics} by Flajolet and Sedgewick~\cite{FlSe09}.

We explore this area from the computational perspective. We present an algorithmic toolkit that starts from a system of recursive combinatorial equations and produces an efficient computation of enumeration sequences and numerical values of the corresponding series. The central idea is to provide an iteration scheme   converging to the combinatorial solution,
and transfer this iteration scheme, both at the series and numerical levels.

Our work is motivated in particular by the needs of random generation in discrete simulation. The recursive method~\cite{FlZiVa94} requires the coefficients of generating series for indices up to the size of the objects being generated. 
This method is {exact}, in the sense that it inputs a size and returns an object of that size, uniformly at random among all objects of its size. 
The more recent Boltzmann sampler~\cite{DuFlLoSc04,FlFuPi07} can draw much larger objects with this uniformity property, the size itself becoming a random variable. This sampler relies on an \emph{oracle}, that computes numerical values of the generating series inside their disk of convergence. We provide such an oracle for a large class of combinatorial structures and also give fast algorithms for the computation of enumeration sequences.

We articulate the combinatorial framework of \emph{species}~\cite{BeLaLe98,Joyal81} with the framework of \emph{constructible classes}~\cite{FlSe09}\footnote{Except for the powerset operator, that we treat separately at the end of this article (\S\ref{sec:powerset}).}: our results hold for
combinatorial structures defined by systems of equations using the operations of union (denoted by `$+$'), Cartesian product (denoted by `$\cdot$'), grouping in a set ($\Set$), a sequence ($\Seq$), or a cycle ($\Cyc$), possibly with cardinality restrictions. 
There are actually two enumeration problems for such combinatorial classes. The \emph{labeled} one deals with structures whose individual atoms are all considered as distinct. In the \emph{unlabeled} enumeration problem, the individual atoms are considered as identical, and it is necessary to account for internal symmetries of the structures.
Many recursive structures fall into this framework; numerous examples can be found in the literature, see e.g., \cite{BeLaLe98,FlSe09,Stanley1999}.
Illustrations in this article are based on typical equations describing trees: ${\mathcal T}={\mathcal Z}\cdot\Seq({\mathcal T})$ for Catalan trees, i.e. planar trees whose nodes have unbounded arity; ${\mathcal G}={\mathcal Z}\cdot\Set({\mathcal G})$ for Cayley trees, i.e., unordered rooted trees;  and a system describing series-parallel graphs:
\[\{{\mathcal C}={\mathcal Z}+{\mathcal S}+{\mathcal P},{\mathcal S}=\Seq_{\ge2}({\mathcal Z}+{\mathcal P}),{\mathcal P}=\Set_{\ge2}({\mathcal Z}+{\mathcal S})\}.
\]
(The precise meaning of these equations is described in Section~\ref{species}.)

Our main result concerning enumeration consists of algorithms that are quasi-optimal: their complexity is linear, up to logarithmic factors, in the size of their output. More precisely, we show that for any constructible class, the first $N$ terms of both the {unlabeled} and {labeled} enumeration problems can be computed in~$O(N\log N)$ arithmetic operations; the required number of bit operations is~$O(N^2\log^2N\log\log N)$ for the \emph{unlabeled} problem and $O(N^2\log^3N\log\log N)$ for the \emph{labeled} problem. We also give efficient numerical algorithms computing the values of the generating series of constructible classes inside their disk of convergence.

The key tool in this work is \emph{a combinatorially meaningful Newton iteration}. This originates in the work of Labelle and his co-authors~\cite{DeLaLe82,Labelle86b,Labelle86}. The combinatorial basis of the iteration leads to a numerical iteration which is always convergent.
In the classical numerical context, under good conditions, Newton's iteration converges to a root that depends on the choice of its initial point, usually close to the root. In our combinatorial context, we show that when started at the origin, the iterates always converge to the solution corresponding to the generating series of interest rather than to a closer one. This is illustrated in Figure~\ref{fig:numerical_newton}, where for each value of $z$ in an interval, we have plotted the real solutions $(z,C_0)$ of the system of equations over generating functions corresponding to a combinatorial structure ${\mathcal C}_0$, defined by a recursive combinatorial specification. The curve marked in red corresponds to \emph{the actual} generating series for ${\mathcal C}_0$. Newton's iteration converges to this solution, and the crosses in the zoomed area indicate the successive values of Newton's iteration starting from $C_0=0$, for $z=0.275$.

The use of Newton's iteration over power series is well-known to be very efficient in terms of complexity, leading to the best known algorithms for many operations and making it a standard tool in computer algebra~\cite{BrKu78,GaGe99}. We show that the systems of equations for generating series of constructible classes can be treated this way, that the iterates converge quadratically to the generating series and that this computation can be performed in good complexity. 
We presented the basic ideas of Newton's iteration on combinatorial systems in the labeled case in~\cite{PiSaSo08}.  
\begin{figure}
\centerline{
\hskip-0cm\begin{minipage}[b]{4cm}
\begin{small}
\begin{align*}
	{\mathcal C}_0 &= {\mathcal Z} {\mathcal C}_1 {\mathcal C}_2 {\mathcal C}_3({\mathcal C}_1+{\mathcal C}_2)\\
	{\mathcal C}_1 &= {\mathcal Z}+{\mathcal Z} \Seq({\mathcal C}_1^2 {\mathcal C}_3^2)\\
	{\mathcal C}_2 &= {\mathcal Z}+{\mathcal Z}^2 \Seq({\mathcal Z} C_2^2 \Seq({\mathcal Z})) \Seq({\mathcal C}_2)\\
	{\mathcal C}_3 &= {\mathcal Z}+{\mathcal Z} (3 {\mathcal Z}+{\mathcal Z}^2+{\mathcal Z}^2 {\mathcal C}_1 {\mathcal C}_3) \Seq({\mathcal C}_1^2)   \\
	 & 
\end{align*}
\end{small}
\end{minipage}\qquad
\includegraphics[height=3.3cm]{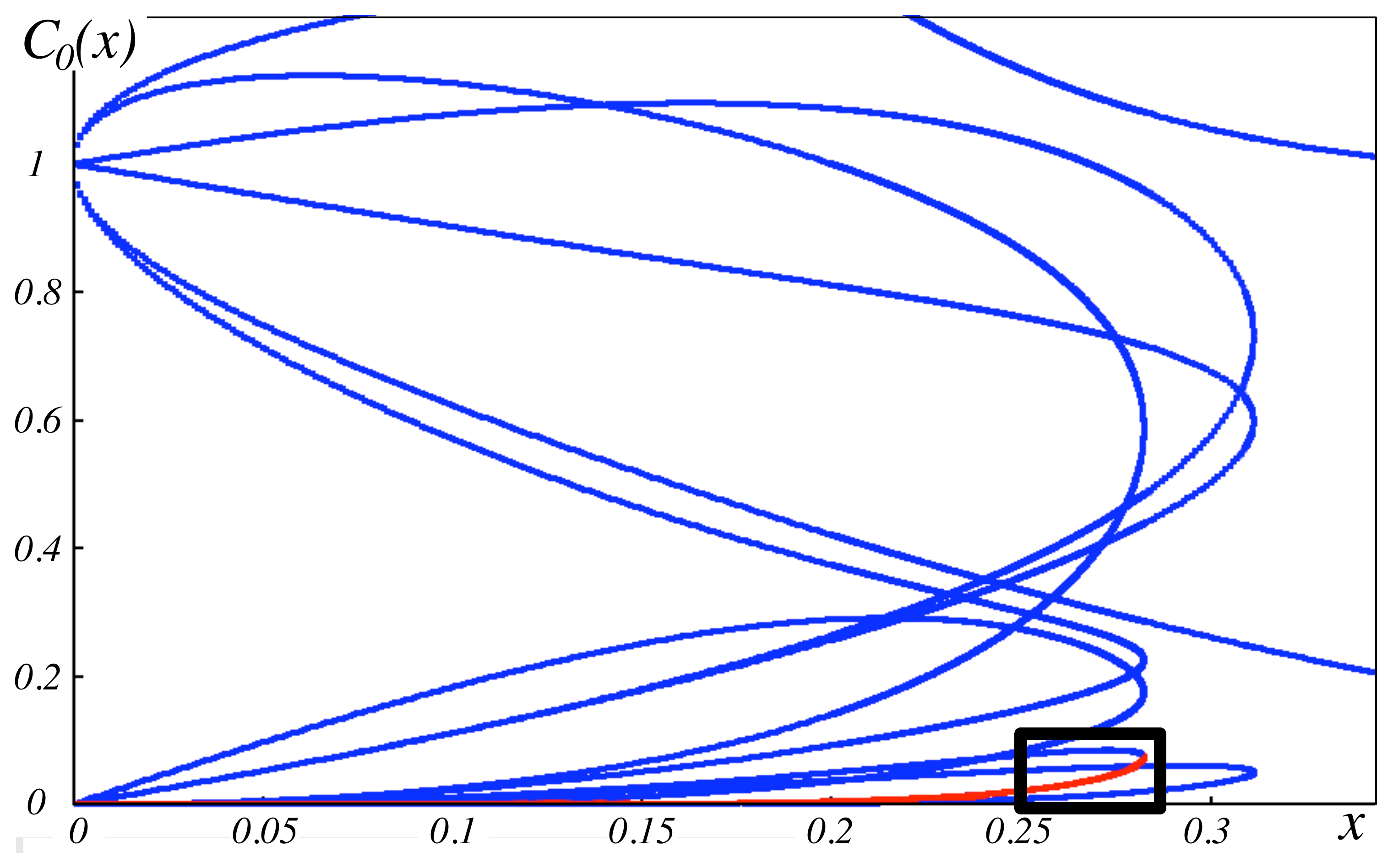}\quad\includegraphics[height=3.3cm]{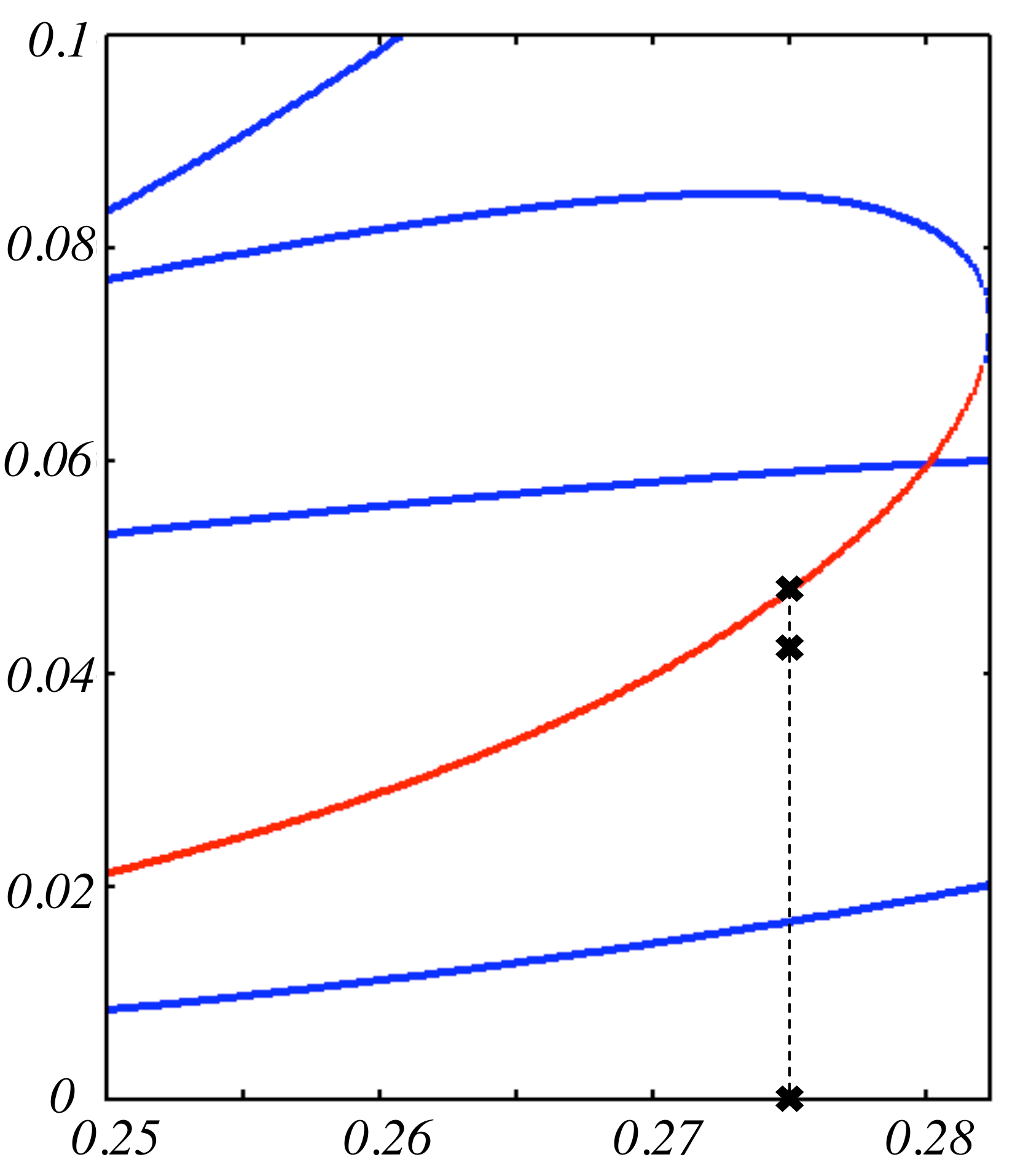}
}\vspace{-3mm}
\caption{\small Combinatorial system (left); real values of~$C_0$ solutions of the corresponding system of equations over generating series (middle), with the generating series in red and the other real solutions in blue, for values of~$z$ between 0 and 0.35; zoom on the rectangular area and iterates of Newton's iteration started at~0 (right).\label{fig:numerical_newton}}
\end{figure}
In the unlabeled case, new difficulties arise since inner symmetries make different labeled objects become identical when the labels are removed. Ordinary generating series do not compose or differentiate well. This is dealt with using \emph{Pólya operators}, that are nicely explained using the theory of species of structures~\cite{BeLaLe98,Joyal81}.
For instance, the generating series of unlabeled Cayley trees above satisfies the functional equation
\[\tilde{G}(z)=z\exp\left(\sum_{j\ge1}\frac{\tilde{G}(z^j)}{j}\right).\]
Using the framework of species of structures, Labelle and his co-authors obtained Newton's iteration for this type of equation.
  In the case of Cayley trees, the resulting Newton operator is
\[{\mathcal Y}\mapsto {\mathcal Y}+\Seq({\mathcal Z}\cdot\Set({\mathcal Y}))\cdot({\mathcal Z}\cdot\Set({\mathcal Y})-{\mathcal Y}),\]
which yields the corresponding Newton operator for power series:
\[\tilde{Y}(z)\mapsto\tilde{Y}(z)+\frac{\tilde{B}(z)-\tilde{Y}(z)}{1-\tilde{B}(z)},\quad\text{with}\quad\tilde{B}(z)=z\exp\left(\sum_{j\ge1}\frac{\tilde{Y}(z^j)}{j}\right).\]
Iterating this operator starting from~0 converges to~$\tilde{G}(z)$ by doubling the number of correct coefficients at each step. Such a convergence is called quadratic.

Newton's iteration on species extends to systems. In this article, we also present an optimized Newton operator that requires fewer operations. For instance, the ordinary generating series of series-parallel graphs are given as solution to the system of functional equations:
\[\tilde{C}(z)=z+\tilde{S}(z)+\tilde{P}(z),\quad \tilde{S}(z)=\frac{1}{1-z-\tilde{P}(z)}-1-z-\tilde{P}(z),\quad\tilde{P}(z)=\exp\left(\sum_{j\ge1}\frac{z^j+\tilde{S}(z^j)}{j}\right)-1-z-S(z).\]
Our method yields a \emph{completely mechanical} derivation of the following efficient iteration (where the upper brackets contain the indices of iteration, and 
 mod~$z^N$ means that the series is truncated at precision~$N$.)
\[
\begin{pmatrix}\tilde{S}^{[n+1]}(z)\\ \tilde{P}^{[n+1]}(z)\end{pmatrix}=
\begin{pmatrix}\tilde{S}^{[n]}(z)\\ \tilde{P}^{[n]}(z)\end{pmatrix}
+{\bd{\tilde U}}^{[n+1]}(z)
\begin{pmatrix}s-1-z-\tilde{P}^{[n]}(z)-\tilde{S}^{[n]}(z)\\ p-1-z-\tilde{S}^{[n]}(z)-\tilde{P}^{[n]}(z)\end{pmatrix}\bmod z^{2^{n+1}}
\]

\[\text{with}\quad s=(1-z-\tilde{P}^{[n]}(z))^{-1}\bmod z^{2^{n+1}},\qquad p=\exp\left(\sum_{j\ge1}({z^j+\tilde{S}^{[n]}(z^j)})/{j}\right)\bmod z^{2^{n+1}},\] 

\[\text{and}\quad{\bd{\tilde U}}^{[n+1]}(z)={\bd{\tilde U}}^{[n]}(z)+{\bd{\tilde U}}^{[n]}(z)
\left(
\begin{pmatrix}0&s^2-1\\ p-1&0\end{pmatrix}
{\bd{\tilde U}}^{[n]}(z)+\Id-{\bd{\tilde U}}^{[n]}(z)\right)\bmod z^{2^{n}}.
\]
	
Initialized with~$\tilde{S}^{[0]}(z)=\tilde{P}^{[0]}(z)=0$ and~${\bd{\tilde U}}^{[0]}(z)=\Id$, 
this iteration converges quadratically to the ordinary generating series~$\tilde{S}(z)=\lim\tilde{S}^{[n]}(z)$ and $\tilde{P}(z)=\lim\tilde{P}^{[n]}(z)$.

Joyal's \emph{Implicit Species Theorem}~\cite{Joyal81} provides the natural context for these operations. It gives conditions under which a square system of combinatorial equations admits a unique vector of species solutions, up to isomorphism.
We extend the implicit species theorem to allow for structures of size~0. This covers all cases of constructible structures we are interested in, and we show that Newton's iteration solves them all. We also show that our definition of well-founded systems is essentially optimal and give an effective criterion to check whether a system is well-founded. In passing, we give a combinatorial interpretation to the iterates in Newton's iteration: they generate the structures of the solution by increasing Strahler number. In order to complete the bridge between species theory and the constructible classes of~\cite{FlSe09}, we define constructible species and analytic species. From there, we prove the analyticity of both exponential and ordinary generating series of the constructible species and give the numerical versions of Newton's iteration. We also deal with the case of integral equations relevant to the study of ordered structures.

{}From the point of view of constructible classes, our contributions are: efficient algorithms for enumeration (cor.~\ref{prop:combSystem}, p.~\pageref{prop:combSystem} and thm.~\ref{th:constructible_linear_species}) improving by a factor~$\log N$ the theoretical arithmetic complexity that can be deduced from the best previous result~\cite{Hoeven2002}; an analysis of the bit complexity of this computation for both ordinary generating series (cor.~\ref{coro:sgo-constructible}, p.~\pageref{coro:sgo-constructible}) and exponential generating series (cor.~\ref{coro:sge-constructible}, p.~\pageref{coro:sge-constructible}); numerical oracles for both exponential (th.~\ref{th:newtonnum}, p.~\pageref{th:newtonnum}) and ordinary generating series  (th.~\ref{th:newt_num_ord}, p.~\pageref{th:newt_num_ord}); a criterion   to decide whether a combinatorial system is well-founded (def.~\ref{def:wf}, p.~\pageref{def:wf}) that is easy to implement; also possibly new is the proof that all constructible classes have an analytic ordinary generating series (th.~\ref{cor:ogf}, p.~\pageref{cor:ogf}). 
As regards random generation, the numerical computations give oracles for the Boltzmann sampler for all constructible classes, and with the algorithms for enumeration, we improve the precomputation stage of the recursive method so that this stage is no longer a limiting factor for the size of objects being generated.

{}From the point of view of species theory, we mainly extend existing ideas to make them applicable to all constructible classes: we give a complete and self-contained presentation of Newton's iteration for implicit species, we treat truncated (cor.~\ref{coro:newton_trunc}, p.~\pageref{coro:newton_trunc}) and nontruncated (th.~\ref{th:newton}, p.~\pageref{th:newton}) variants of Newton's iteration for systems with species of size~0, as well as an optimized version (prop.~\ref{prop:newt_opt}, p.~\pageref{prop:newt_opt}); we deal with polynomial implicit species in detail (sec.~\ref{subsub:partpol}, p.~\pageref{subsub:partpol}); we extend the implicit species theorem to species of size~0 (th.~\ref{th:GIST}, p.~\pageref{th:GIST}); we define analytic species (def.~\ref{def:analyticSpecies}, p.~\pageref{def:analyticSpecies}) as a first step towards analytic combinatorics with species; we completely solve integral systems with Newton's iteration (th.~\ref{newtdiffsys} p.~\pageref{newtdiffsys}).

This article is structured as follows. Part~\ref{part:combi} deals with the combinatorial side of the iteration.
The basic definitions and properties in the theory of species are first recalled, so that this article is self-contained and can be used as a dictionary between the theory of species and the symbolic method of Flajolet and Sedgewick~\cite{FlSe09}. The proof of the implicit species theorem is given using the vocabulary of Bergeron, Labelle and Leroux~\cite{BeLaLe98}. Special classes of species are then presented, including constructible, flat and polynomial species. Then we consider implicit species with structures of size~0. We conclude this section by the combinatorial avatar of Newton's iteration.
Part~\ref{part:computation} deals with the computational side of this work.
Section~\ref{sec:series} is devoted to generating series. Again, we start by recalling the basic facts in the theory, then we present the iterations on power series, analyze their arithmetic complexity and show how the bit complexity can be maintained small. The numerical iteration is treated in Section~\ref{sec:num}. For the computation of numerical values to make sense, the generating series need to be convergent. Accordingly, we define a notion of analytic species and give its basic properties. In particular, constructible species are shown to be analytic. The iterations on power series are then transferred to the numerical domain, using \emph{ad hoc} techniques to deal with Pólya operators in the case of ordinary generating series.
At this stage, all the main results have been presented. Section~\ref{sec:extend} extends many of these results to systems that occur when the integral operator is used to impose orders on the labels of the structures.
We conclude by dealing with the strange case of powersets.

\subsection*{Notations}
We use boldfaced characters for vectors, matrices, or tuples of species; for example, a multisort species ${\mathcal H}({\mathcal Y}_1,{\mathcal Y}_2, \dots, {\mathcal Y}_k)$ is written ${\mathcal H}(\bc Y)$, where $\bc Y$ stands for the vector $({\mathcal Y}_1,{\mathcal Y}_2, \dots, {\mathcal Y}_k)$; and a vector of multisort species $({\mathcal H}_1(\bc Y), {\mathcal H}_2(\bc Y),\dots, {\mathcal H}_m(\bc Y))$  is consistently written $\bc H(\bc Y)$.

We use Gantmacher's notation $a_{1:k}$ to denote the $k$-tuple $(a_1,\dots, a_k)$. Thus, the species~$\bc H(\bc Y)$ can also be written~$\bc H_{1:m}(\bc Y_{1:k})$ if we need its dimensions explicitly.

The coefficient of~$z^n$ in a power series~$f(z)$ is denoted~$[z^n]f(z)$.


\part{Combinatorial Systems}\label{part:combi}


In this part, we explore the combinatorial side of the iteration, within the framework of species of structures.
In Section~\ref{species}, we first recall basic definitions of the theory of species of structures in order to express Joyal's \emph{Implicit Species Theorem} (theorem \ref{th:IST}). Joyal's proof consists in showing that, provided some conditions on $\bc H$ are satisfied,  the iteration
\begin{equation}\label{joyal-iteration}
  \bc Y^{[n+1]}=\bc H\left(\bc Z,\bc Y^{[n]}\right),\qquad \bc Y^{[0]}=\bs{0}
\end{equation}
converges, and  the limit is the unique solution of the system $\bc Y=\bc H(\bc Z,\bc Y)$, $\bc Y(\bs{0})=\bs{0}$, up to isomorphism. 
Section~\ref{seqspecies} is devoted to describing this proof, in the language of species (for which we follow~\cite{BeLaLe98}), and  isolating some building blocks that are used in the rest of the article.

Section~\ref{subseq:wf0} characterizes combinatorial systems that we call \emph{well-founded at $0$}, \textit{i.e.} systems such that $\bc H(\bs{0},\bs{0})=\bs{0}$ and iteration~\eqref{joyal-iteration} converges to a limit without zero coordinates. This constitutes the starting point  for our extension  of the Implicit Species Theorem that includes combinatorial systems allowing for structures of size 0 (theorem \ref{th:GIST}). In Section~\ref{subseq:polspecies},  we first introduce  polynomial species, which have a finite number of structures, and the corresponding notion of partially polynomial species in the multisort case.  Section~\ref{subsec:gen_IST} then focuses on the general notion of well-founded combinatorial systems, where $\bc H(\bs{0},\bs{0})$ is not necessarily $\bs{0}$, providing conditions for Joyal's iteration to converge in this case, and leading to a General Implicit Species Theorem (\ref{th:GIST}).

Joyal's iteration~\eqref{joyal-iteration} is sufficient to derive algorithms for computing enumeration sequences and numerical values of generating series. However, it is well-known that Newton's iteration leads to much better efficiency, at least when it converges. Newton's iteration, lifted to species of structures, writes
\[                                      
\bc Y^{[n+1]}=\bc Y^{[n]} + \left(\Id-\frac{\bc\partial\bc H}{\bc\partial\bc Y}\left(\bc Z,\bc Y^{[n]}\right)\right)^{-1}  \cdot\left(\bc H\left(\bc Z,\bc Y^{[n]}\right)-\bc Y^{[n]}\right).
\]
In Section~\ref{subsec:newton}, we show that Newton's iteration applies whenever the General Implicit Species Theorem holds.
Finally Section~\ref{sec:special} gathers some additional information on special classes of species useful from the analytic point of view of the second part of this article.

\section{Species Theory}\label{species} 
We gather here the basic facts of species theory that we use in this article. We begin by briefly introducing  some  vocabulary, and refer to the book by Bergeron, Labelle, Leroux~\cite{BeLaLe98} for more intuition and examples. A reader familiar with species theory may notice that our notations slightly differ from those in~\cite{BeLaLe98}: ours are borrowed from Flajolet and Sedgewick's \emph{Analytic Combinatorics}~\cite{FlSe09} and are convenient to make a bridge between these two theories, in particular in sections~\ref{sec:series} and~\ref{sec:num}.

\begin{definition}\label{def:species}
A \emph{species of structures}~${\mathcal F}$ is a rule that, for each finite set~$U$ produces a finite set~${\mathcal F}[U]$; and for each bijection~$\sigma:U\rightarrow V$ produces a bijection~${\mathcal F}[\sigma]:{\mathcal F}[U]\rightarrow {\mathcal F}[V]$, in such a way that for two bijections $\sigma$ and $\tau$, ${\mathcal F}[\tau\circ\sigma]={\mathcal F}[\tau]\circ {\mathcal F}[\sigma]$ and ${\mathcal F}[\operatorname{Id}_U]=\operatorname{Id}_{{\mathcal F}[U]}$ (this is called the transport of structures). An element $s$ of ${\mathcal F}[U]$ is called an ${\mathcal F}$-{\em structure} on $U$. The \emph{size} of an ${\mathcal F}$-structure is the cardinality of its underlying~set. An element of ${\mathcal F}[U]$ is graphically depicted as in Figure~\ref{fig:def_esp}, with  dots representing  elements of $U$. 
\end{definition}

\begin{figure}%
  \centering
  \begin{tabular}[c]{lll}
    \begin{minipage}[c]{4cm}
 \centering
      \includegraphics{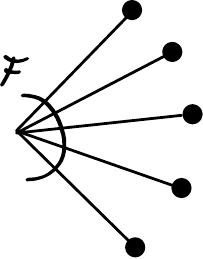}
 \end{minipage}%
    &\qquad&
   \begin{minipage}[c]{10cm}
 \centering
      \includegraphics[scale=0.9]{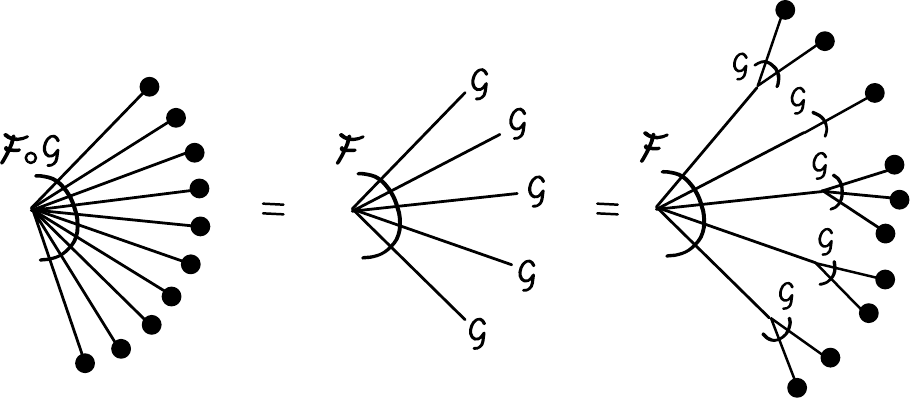} 
 \end{minipage}%
\\
    \begin{minipage}{4cm}
 \centering
      \caption{${\mathcal F}$-structure.}%
      \label{fig:def_esp}
    \end{minipage}%
    &&
    \begin{minipage}{10cm}%
 \centering
       \caption{${\mathcal F}\circ {\mathcal G}$-structure.}%
      \label{fig:def_subs}%
    \end{minipage}%
  \end{tabular}%
\end{figure}

\subsection{Explicit Species}  Species can be defined in different ways. A few special cases are explicit enough to be given directly (in each case, the transport of structures is obvious): the \emph{empty} species, denoted by~$0$, is defined by~$0[U]=\eset$ for all~$U$; 
the species~$1$, characteristic of the empty set, is defined by~$1[U]=\eset$ if~$U\neq\eset$ and~$1[\eset]=\{\eset\}$; the species~$\mathcal Z$ of \emph{singletons} is defined by~${\mathcal Z}[U]=\{U\}$ if~$|U|=1$ and~${\mathcal Z}[U]=\eset$ otherwise.
Among all  nontrivial species, we specially focus on sets, sequences and cycles, that are basic constructors of combinatorial structures in the framework of~\cite{FlSe09}. Examples of structures are given by Figure~\ref{fig:ex-basic-species}.

\begin{itemize}  
	\item The species of {\em sets}, denoted
		by $\Set$ is defined by $\Set[U]=\{U\}$. 
	\item The species ${\mathcal P}$ of permutations defined by ${\mathcal P}[U]=\{\psi: U\rightarrow U \mid \forall v \in U, \exists ! \,u\in U, \psi(u)=v \}$. In particular ${\mathcal P}_n={\mathcal P}[\{1,\dots,n\}]$ denotes the set of permutations over~$\{1,\dots,n\}$. 
  	\item The species $\Seq$ of \emph{sequences} (or linear orders) can be described by $\Seq[\eset]=\{\eset\}$ and for $U=\{u_1,\dots,u_n\}\neq\eset$,
$\Seq[U]=\{(u_{\sigma(1)},\dots,u_{\sigma(n)})\mid\sigma\in{\mathcal P}_n\}$.
			\item The species of {\em cycles}, denoted by $\Cyc$, composed of cyclic ordered lists can be described by $\Cyc[\eset]=\eset$ and for $U\neq\eset$, $\Cyc[U]=\{\sigma\mid\sigma\in{\mathcal P}[U] \mbox{ is composed of a unique cycle}\}$.
\end{itemize}

\begin{figure}[h]%
 \includegraphics[width=\textwidth]{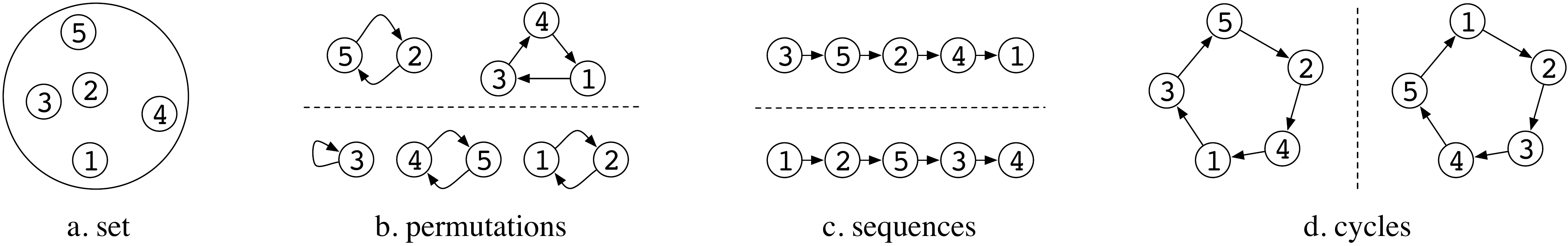}	
\caption{${\mathcal F}$-structures on $U=\{1,\dots,5\}$, with a. ${\mathcal F}=\Set$, b. ${\mathcal F}={\mathcal P}$, c. ${\mathcal F}=\Seq$, d.~${{\mathcal F}=\Cyc}$.}\label{fig:ex-basic-species}
\end{figure}

\subsection{Operations on Species}
Many operations on species are defined, such as sum, product, substitution and differentiation. In this short presentation we only give the action on finite sets, the bijections obeying natural constraints. The \emph{sum} of species is defined by
\[
({\mathcal F}+{\mathcal G})[U]={\mathcal F}[U]+{\mathcal G}[U]
\]
where `$+$' in the right-hand side denotes disjoint union of sets. The symbol~$\sum$ is also used for sums of several species. The \emph{product} of two species~${\mathcal F}$ and~${\mathcal G}$, denoted by ${\mathcal F} \cdot {\mathcal G}$ or ${\mathcal F} {\mathcal G}$, is given by  
\[
({\mathcal F} \cdot {\mathcal G})[U]=\sum_{\substack{(U_1,U_2),\\U=U_1 + U_2}}{{\mathcal F}[U_1]\times {\mathcal G}[U_2]},
\]
where the sum is over all decompositions of~$U$ as a disjoint union and `$\times$' on the right-hand side denotes the Cartesian product. 

Let ${\mathcal F}$ and ${\mathcal G}$ be two species such that ${\mathcal G}[\eset]=\eset$ (there is no ${\mathcal G}$-structure of size~$0$). \emph{Composition} of ${\mathcal F}$ with ${\mathcal G}$ is denoted by ${\mathcal F}\circ {\mathcal G}$ or ${\mathcal F}({\mathcal G})$; the $({\mathcal F}\circ {\mathcal G})$-structures are \emph{${\mathcal F}$-assemblies} whose \emph{members} are  ${\mathcal G}$-structures, more precisely:
\[
({\mathcal F}\circ {\mathcal G})[U]=\sum_{\text{$\pi$ partition of $U$}}{{\mathcal F}[\pi]\times\prod_{p\in\pi}{{\mathcal G}[p]}}.
\]
A graphical description of the composition of species is given in Figure~\ref{fig:def_subs}. Note that, using composition, the property ${\mathcal G}[\eset]=\eset$ is equivalent to ${\mathcal G}(0)=0$.

\subsection{Relations Between Species}\label{parag:rel-species}
Two species ${\mathcal F}$ and ${\mathcal G}$ are \emph{equal} if they produce the same sets and bijections. The definitions in the theory of species are set up in such a way that classical equalities of calculus still hold between species.   
More generally, equality leads to equations and systems, whose solutions we set to study in this work. 

An \emph{isomorphism} from ${\mathcal F}$ to ${\mathcal G}$ is a family of bijections~${\alpha_U:{\mathcal F}[U]\rightarrow {\mathcal G}[U]}$, that makes the expected diagrams commute, that is, for any bijection $\sigma: U \to V$ between two finite sets and for any~${\mathcal F}$-structure~$s$,  ${\mathcal G}[\sigma](\alpha_U (s)) = \alpha_V ({\mathcal F}[\sigma](s))$. 
Even if weaker than equality, isomorphism implies that the structures possess the same combinatorial properties; hence, following~\cite{BeLaLe98},  we  say that there is a \emph{combinatorial equality} between two isomorphic species  ${\mathcal F}$ and ${\mathcal G}$, and write  ${\mathcal F}={\mathcal G}$. For example, the combinatorial equality~${\mathcal F}={\mathcal F}({\mathcal Z})$ holds for any species~${\mathcal F}$. 

Another type of isomorphism exists between \emph{structures} of the same species.
Two ${\mathcal F}$-structures~$s$ and~$t$ over~$\{1,\dots,n\}$ are \emph{isomorphic} when there exists a permutation~$\pi\in{\mathcal P}_n$ such that ${\mathcal F}[\pi](s) = t$. An \emph{isomorphism type} of ${\mathcal F}$-structures over $\{1,\dots,n\}$ is an equivalence class modulo this isomorphism. Such an equivalence class is also called an \emph{unlabeled} ${\mathcal F}$-structure of size~$n$.

The notion of \emph{equipotence} that only replaces set equalities by bijections is even weaker: two species~${\mathcal F}$ and ${\mathcal G}$ are \emph{equipotent} when the numbers of ${\mathcal F}$-structures and ${\mathcal G}$-structures are equal on all finite sets; this is denoted by ${\mathcal F}\equiv{\mathcal G}$. A typical example is that of sequences and permutations: ${\mathcal P}\equiv\Seq$ but ${\mathcal P}\neq\Seq$ since these two species are not transported in the same way along bijections.

A species~${\mathcal F}$ is a \emph{subspecies} of~${\mathcal G}$, denoted by~${\mathcal F}\subset {\mathcal G}$, when for any finite set~$U$, ${\mathcal F}[U]\subset {\mathcal G}[U]$ and for any bijection~$\sigma:U\rightarrow V$, ${\mathcal F}[\sigma]={\mathcal G}[\sigma]|_{{\mathcal F}[U]}$.
For ${\mathcal F}\subset {\mathcal G}$, the \emph{subtraction} ${\mathcal H}={\mathcal G}-{\mathcal F}$ is defined by the equation ${\mathcal G}={\mathcal F}+{\mathcal H}$. When~${\mathcal F}\subset {\mathcal G}$ with ${\mathcal G}(0)=0$, the inclusion is preserved by composition with an arbitrary~${\mathcal H}$: ${\mathcal H}\circ{\mathcal F}\subset{\mathcal H}\circ{\mathcal G}$. 
Two species~${\mathcal F}$ and ${\mathcal G}$ are called \emph{disjoint} if for all finite sets~$U$, ${\mathcal F}[U]\cap{\mathcal G}[U]=\eset$. If the species~${\mathcal F}$ and ${\mathcal G}$ are subspecies of ${\mathcal H}$ and they are disjoint, then~${\mathcal F}+{\mathcal G}\subset {\mathcal H}$.

\subsection{Derivative and Related Species}
The \emph{derivative} ${\mathcal F}'$ of a species ${\mathcal F}$ is defined by~${\mathcal F}'[U]={\mathcal F}[U+\{\star\}]$, where~$\star$ is an element chosen outside of~$U$.  For instance, derivatives of the explicit species introduced earlier are given by Table~\ref{tab:deriv}. 

\begin{table}[ht]
\begin{center}
\begin{tabular}{|l|l|l|l|l|l|l|l|l|} 
\hline
species & 0& 1&${\mathcal Z}$&${\mathcal A}+ {\mathcal B} $ & ${\mathcal A}\cdot {\mathcal B}\phantom{\Bigl|}$ & $\Seq $ & $\Set $ & $\Cyc $ \\[2pt]
\hline
derivative$\phantom{\Bigl|}$ & 0& 0&1& ${\mathcal A}'+ {\mathcal B}'$ & ${\mathcal A}'\cdot {\mathcal B} + {\mathcal A}\cdot {\mathcal B}'$ & $\Seq  \cdot \Seq $ & $\Set $ & $\Seq $ \\
\hline
\end{tabular}            
\end{center}
\caption{Derivatives of classical species.}\label{tab:deriv}	
\end{table}

An $ {\mathcal H}'({\mathcal Z})$-structure can be interpreted as an ${\mathcal H}$-assembly  where the element $\star$ (called a~\emph{bud} by Labelle~\cite{Labelle85b}) marks one of the possible locations for a singleton~${\mathcal Z}$ (see Figure~\ref{fig:def_diff_z}).

\begin{figure}%
  \centering
  \begin{tabular}[c]{lll}
    \begin{minipage}[c]{7.3cm}
 \centering
      \includegraphics[scale=0.95]{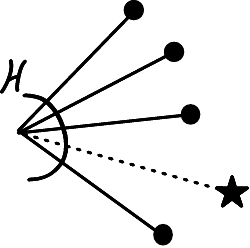}
 \end{minipage}%
    &\qquad&
   \begin{minipage}[c]{8cm}
 \centering
      \includegraphics[scale=1]{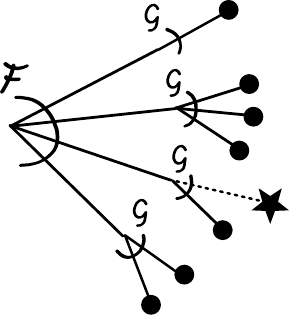}
 \end{minipage}%
\\
    \begin{minipage}{7.3cm}
 \centering
      \caption{${\mathcal H}'$-structure.}%
      \label{fig:def_diff_z}
    \end{minipage}%
    &&
    \begin{minipage}{8cm}%
 \centering
       \caption{$({\mathcal F}\circ {\mathcal G})'$-structure}%
      \label{fig:def_subs_diff}%
    \end{minipage}%
  \end{tabular}%
\end{figure}

For example, the derivative of the composition of two species is given by $({\mathcal F}\circ {\mathcal G})'=({\mathcal F}'\circ {\mathcal G})\cdot{\mathcal G}'$. The interpretation is the following: to replace one of the singletons of~${\mathcal F}\circ{\mathcal G}$ by a $\star$, one first marks the branch of the~${\mathcal F}$-structure where this is going to take place and then grafts on this branch a~${\mathcal G}$-structure with one of its element replaced by a~$\star$, \emph{i.e.}, a ${\mathcal G}'$-structure (see Figure~\ref{fig:def_subs_diff}).

For any species~${\mathcal H}$ and any two species~${\mathcal A}\supset {\mathcal B}$, such that ${\mathcal A}(0)=0$, the following inclusion holds, up to isomorphism:
\begin{equation}\label{eq:taylor1}
{\mathcal H}({\mathcal A})\supset{\mathcal H}({\mathcal B})+{\mathcal H}'({\mathcal B})\cdot({\mathcal A}-{\mathcal B}).
\end{equation}
The interpretation is as follows: the structures on the right-hand side are either in~${\mathcal H}({\mathcal B})$, that is to say ${\mathcal H}$-assemblies of only ${\mathcal B}$-structures; or in the disjoint species~${\mathcal H}'({\mathcal B})\cdot {\mathcal A}$ whose structures are ${\mathcal H}$-assemblies whose members are~${\mathcal B}$-structures, except for exactly one member which is an~${\mathcal A}$-structure and not a ${\mathcal B}$-structure. 
Labelle actually developed a complete Taylor formula in this context~\cite{Labelle90} that generalizes this inclusion.

\subsection{Multisort Species}
Species can also be defined for structures constructed on sets with several sorts of elements, as for functions of several variables. Such a species is called a \emph{multisort} species, and denoted by ${\mathcal F}[U_1,\dots,U_k]$. It produces a set from each $k$-tuple of finite sets~$U_1,\dots,U_k$. 
Then, the size of a multisort structure is the sum of the cardinalities of its underlying sets.

The operations of sum and product easily extend to multisort species. 
For composition, the multisort analogue is more complicated: we present for example the case of an~${\mathcal H}$-assembly of~${\mathcal G}_1$ and~${\mathcal G}_2$ structures, where~${\mathcal H}(\bc Z)$ is two-sort, while~${\mathcal G}_1$ and~${\mathcal G}_2$ are unisort:
\begin{equation}\label{eq:2sort-comp}
{\mathcal H}({\mathcal G}_1,{\mathcal G}_2)[U]=\sum_{\substack{\text{$\pi$ partition of~$U$}\\ \pi_1+\pi_2=\pi}}{{\mathcal H}[\pi_1,\pi_2]\times\prod_{s\in\pi_1} {\mathcal G}_1[s]\times\prod_{t\in\pi_2} {\mathcal G}_2[t]}.
\end{equation}
For instance, sums and products are special cases of multisort species: $+({\mathcal G}_1,{\mathcal G}_2)[U]={\mathcal G}_1[U]+{\mathcal G}_2[U]$ and $\cdot({\mathcal G}_1,{\mathcal G}_2)[U]={\mathcal G}_1[U]\cdot{\mathcal G}_2[U]$ are obtained by defining $+[U,V]$ as~$\{\eset\}$ if either $|U|=1$ or $|V|=1$ 
and $\cdot[U,V]$ as $\{\eset\}$ when $|U|=|V|=1$ and $\eset$ otherwise. Thus, in the sequel, we consider these operations as species.

The notion of derivative also extends to multisort species: for a $k$-sort species ${\mathcal H}({\mathcal Y}_{1:k})$, one sets
\[
\frac{ \partial {\mathcal H}}{\partial{\mathcal Y}_{i}}[U_1,\dots,U_k]= {\mathcal H}[U_1,\dots,U_{i-1},U_i+\{\star_i\},U_{i+1},\dots,U_k].
\]

A $\partial{\mathcal H}/\partial{\mathcal Y}_{i}$-structure can be interpreted as an ${\mathcal H}$-assembly  where the bud~$\star_i$ of sort $i$ marks one of the possible locations for a ${\mathcal Y}_{i}$-structure. Figure~\ref{fig:def_diff} illustrates the case of a two-sort species ${\mathcal H}({\mathcal Z},{\mathcal Y})$, where dots represent the species ${\mathcal Z}$.   
For instance, the  structures in the  product species 
\[
\frac{\partial{\mathcal H}}{\partial {\mathcal Y}}({\mathcal Z},{\mathcal Y}) \cdot \frac{\partial{\mathcal H}}{\partial {\mathcal Y}}({\mathcal Z},{\mathcal Y}) = \left(\frac{\partial{\mathcal H}}{\partial {\mathcal Y}}({\mathcal Z},{\mathcal Y})\right)^2
\]
are~${\mathcal H}$-assemblies whose members are singletons and  ${\mathcal Y}$-structures, except for one member which is a ${\partial{\mathcal H}}/{\partial {\mathcal Y}}$-structure (in the location for a ${\mathcal Y}$-structure), as depicted by Figure~\ref{fig:ex_dL2}. 
More generally, a sequence~$\Seq\left({\partial {\mathcal H}}/{\partial {\mathcal Y}}({\mathcal Z},{\mathcal Y})\right)$ consists of trees built up by iterating this process.

The derivative of a composition behaves as in the classical case. For example, the composition of the species~${\mathcal F}({\mathcal X},{\mathcal Y})$ with two unisort species~${\mathcal G}_1$ and ${\mathcal G}_2$ is differentiated as
\begin{equation}\label{eq:diffcomp}
({\mathcal F}({\mathcal G}_1,{\mathcal G}_2))'=
\frac{\partial{\mathcal F}}{\partial{\mathcal X}}\cdot{\mathcal G}_1'+
\frac{\partial{\mathcal F}}{\partial{\mathcal Y}}\cdot{\mathcal G}_2'.
\end{equation}

\begin{figure}%
  \centering
  \begin{tabular}[c]{lll}
    \begin{minipage}[c]{8.3cm}
 \centering
      \includegraphics[scale=0.95]{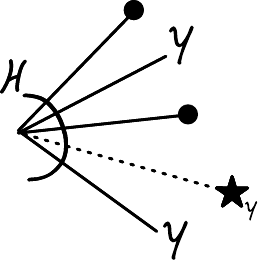}
 \end{minipage}%
    &\qquad&
   \begin{minipage}[c]{7cm}
 \centering
      \includegraphics[scale=1]{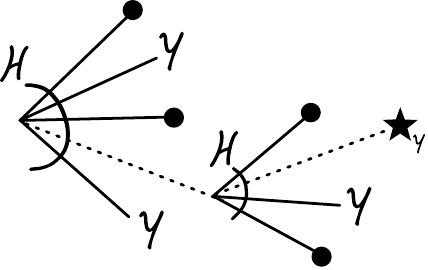}
 \end{minipage}%
\\
    \begin{minipage}{8.3cm}
 \centering
      \caption{A structure of the species $\cal\partial {\mathcal H}({\mathcal Z},{\mathcal Y})/\cal\partial{\mathcal Y}$}%
      \label{fig:def_diff}
    \end{minipage}%
    &&
    \begin{minipage}{7cm}%
 \centering
       \caption{$(\cal\partial {\mathcal H}/\cal\partial{\mathcal Y})^2$-structure}%
      \label{fig:ex_dL2}%
    \end{minipage}%
  \end{tabular}%
\end{figure}

\subsection{Jacobian Matrix}
Matrices and vectors of species are defined as usual; they can likewise be viewed as species whose structures are matrices or vectors, the size of a structure being the sum of the sizes of its components.
The product of a matrix by a matrix or a vector is given by the usual rules, sums and products being replaced by sums and products of species. The identity matrix for species, denoted by $\Id$, is naturally defined as the matrix whose entries are the species~$1$ on the diagonal and~$0$ anywhere else.

Let~$\bc H=({\mathcal H}_{1:m})$ be a vector of $m+1$-sort species, and let $\bc Y=({\mathcal Y}_{1:m})$ be a vector of species. As in the classical case, the \emph{Jacobian matrix} of the vector of species~$\bc H(\bc{Z}, \bc Y)$
with respect to~$\bc Y$, denoted by~$\bc\partial\bc H/\bc\partial\bc Y$, is the matrix whose entry $(i,j)$ is $\partial {\mathcal H}_i(\bc{Z}, \bc Y)/\partial {\mathcal Y}_j$ for $1\le i,j\le m$. 
Finally, a matrix of combinatorial species is nilpotent if one of its powers is~$\bs{0}$ (all its entries are the~0 species). The order of nilpotence (the minimal power such that~$\bs{0}$ is reached) is bounded by the dimension of the matrix, exactly like in classical linear algebra.

\begin{example}\label{ex:spg} Series-parallel graphs are specified by~$({\mathcal Y}_1,{\mathcal Y}_2)=\SP({\mathcal Z},{\mathcal Y}_1,{\mathcal Y}_2)$, with
\[\SP({\mathcal Z},{\mathcal Y}_1,{\mathcal Y}_2)=\begin{pmatrix}\Seq_{\ge2}({\mathcal Z}+{\mathcal Y}_2)\\ \Set_{\ge2}({\mathcal Z}+{\mathcal Y}_1)\end{pmatrix},\]
denoting by~$\Seq_{\ge k}$ 
the species $\Seq$ restricted to structures of size at least~$k$ and similarly for~$\Set_{\ge k}$. Linearity of the derivative implies that
\[\frac{\partial\Seq_{\ge k}({\mathcal Y})}{\partial{\mathcal Y}}=\Seq_{\ge k-1}({\mathcal Y})\times{\mathcal Y}^\star\times\Seq({\mathcal Y})+{\mathcal Y}^\star\times\Seq_{\ge k-1}({\mathcal Y}),\quad
\frac{\partial\Set_{\ge k}({\mathcal Y})}{\partial{\mathcal Y}}=\Set_{\ge k-1}({\mathcal Y})\times{\mathcal Y}^\star.\]
The Jacobian matrix is therefore
\[
\frac{\bc\partial\SP}{\bc\partial\bc Y}=\begin{pmatrix} 
0&\Seq_{\ge1}({\mathcal Z}+{\mathcal Y}_2)\times{\mathcal Y}_2^\star\times\Seq({\mathcal Z}+{\mathcal Y}_2)+{\mathcal Y}_2^\star\times\Seq_{\ge1}({\mathcal Z}+{\mathcal Y}_2)\\ 
\Set_{\ge 1}({\mathcal Z}+{\mathcal Y}_1)\times {\mathcal Y}_1^\star&0 \end{pmatrix}.
\]       
This matrix evaluated at~$({\mathcal Z},\bc Y)=(0,\bs{0})$ gives $\left(\begin{smallmatrix}0&0\\ 0&0\end{smallmatrix}\right)$, which makes it nilpotent of order~1.
Considering graphs that are either series or parallel graphs, leads to a system with a third equation~${\mathcal Y}_3={\mathcal Y}_1+{\mathcal Y}_2$. In this extended case, the  $3\times 3$ Jacobian matrix at~$({\mathcal Z},\bc Y)=(0,\bs{0})$ is nilpotent of order 3.
\end{example}
 
\paragraph{Combinatorial interpretation of the Jacobian matrix}
The Jacobian matrix plays an important role in the characterization of species implicitly defined by a system of equations $\bc Y=\bc H(\bc Z,\bc Y)$. Such a system can be seen as a set of rewriting rules  stating how to construct the coordinates of $\bc Y$, and the Jacobian matrix $\bc J= \bc\partial\bc H/\bc\partial\bc Y$ encodes a valued dependency graph of the system. Each entry ${(i,j)}$ of $\bc J$ expresses how the species ${\mathcal Y}_i$ depends on~${\mathcal Y}_j$.

The $p$-th power of the Jacobian matrix thus describes the paths of length $p$ in the dependency graph.
When $\bc J^p(\bc Z,\bc Y)=\bs{0}$ (the matrix is nilpotent), the graph has no cycle; this will be a crucial condition for the finiteness of the number of structures in the solution (Prop.~\ref{prop:polynomial_implicit}). 
The weaker condition $\bc J^p(\bs{0},\bs{0})=\bs{0}$ is one of the basic conditions for the implicit species theorem to hold (Theorem~\ref{th:IST}).
It corresponds to the absence of cycles preserving the size of structures.

\section{Joyal's Implicit Species Theorem}\label{seqspecies}

This section is devoted to Joyal's implicit species theorem, which constitutes a pillar in the theory of species, since it gives a meaning to solutions of equations. Our interest in this presentation is an analysis of the proof, aiming both at introducing notions and techniques on species that will be useful in the rest of our article, and at focusing on the hypotheses of this theorem, that we extend later.

\subsection{Implicit Species}
We consider vectors of species, implicitly defined by a recursive square system of combinatorial equations $\bc Y=\bc H(\bc Z,\bc Y)$, where $\bc Y=({\mathcal Y}_{1:m})$ and $\bc H=({\mathcal H}_{1:m})$ are vectors of species.
The Implicit Species Theorem~\cite{Joyal81} requires hypotheses ensuring that  such a system actually defines a species of structures. 
The first condition,  $\bc H(\bs{0},\bs{0})=\bs{0}$, is a restriction on species, implying that there is no structure on the empty set (we give conditions to remove this restriction in Section~\ref{subsec:gen_IST}). The second condition, on the nilpotence of the Jacobian matrix, prevents from building infinitely many structures of the same size.

\begin{theorem}[Implicit Species Theorem \cite{Joyal81}] \label{th:IST}
Let~$\bc H$ be a vector of multisort species, with ${\bc H(\bs{0},\bs{0})=\bs{0}}$ and such that the Jacobian matrix ${\bd\partial \bc H}/{\bd\partial\bc Y}(\bs{0},\bs{0})$ is nilpotent.
The system of equations
\begin{equation}\label{impliciteq}
\bc Y=\bc H(\bc Z,\bc Y), \quad\text{where}\quad \bc Y=({\mathcal Y}_{1:m})
\end{equation} 
admits a vector of species solution~$\bc S$  such that~$\bc S(\bs{0})=\bs{0}$, which is unique  up to isomorphism.  
\end{theorem}   

The solution of the implicit system of Theorem~\ref{th:IST} is the species of $\bc H$-rooted trees, that is to say  $\bc H$-assemblies of $\bc Y$-structures, that are, recursively, $\bc H$-rooted trees. A graphical representation of such a system is given in Figure~\ref{fig:H-rooted}, together with a representation of a structure of its solution. A proof of this theorem is given in the next section. A generalization is given in Theorem~\ref{th:GIST}.

\begin{figure}                                           
	\centerline{\includegraphics[scale=0.95]{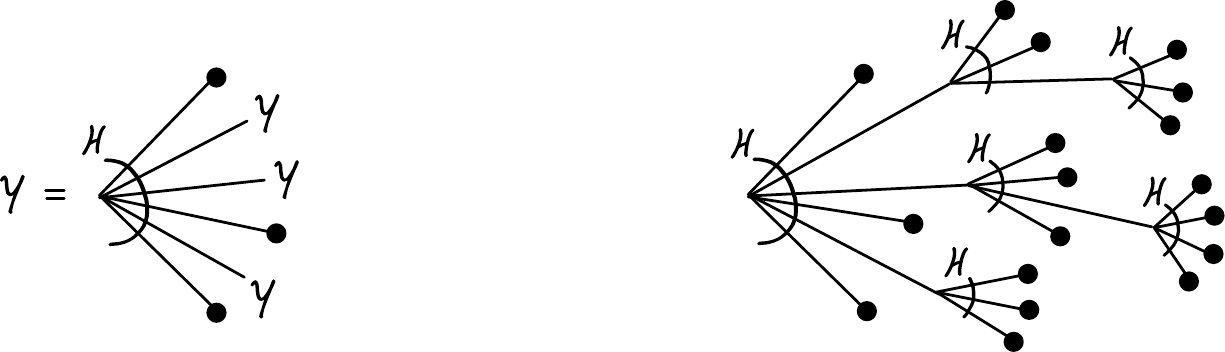}}
\caption{Left: graphical representation of $\bc Y=\bc H({\mathcal Z},\bc Y)$. Right: an example of an $\bc H$-rooted tree. \label{fig:H-rooted}}	                                 
\end{figure}

Joyal~\cite{Joyal81} and Labelle~\cite{Labelle85} give two different constructive proofs of the implicit species theorem. Whereas Labelle's proof is a generalization of the method of blooming, the original proof by Joyal follows the classical proof of the  implicit function theorem, and asserts the existence and uniqueness of the solution of implicit combinatorial systems. Joyal's proof is obtained by constructing an iterated sequence of species that converges (slowly) to the solution. We extract the basic blocks from this proof; they are used further in the rest of this combinatorial section.

\subsection{Contact and Convergence}
Two (possibly multisort) species ${\mathcal F}$ and~${\mathcal G}$ have a \emph{contact} of order~$p$, denoted by ${\mathcal F}=_p{\mathcal G}$, when there exists a species isomorphism from~${\mathcal F}_{\le p}$ to~${\mathcal G}_{\le p}$, where~${\mathcal F}_{\le p}$ denotes the species~${\mathcal F}$ restricted to ${\mathcal F}$-structures of size at most~$p$. Similarly, ${\mathcal F}_{\ge p}$ denotes the restriction to structures of size at least~$p$.

\begin{definition}[Convergence of a sequence of species]
The sequence of species~$({\mathcal Y}^{[n]})_{n\in\N}$ \emph{converges} to a species~${\mathcal Y}$ if for all~$p\ge 0$, there exists~$N\ge0$ such that for all~$n\ge N$, ${\mathcal Y}^{[n]}=_p {\mathcal Y}$. This is denoted by~$\lim_{n\rightarrow\infty}{\mathcal Y}^{[n]}={\mathcal Y}$.         
\end{definition}
In addition, a  sequence of species~$({\mathcal Y}^{[n]})_{n\in\N}$ is \emph{increasing} if ${\mathcal Y}^{[n-1]}\subset  {\mathcal Y}^{[n]}$, for all $n\geq 1$.

\begin{lemma}\label{lemma1} Let~$({\mathcal Y}^{[n]})_{n\in\N}$ be a
		sequence of species and let~$(u_n)_{n\in\N}$ be a sequence of positive integers such that ${\mathcal Y}^{[n]}=_{u_n}{\mathcal Y}^{[n+1]}$. If~$\lim_{n\rightarrow\infty}{u_n}=\infty$, then there exists a species~${\mathcal Y}$ to which the sequence~$({\mathcal Y}^{[n]})_{n\in\N}$ converges. This limit is unique up to isomorphism.
\end{lemma}
\begin{proof}  %
We define the species~${\mathcal Y}$ by giving its values for all sizes~$p\in\N$. Let thus~$p$ be a nonnegative integer. The limit of~$(u_n)$ implies that there exists~$N$ such that for all~$n\ge N$, $u_n\ge p$. Therefore, for all such~$n$, ${\mathcal Y}^{[n]}=_p{\mathcal Y}^{[n-1]}=_p\dotsb=_p{\mathcal Y}^{[N]}$. As a consequence, for all~$n\ge N$, ${\mathcal Y}^{[n]}$ and~${\mathcal Y}^{[N]}$ coincide on all finite sets of size~$p$, as well as on all bijections between them. We then define their common values as those of~${\mathcal Y}$. The properties required of ${\mathcal Y}[\tau\circ\sigma]$ and~${\mathcal Y}[\operatorname{Id}_U]$ then follow from the same properties for~${\mathcal Y}^{[N]}$. By definition of the limit, one then has $\lim_{n\rightarrow\infty}{\mathcal Y}^{[n]}={\mathcal Y}$.

The existence of isomorphic limits is rooted in the definition of limits: a species~${\mathcal W}$ is another limit of~$({\mathcal Y}^{[n]})$ if and only if~${\mathcal Y}=_p{\mathcal W}$ for all~$p\ge0$; this in turn implies the existence of a species isomorphism from~${\mathcal Y}_{\le p}$ to~${\mathcal W}_{\le p}$ for all~$p$, which gives an isomorphism between~${\mathcal Y}$ and~${\mathcal W}$.
\end{proof}

Convergence of vectors or matrices of species is defined as component-wise convergence. The next building block in the proof of the implicit function theorem is the following.
\begin{lemma}\label{lemma2} 
	Let~$(\bc Y^{[n]})_{n\in\N}$ be a sequence of vectors of species converging to~$\bc Y$. If the sequence $\bc H(\bc Z,\bc Y^{[n]})$ also converges to~$\bc Y$, then~$\bc Y$ is a solution of~$\bc Y=\bc H(\bc Z,\bc Y)$.
\end{lemma}
\begin{proof}
In order to show that~$\bc Y=\bc H(\bc Z,\bc Y)$, it is sufficient to prove that both sides of the equation coincide on finite sets and their bijections, which follows from their convergence.
\end{proof}

\subsection{Proof of the Implicit Species Theorem}\label{IST_proof}

\def\jaco{\bc{\partial \bc H}/\bc{\partial \bc Y}(\bs{0},\bs{0})}    

Joyal's proof of the Implicit Species Theorem (in the case when~$m=1$) is based on a sequence of species defined by a simple iteration. 
\begin{proposition}\label{prop:is-cvg} Let $\bc Y=\bc H(\bc Z,\bc Y)$ be a system such that $\bc H(\bs{0},\bs{0})=\bs{0}$ and the matrix~$\jaco$ is nilpotent.
The sequence~$(\bc Y^{[n]})_{n\in\N}$  defined by
\begin{equation}\label{phi-iteration}\tag{\ref{joyal-iteration}}
        \bc Y^{[0]}=\bd{0}\quad\text{and}\quad \bc Y^{[n+1]}=\bc H(\bc Z,\bc Y^{[n]}),\quad n\ge0
\end{equation}
is convergent.
\end{proposition}
Given this property, Lemma~\ref{lemma2} shows that the limit~$\bc Y$ of the sequence~$(\bc Y^{[n]})$ of Proposition~\ref{prop:is-cvg} 
is actually a solution of the system~$\bc Y=\bc H(\bc{Z},\bc Y)$. If $\bc S$ is another solution of the system with~$\bc S(\bs{0})=\bs{0}$, then $\bc S=_0\bc Y^{[0]}=_0\bc Y$ and by induction using Lemma~\ref{lem:increasing-contact} below, $\bc S=_k\bc Y^{[pk]}=_k\bc Y$ for all $k\geq 0$. 
Thus there exists a unique solution with~$\bc S(\bs{0})=\bs{0}$, up to isomorphism.

\begin{proof}[Proof of Proposition~\ref{prop:is-cvg}]      
The first step is to show that~$(\bc Y^{[n]})$ is an increasing sequence of species. This is proved by induction: for~$n=0$, the assertion comes from the definition of the 0 species; then the inclusion~$\bc Y^{[n]}\subset \bc Y^{[n+1]}$ is preserved by composition with $\bc H$: $\bc H(\bc Z,\bc Y^{[n]})\subset\bc H(\bc Z,\bc Y^{[n+1]})$, and by definition of the iteration, this is~$\bc Y^{[n+1]}\subset\bc Y^{[n+2]}$. 

The rest of the proof consists in showing that the sequence $(\bc Y^{[n]})$ converges. This is a consequence of the following lemma, which states that  $p$ iterations of a species $\bc H$ 
with index of nilpotence $p$ increase the contact. We denote by~$\bc H^i$  the $i$th iterate of $\bc Y\mapsto\bc H(\bc Z,\bc Y)$.
\begin{lemma}\label{lem:increasing-contact}
        Let $\bc H(\bc Z,\bc Y)$ be a vector of multisort species such that $\bc H(\bs{0},\bs{0})=\bs{0}$ and the Jacobian matrix~$\jaco$ is nilpotent.
        Let~$\bc A$ and~$\bc B$ be two species such that $\bc B\subset\bc A$. If~$\bc B=_k\bc A$, then~$\bc H^p(\bc Z,\bc B)=_{k+1}\bc H^p(\bc Z,\bc A)$ where~$1\le p\le m$ is the index of nilpotence of the Jacobian matrix. 
\end{lemma}
\begin{proof}
        
The idea is that if an $\bc H$-rooted tree of the species~$\bc H^p(\bc Z,\bc A)-\bc H^p(\bc Z,\bc B)$ had size~$\le k+1$, then one of its branches would contain a structure of~$(\jaco)^p$, but this is $\bs{0}$. The subtraction (and all those that appear in this proof) is defined since inclusion is preserved by composition.
We first show that, for $i\ge 1$, any structure $\gamma$ belonging to~$\bc H^i(\bc Z,\bc A)-\bc H^i(\bc Z,\bc B)$ with size at most~$k+1$ rewrites as a structure of~$\jaco (\bc H^{i-1}(\bc Z,\bc A)-\bc H^{i-1}(\bc Z,\bc B))$. 
By definition, the structure~$\gamma$ is an $\bc H$-assembly of $\bc H^{i-1}(\bc Z,\bc A)$-structures. 
At least one of these structures, say $\beta$,  belongs to the species~$\bc H^{i-1}(\bc Z,\bc A)-\bc H^{i-1}(\bc Z,\bc B)$, otherwise $\gamma$ would be an~$\bc H$-assembly of~$\bc H^{i-1}(\bc Z,\bc B)$-structures, {\em i.e.}, an~$\bc H^{i}(\bc Z,\bc B)$-structure. 
Since contact of order $k$ is maintained by composition, the hypotheses imply that  $\bc H^{i-1}(\bc Z,\bc A)=_k\bc H^{i-1}(\bc Z,\bc B)$; thus~$\beta$ is of size~$>k$, that is exactly~$k+1$, since $\gamma$ is of size at most $k+1$. Moreover, all the other structures composing the structure~$\gamma$ are of size~$0$. But, given that~$\bc H(\bs{0},\bs{0})=\bs{0}$, there is no $\bc H$-rooted tree of size~$0$ and thus the structure~$\gamma$ is an $\bc H$-assembly whose unique member is~$\beta$. Therefore, $\gamma$ belongs to the species~$\jaco (\bc H^{i-1}(\bc Z,\bc A)-\bc H^{i-1}(\bc Z,\bc B))$.

Iterating this, we deduce that a structure of size at most equal to~$k+1$, and belonging to the species $\bc H^p(\bc Z,\bc A)-\bc H^p(\bc Z,\bc B)$, rewrites as a structure of~$(\jaco)^p(\bc A-\bc B)$, which is $\bs{0}$. In other words:
\[
\bc H^p(\bc Z,\bc A)-\bc H^p(\bc Z,\bc B)=_{k+1} \bs{0}.\qedhere
\]      
\end{proof}

Lemma~\ref{lem:increasing-contact}, with $\bc A=\bc Y^{[n+1]}$ and $\bc B=\bc Y^{[n]}$, shows that~$\bc Y^{[n+p]}=_{k+1}\bc Y^{[n+p+1]}$ where~$1\le p\le m$.
The proof of Proposition~\ref{prop:is-cvg} is now concluded by invoking Lemma~\ref{lemma1}.
\end{proof}

\begin{example}\label{ex:catalan}
The species of Catalan trees is defined by the implicit equation ${\mathcal T} ={\mathcal Z} \cdot \Seq({\mathcal T})$, ${\mathcal T}(0)=0$.  
Figure~\ref{fig:trees-slow} shows the structures (omitting the elements of the underlying sets) produced in the first five iterations of ${\mathcal T}^{[n+1]}={\mathcal Z}\cdot \Seq({\color{blue}{\mathcal T}^{[n]}})$.
Rectangles enclose structures of a given size when all structures  of the limit~${\mathcal T}$ with this size have been produced: for example iteration ${\mathcal T}^{[4]}$ contains all structures of the solution up to size~4; and iteration ${\mathcal T}^{[5]}$ contains all structures of the solution up to size~5, i.e.,  ${\mathcal T}^{[5]}=_5{\mathcal T}$      
\end{example} 
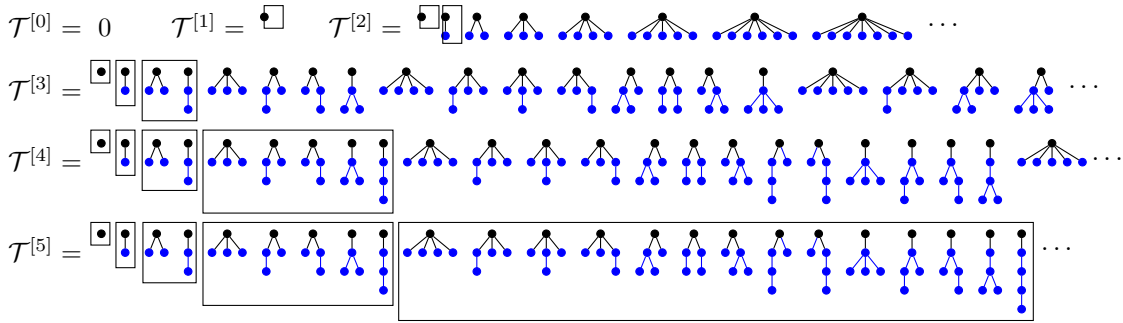
\begin{figure}[ht] 
\noindent \begin{tikzpicture} 	
	\tikzstyle{level 1}=[sibling distance=2mm]            
	\node[rectangle,minimum width=14cm,minimum height=2mm,anchor=west]{};
	\matrix[mat, inner sep=2mm, row sep=0mm, column sep=2mm]{  
	\node[rectangle,anchor=west]{${\mathcal T}^{[0]}=~0~~~~~$}; \& 
	\node[rectangle]{${\mathcal T}^{[1]}=$}; \&
	\node[yshift=1mm][rond]{};\&  
	\node[rectangle]{$~~~~~{\mathcal T}^{[2]}=$}; \&
	\node[yshift=1mm][rond]{}; \& 
	\node[yshift=1mm][rond] {} child{node[rond][color=blue] {}} ;\&
	\node[yshift=1mm][rond] {} child{node[rond][color=blue] {}} 
	                        child{node[rond][color=blue] {}};\& 
	\node[yshift=1mm][rond] {} child{node[rond][color=blue] {}}
			   	child{node[rond][color=blue] {}}
			   	child{node[rond][color=blue] {}} ;\&
	\node[yshift=1mm][rond] {} child{node[rond][color=blue] {}} 
			   	child{node[rond][color=blue] {}}
				child{node[rond][color=blue] {}}
			   	child{node[rond][color=blue] {}} ;\&
	\node[yshift=1mm][rond] {} child{node[rond][color=blue] {}} 
			   	child{node[rond][color=blue] {}}
			   	child{node[rond][color=blue] {}}
				child{node[rond][color=blue] {}}
				child{node[rond][color=blue] {}} ;\&
	\node[yshift=1mm][rond] {} child{node[rond][color=blue] {}} 
			   	child{node[rond][color=blue] {}}
				child{node[rond][color=blue] {}}
				child{node[rond][color=blue] {}}
				child{node[rond][color=blue] {}}
			   	child{node[rond][color=blue] {}} ;\&  
	\node[yshift=1mm][rond] {} child{node[rond][color=blue] {}} 
			   	child{node[rond][color=blue] {}}
				child{node[rond][color=blue] {}}
				child{node[rond][color=blue] {}}
				child{node[rond][color=blue] {}}
				child{node[rond][color=blue] {}}
			   	child{node[rond][color=blue] {}} ;\&
        \node{$\dots$};\& \hfill 	
	\\ };     
\node[rectangle, draw, draw opacity=1,fill opacity=0, minimum width=0.25cm, minimum height=0.5cm] at (6.05,0){};  
\node[rectangle, draw, draw opacity=1,fill opacity=0, minimum width=0.25cm, minimum height=0.3cm] at (5.75,0.1){};
\node[rectangle, draw, draw opacity=1,fill opacity=0, minimum width=0.25cm, minimum height=0.3cm] at (3.7,0.1){}; 
\end{tikzpicture} 

\vskip-1mm
\noindent \begin{tikzpicture}          
\tikzstyle{level 1}=[sibling distance=2mm]
\tikzstyle{level 2}=[sibling distance=2mm]
	\tikzstyle{every node}=[circle,inner sep=0,minimum size=0mm]   
	\node[rectangle,minimum width=14cm,minimum height=2mm,anchor=west]{};
	\matrix[mat, inner sep=2mm, row sep=0mm, column sep=2mm]{
	\node[rectangle]{${\mathcal T}^{[3]}=$}; \&  
	\node[yshift=2mm][rond]{}; \&  
	\node[yshift=2mm][rond] {} child{node[rond][color=blue] {}} ;\&
	\node[yshift=2mm][rond] {} child{node[color=blue][rond]{}} 
	                        child{node[color=blue][rond]{}} ;\&  
	\node[yshift=2mm][rond] {} child{node[rond][color=blue] {} 
	                        child[color=blue]{\nrp}} ;\&
	\node[yshift=2mm][rond] {} child{node[rond][color=blue] {}} 
			   	child{node[rond][color=blue] {}}
			   	child{node[rond][color=blue] {}} ; \&
        \node[yshift=2mm][rond] {} child{ node[rond][color=blue]{}  
					child[color=blue]{\nrp}} 
					child{node[color=blue][rond]{}};\&
        \node[yshift=2mm][rond] {} child{ node[rond][color=blue]{}} 
					child{node[color=blue][rond]{}
   					child[color=blue]{\nrp}};\& 
	\node[yshift=2mm][rond] {} child{node[rond][color=blue] {} 
					child[color=blue]{\nrp}
					child[color=blue]{\nrp}} ; \&
        \node[yshift=2mm][rond] {} child{node[rond][color=blue] {}} 
				   	child{node[rond][color=blue] {}}
					child{node[rond][color=blue] {}}
				   	child{node[rond][color=blue] {}} ; \& 
        \node[yshift=2mm][rond] {} child{node[rond][color=blue] {} 
			   		child[color=blue]{\nrp}}
					child{node[rond][color=blue] {}}
			   		child{node[rond][color=blue] {}} ; \& 
        \node[yshift=2mm][rond] {} child{node[rond][color=blue] {}}
	     				child{node[rond][color=blue] {}
						  child[color=blue]{\nrp}}
			      		child{node[rond][color=blue] {}} ; \&
	\node[yshift=2mm][rond] {} child{node[rond][color=blue] {}}
					child{node[rond][color=blue] {}}
					child{node[rond][color=blue] {}
						child[color=blue]{\nrp}} ; \&
        \node[yshift=2mm][rond] {} child{node[color=blue][rond] {} 
					child[color=blue]{\nrp} 
					child[color=blue]{\nrp}} 
				        child{ node[rond][color=blue]{}};\& 
   	\node[yshift=2mm][rond] {} child{node[color=blue][rond] {}
				        child[color=blue]{\nrp}}
				        child{ node[rond][color=blue]{}  	
					child[color=blue]{\nrp}};\&
        \node[yshift=2mm][rond] {} child{node[color=blue][rond] {}} 
				        child{ node[rond][color=blue]{} 
				 		child[color=blue]{\nrp}      	
				  		child[color=blue]{\nrp}};\&    
  	\node[yshift=2mm][rond] {} child{node[rond][color=blue] {}    
					child[color=blue]{\nrp}  
					child[color=blue]{\nrp} 	
					child[color=blue]{\nrp}} ; \&  
        \node[yshift=2mm][rond] {} child{node[rond][color=blue] {}} 
			   		child{node[rond][color=blue] {}}
					child{node[rond][color=blue] {}} 
					child{node[rond][color=blue] {}}
			   		child{node[rond][color=blue] {}} ; \& 
	\node[yshift=2mm][rond] {} child{node[rond][color=blue] {} 
					child[color=blue]{\nrp}}    
		   			child{node[rond][color=blue] {}}	
		   			child{node[rond][color=blue] {}}
		   	   		child{node[rond][color=blue] {}} ;\& 							
        \node[yshift=2mm][rond] {} child{node[rond][color=blue] {} 
   	   				child[color=blue]{\nrp} 
   	   				child[color=blue]{\nrp}}  
    	   		                child{node[rond][color=blue] {}} 
	   			        child{node[rond][color=blue] {}} ;\&  							
        \node[yshift=2mm][rond] {} child{node[rond][color=blue] {} 
   	   				child[color=blue]{\nrp}    
   	   				child[color=blue]{\nrp}
   	   				child[color=blue]{\nrp}}  
    	   		                child{node[rond][color=blue] {}} ;\&
        \node{$\dots$};
	\\ };     
\node[rectangle,draw, draw opacity=1,fill opacity=0, minimum width=0.25cm, minimum height=0.3cm] at (1.42,0.25){};   
\node[rectangle,draw, draw opacity=1,fill opacity=0, minimum width=0.25cm, minimum height=0.6cm] at (1.75,0.1){};
\node[rectangle,draw, draw opacity=1,fill opacity=0, minimum width=0.72cm, minimum height=0.8cm] at (2.33,0){};
\end{tikzpicture}  

\vskip-1mm      
\noindent \begin{tikzpicture}  
	\tikzstyle{level 1}=[sibling distance=2mm]
	\tikzstyle{level 2}=[sibling distance=2mm] 
	\tikzstyle{level 3}=[sibling distance=2mm]
	\tikzstyle{level 4}=[sibling distance=2mm]             
	\node[rectangle,minimum width=14cm,minimum height=2mm,anchor=west]{};
	\matrix[mat,inner sep=2mm,row sep=3mm, column sep=2mm]{ 
	\node[rectangle]{${\mathcal T}^{[4]}=$}; \& 
	\node[yshift=2mm][rond]{}; \&  
	\node[yshift=2mm][rond] {} child{node[rond][color=blue] {}} ;\&
	\node[yshift=2mm][rond] {} child{node[color=blue][rond]{}} 
	                        child{node[color=blue][rond]{}} ;\&  
	\node[yshift=2mm][rond] {} child{node[rond][color=blue] {} 
	                        child[color=blue]{\nrp}} ;\&
	\node[yshift=2mm][rond] {} child{node[rond][color=blue] {}} 
			   	child{node[rond][color=blue] {}}
			   	child{node[rond][color=blue] {}} ; \&
        \node[yshift=2mm][rond] {} child{ node[rond][color=blue]{}  
				child[color=blue]{\nrp}} 
				child{node[color=blue][rond]{}};\&
        \node[yshift=2mm][rond] {} child{ node[rond][color=blue]{}} 
				child{node[color=blue][rond]{}
				child[color=blue]{\nrp}};\& 
        \node[yshift=2mm][rond] {} child{node[rond][color=blue] {} 
		   		child[color=blue]{\nrp}
		   		child[color=blue]{\nrp}} ; \&  
        \node[yshift=2mm][rond] {} child{node[rond][color=blue] {} 
				child[color=blue]{\nrp 
				child[color=blue]{\nrp}}} ;\&
        \node[yshift=2mm][rond] {} child{node[rond][color=blue] {}} 
			   	child{node[rond][color=blue] {}}
				child{node[rond][color=blue] {}}
			   	child{node[rond][color=blue] {}} ; \& 
        \node[yshift=2mm][rond] {} child{node[rond][color=blue] {} 
		   			child[color=blue]{\nrp}}
				child{node[rond][color=blue] {}}
		   		child{node[rond][color=blue] {}} ; \& 
        \node[yshift=2mm][rond] {} child{node[rond][color=blue] {}}
     				child{node[rond][color=blue] {}
					  	child[color=blue]{\nrp}}
		      		child{node[rond][color=blue] {}} ; \&
        \node[yshift=2mm][rond] {} child{node[rond][color=blue] {}}
			   	child{node[rond][color=blue] {}}
			        child{node[rond][color=blue] {}
				   	child[color=blue]{\nrp}} ; \&

        \node[yshift=2mm][rond] {} child{node[color=blue][rond] {}  
					child[color=blue]{\nrp} 
					child[color=blue]{\nrp}} 
				child{ node[rond][color=blue]{}};\& 
        \node[yshift=2mm][rond] {} child{node[color=blue][rond] {}
					child[color=blue]{\nrp}}
				child{ node[rond][color=blue]{}  	
					child[color=blue]{\nrp}};\&
        \node[yshift=2mm][rond] {} child{node[color=blue][rond] {}} 
			        child{ node[rond][color=blue]{} 
			 		child[color=blue]{\nrp}      	
			  		child[color=blue]{\nrp}};\&  
        \node[yshift=2mm][rond] {} child{node[rond][color=blue] {} 
					child[color=blue]{\nrp 
						child[color=blue]{\nrp}}}
				child[color=blue]{\nrp} ;		\& 
        \node[yshift=2mm][rond] {} child[color=blue]{\nrp}
				child{node[rond][color=blue] {} 
					child[color=blue]{\nrp 
						child[color=blue]{\nrp}}} ;\&
        \node[yshift=2mm][rond] {} child{node[rond][color=blue] {}    
				child[color=blue]{\nrp}  
				child[color=blue]{\nrp} 	
				child[color=blue]{\nrp}} ; \&  
        \node[yshift=2mm][rond] {} child{node[rond][color=blue] {} 
				child[color=blue]{\nrp 
					child[color=blue]{\nrp}}
					child[color=blue]{\nrp}} ;\&  
        \node[yshift=2mm][rond] {} child{node[rond][color=blue] {}		
        			child[color=blue]{\nrp} 
        			child[color=blue]{\nrp 
        				child[color=blue]{\nrp}}} ;\&   
        \node[yshift=2mm][rond] {} child{node[rond][color=blue] {} 
				child[color=blue]{\nrp
					child[color=blue]{\nrp} 
					child[color=blue]{\nrp}}} ; \& 
        \node[yshift=2mm][rond] {} child{node[rond][color=blue] {}} 
		   		child{node[rond][color=blue] {}}
				child{node[rond][color=blue] {}} 
				child{node[rond][color=blue] {}}
		   		child{node[rond][color=blue] {}} ; \&
        \node (n) {}; 
        	\\	 }; 
\node[rectangle,draw, draw opacity=1,fill opacity=0, minimum width=0.25cm, minimum height=0.3cm] at (1.42,0.4){};   
\node[rectangle,draw, draw opacity=1,fill opacity=0, minimum width=0.25cm, minimum height=0.6cm] at (1.75,0.25){};
\node[rectangle,draw, draw opacity=1,fill opacity=0, minimum width=0.72cm, minimum height=0.8cm] at (2.33,0.15){};
\node[rectangle,draw, draw opacity=1,fill opacity=0, minimum width=2.5cm, minimum height=1.1cm] at (4.02,0){};    
\node[xshift=1mm] at (n) {$\dots$};
\end{tikzpicture}    

\vskip-1mm      
\noindent \begin{tikzpicture}  
	\tikzstyle{level 1}=[sibling distance=2mm]
	\tikzstyle{level 2}=[sibling distance=2mm] 
	\tikzstyle{level 3}=[sibling distance=2mm]
	\tikzstyle{level 4}=[sibling distance=2mm]    
	\node[rectangle,minimum width=14cm,minimum height=2mm,anchor=west]{};
	\matrix[mat,inner sep=2mm,row sep=3mm, column sep=2mm]{ 
	\node[rectangle]{${\mathcal T}^{[5]}=$}; \& 
	\node[yshift=2mm][rond]{}; \&  
	\node[yshift=2mm][rond] {} child{node[rond][color=blue] {}} ;\&
	\node[yshift=2mm][rond] {} child{node[color=blue][rond]{}} 
	                        child{node[color=blue][rond]{}} ;\&  
	\node[yshift=2mm][rond] {} child{node[rond][color=blue] {} 
	                        child[color=blue]{\nrp}} ;\&
	\node[yshift=2mm][rond] {} child{node[rond][color=blue] {}} 
			   	child{node[rond][color=blue] {}}
			   	child{node[rond][color=blue] {}} ; \&
        \node[yshift=2mm][rond] {} child{ node[rond][color=blue]{}  
                			child[color=blue]{\nrp}} 
                		child{node[color=blue][rond]{}};\&
        \node[yshift=2mm][rond] {} child{ node[rond][color=blue]{}} 
				child{node[color=blue][rond]{}
					child[color=blue]{\nrp}};\& 
        \node[yshift=2mm][rond] {} child{node[rond][color=blue] {} 
		   		child[color=blue]{\nrp}
		   		child[color=blue]{\nrp}} ; \&  
        \node[yshift=2mm][rond] {} child{node[rond][color=blue] {} 
				child[color=blue]{\nrp 
					child[color=blue]{\nrp}}} ; \& 
        \node[yshift=2mm][rond] {} child{node[rond][color=blue] {}} 
			   	child{node[rond][color=blue] {}}
				child{node[rond][color=blue] {}}
			   	child{node[rond][color=blue] {}} ; \& 
        \node[yshift=2mm][rond] {} child{node[rond][color=blue] {} 
		   			child[color=blue]{\nrp}}
				child{node[rond][color=blue] {}}
		   		child{node[rond][color=blue] {}} ; \& 
        \node[yshift=2mm][rond] {} child{node[rond][color=blue] {}}
     				child{node[rond][color=blue] {}
					  	child[color=blue]{\nrp}}
		      		child{node[rond][color=blue] {}} ; \&
        \node[yshift=2mm][rond] {} child{node[rond][color=blue] {}}
			   	child{node[rond][color=blue] {}}
			        child{node[rond][color=blue] {}
				   	child[color=blue]{\nrp}} ; \&
        \node[yshift=2mm][rond] {} child{node[color=blue][rond] {}  
					child[color=blue]{\nrp} 
					child[color=blue]{\nrp}} 
				child{ node[rond][color=blue]{}};\& 
        \node[yshift=2mm][rond] {} child{node[color=blue][rond] {}
						child[color=blue]{\nrp}}
				child{ node[rond][color=blue]{}  	
					child[color=blue]{\nrp}};\&
        \node[yshift=2mm][rond] {} child{node[color=blue][rond] {}} 
			        child{ node[rond][color=blue]{} 
			 		child[color=blue]{\nrp}      	
			  		child[color=blue]{\nrp}};\&  
        \node[yshift=2mm][rond] {} child{node[rond][color=blue] {} 
					child[color=blue]{\nrp 
						child[color=blue]{\nrp}}}
				child[color=blue]{\nrp} ;\&
        \node[yshift=2mm][rond] {} child[color=blue]{\nrp}
				child{node[rond][color=blue] {} 
					child[color=blue]{\nrp 
						child[color=blue]{\nrp}}} ;\&
        \node[yshift=2mm][rond] {} child{node[rond][color=blue] {}   
				child[color=blue]{\nrp}  
				child[color=blue]{\nrp} 	
				child[color=blue]{\nrp}} ; \&  
        \node[yshift=2mm][rond] {} child{node[rond][color=blue] {} 
				child[color=blue]{\nrp 
					child[color=blue]{\nrp}}
					child[color=blue]{\nrp}} ;\&  
        \node[yshift=2mm][rond] {} child{node[rond][color=blue] {}			
				child[color=blue]{\nrp} 
				child[color=blue]{\nrp 
					child[color=blue]{\nrp}}} ;\&   
        \node[yshift=2mm][rond] {} child{node[rond][color=blue] {} 
				child[color=blue]{\nrp
					child[color=blue]{\nrp} 
					child[color=blue]{\nrp}}} ; \&  
        \node[yshift=2mm][rond] {} child{node[rond][color=blue] {} 
				child[color=blue]{\nrp 
				    child[color=blue]{\nrp
					    child[color=blue]{\nrp}}}} ; \&
        \node[inner sep=0mm]{$\dots$}; 	 
\\	 };  
\node[rectangle,draw, draw opacity=1,fill opacity=0, minimum width=0.25cm, minimum height=0.3cm] at (1.42,0.5){};   
\node[rectangle,draw, draw opacity=1,fill opacity=0, minimum width=0.25cm, minimum height=0.6cm] at (1.75,0.35){};
\node[rectangle,draw, draw opacity=1,fill opacity=0, minimum width=0.7cm, minimum height=0.8cm] at (2.33,0.25){};  
\node[rectangle,draw, draw opacity=1,fill opacity=0, minimum width=2.5cm, minimum height=1.1cm] at (4.02,0.1){};  
\node[rectangle,draw, draw opacity=1,fill opacity=0, minimum width=8.35cm, minimum height=1.3cm] at (9.52,0){};
\end{tikzpicture}
\caption{ First iterations of ${\mathcal T}^{[n+1]}={\mathcal Z}\cdot \Seq({\color{blue}{\mathcal T}^{[n]}})$.\label{fig:trees-slow}}	                                 
\end{figure}

\section{Well-founded Systems at 0 }\label{subseq:wf0}
In this section, we are interested in systems of the form~$\bc Y=\bc H(\bc Z,\bc Y)$ such that $\bc H(\bs{0},\bs{0})=\bs{0}$.
We focus on the case when the convergent iteration defined by Equation~\eqref{eq:ite-def-wf0} has a solution with no zero (empty species) coordinates. 
This type of combinatorial system, which we call well-founded at~$\bs{0}$, is not only natural, but  also easy to characterize.
Section~\ref{subsubsec:empty_species} gathers useful properties of empty species and how to detect solutions of systems with empty coordinates;  Section~\ref{subsub:wellfoundedat0} defines and characterizes combinatorial systems that are
well-founded at~$\bs{0}$.

\subsection{Empty Species}\label{subsubsec:empty_species}
The empty species plays the role of a zero in the theory of species. We first state an obvious property.
\begin{lemma} For any species $\bc F$, the species~$\bs{0}$ and $\bc F$ are disjoint, 
        $\bc F\cdot \bs{0} = \bs{0} \cdot  \bc F =\bs{0}$, and $\bs{0}\subset \bc F$.
\end{lemma}
\begin{proof}
These are direct consequences of the definitions.
\end{proof}
\noindent The next property is more combinatorial, in the sense that it relies on positivity.

\begin{lemma}\label{lemma:zero}
         Let  $\bc G=({\mathcal G}_{1:m})$ be a vector of (possibly multisort) species, such that ${\mathcal G}_i\neq 0$ and ${\mathcal G}_i(0)=0$, for $i=1,\dots,m$. 
         For any vector of species~$\bc F$, if $\bc F({\mathcal G}_{1:m})=\bs{0}$, then $\bc F=\bs{0}$.   
\end{lemma}
\begin{proof}
For simplicity of notation and without loss of generality, we consider the case when $\bc F$ is a single two-sort species and $\bc G$ is unisort.

We assume that ${\mathcal F}\neq 0$ and show how to build a nonempty ${\mathcal F}({\mathcal G}_1,{\mathcal G}_2)$-structure. The hypotheses imply that there exists a multi-set~$\bs V=(V_1,V_2)$ such that~${\mathcal F}[\bs V]\neq \eset$ and two sets~$U_1$ and $U_2$ such that~${\mathcal G}_i[U_i]\neq \eset$, for $i=1,2$. Let $\bs U=(\{1\}\times U_1)+\dots+(\{v_1\}\times U_1)+(\{1\}\times U_2)+\dots+(\{v_2\}\times U_2)$ where $v_1$ and $v_2$ are the cardinalities of~$V_1$ and~$V_2$. By construction, there exists a natural bijection between $\bs V$ and $\bs U$, so that ${\mathcal F}[\bs U]\neq \eset$. Similarly, each~$\{i\}\times U_1$ is in bijection with~$U_1$ so that~${\mathcal G}_1[\{i\}\times U_1]\neq\eset$ and the same holds for~$U_2$. Thus, there exists a nonempty structure in the set
\[
{\mathcal F}[U_1,U_2]\times\prod_{i=1}^{v_1} {\mathcal G}_1[\{i\}\times U_1]\times\prod_{j=1}^{v_2} {\mathcal G}_2[\{j\}\times U_2],
\]
which shows that the species~${\mathcal F}({\mathcal G}_1,{\mathcal G}_2)$ is not~0, in view of Equation~\eqref{eq:2sort-comp}.
\end{proof}
\begin{example} The product species is not zero, thus if~${\mathcal A}\cdot{\mathcal B}=0$, then one of~${\mathcal A}$ or~${\mathcal B}$ is~0.
\end{example}

\noindent As a corollary, the emergence of nonempty species in composition is restricted to the case when a component species turns from  empty to nonempty.
\begin{corollary}\label{coro:zero}
        Let  $\bc A$ and $\bc B$  be vectors of  (possibly multisort) species;  assume that $\bc B\subset\bc A$ and  ${\bc B(\bs{0})=\bc A(\bs{0})=\bs{0}}$. 
For any vector of $m$~species $\bc F$, if $\bc F(\bc B)=\bs{0}$ and  $\bc F(\bc A)\ne\bs{0}$, then
there exists $i \in \{1,\dots,m\}$, such that ${\mathcal B}_i=0$ and ${\mathcal A}_i\neq0$.     
\end{corollary}
\begin{proof}
        Assume that the conclusion does not hold, that is, for all $i \in \{1,\dots,m\}$, ${\mathcal B}_i=0$ implies that ${\mathcal A}_i=0$.
 Assume without loss of generality that the nonzero coordinates of $\bc B$ are the first $k$ ones, while ${\mathcal B}_{k+1}=\dots={\mathcal B}_m={\mathcal A}_{k+1}=\dots={\mathcal A}_m={0}$. 
The species~$\bc G({\mathcal Y}_{1:k}):=\bc F({\mathcal Y}_{1:k},\bs{0})$ is such that  $\bc G({\mathcal B}_{1:k})=\bc F(\bc B)=\bs{0}$.
By Lemma~\ref{lemma:zero}, $\bc G=\bs{0}$ and thus  $\bs{0}=\bc G({\mathcal A}_{1:k})=\bc F(\bc A)$, a contradiction.
\end{proof}
Finally, we get an effective criterion for detecting the existence of zero coordinates in the solution of a combinatorial system.
\begin{lemma}\label{lem:0coord}
Let $\bc H=({\mathcal H}_{1:m})$ be a vector of species such that $\bc H(\bs{0},\bs{0})=\bs{0}$ and the Jacobian matrix~${\bd\partial \bc H}/{\bd\partial\bc Y}(\bs{0},\bs{0})$ is nilpotent.
The $i$th coordinate of the solution~$\bc S$ of the system~$\bc Y=\bc H(\bc Z,\bc Y)$ is $0$ if and only if ${\mathcal Y}_i^{[m]}=0$, where~${\mathcal Y}_i^{[m]}$ is the $i$th coordinate of the $m$th element of the sequence defined by~\eqref{joyal-iteration}.
\end{lemma}  
\begin{proof}
If ${\mathcal S}_i=0$, then ${\mathcal Y}_i^{[k]}=0$, for all $k\geq 0$.        
Conversely, let ${\mathcal Y}_i^{[m]}=0$ and assume that ${\mathcal S}_i\neq 0$.
Then, there exists $k>m$ such that ${\mathcal Y}_i^{[k]}={\mathcal H}_i(\bc Z,\bc Y^{[k-1]})\neq 0$ and ${\mathcal Y}_i^{[k-1]}={\mathcal H}_i(\bc Z,\bc Y^{[k-2]})= 0$. 
By Corollary~\ref{coro:zero}, this implies that there exists $j\neq i$ such that  ${\mathcal Y}_j^{[k-1]}\neq 0$ and ${\mathcal Y}_j^{[k-2]}= 0$. 
This reasoning cannot be iterated more than $m$ times, which implies a contradiction, that is $k\leq m$.        
\end{proof}

\begin{Algorithm}[t]{0.79\textwidth}   
\SetAlgoRefName{0-coord}
\caption{Detection of zero coordinates in the solution of a system\label{algo:0-coord}}
\DontPrintSemicolon
\Input{$\bc H=({\mathcal H}_{1:m})$: a vector of species such that $\bc H(\bs{0},\bs{0})=\bs{0}$ and the Jacobian matrix~${\bd\partial \bc H}/{\bd\partial\bc Y}(\bs{0},\bs{0})$ is nilpotent.}
\Output{Answer to ``Are there 0 coordinates in the solution of the system $\bc Y=\bc H(\bc Z,\bc Y)$?''}
\SetAlgoNoLine
\Begin{
\SetAlgoVlined
Compute $\,\bc U:= \bc H^m(\bc Z,\bs{0})$\;
\ForEach{coordinate ${\mathcal C}$ \KwOf $\,\bc U$}
        {\lIf{${\mathcal C} = 0$}{\Return YES} }
\Return NO
}
\end{Algorithm}

Lemma~\ref{lem:0coord} leads to Algorithm~\ref{algo:0-coord}. Note that in practice, it is not necessary to compute the whole species~$\bc H^m(\bc Z,\bs{0})$ (which may become very large). Indeed, according to the proof, the only property we use, for each coordinate of the vector, is whether the species is 0 or not. Thus, practically, we compute $\bc F^m(\bc Z,\bs{0})$ instead of $\bc H^m(\bc Z,\bs{0})$, with ${\mathcal F}_i(\bc Z,\bc Y)$ being 0 if ${\mathcal H}_i(\bc Z,\bc Y)=0$ and ${\mathcal Z}$ otherwise. This is essentially the same method as in Algorithm A of~\cite[p. 28]{Zimmermann91}, the zero coordinates being those with an infinite valuation.

\subsection{Well-foundedness at 0}\label{subsub:wellfoundedat0}
In Joyal's implicit species theorem, the nilpotence of the Jacobian matrix appears as a sufficient condition. In this section, we prove a converse of Joyal's implicit species theorem under the extra condition that the solution does not have any empty coordinate.

\begin{definition}\label{def:wf0}
	Let $\bc H(\bc Z,\bc Y)$ be a vector of species such that $\bc H(\bs{0},\bs{0})=\bs{0}$.
	The combinatorial system $\bc Y=\bc H(\bc Z,\bc Y)$ is said to be  \emph{well-founded at}~$\bs{0}$ when the sequence~$(\bc Y^{[n]})_{n\in\N}$  defined by
        \begin{equation}\label{eq:ite-def-wf0}\tag{\ref{phi-iteration}}
                \bc Y^{[0]}=\bs{0}\quad\text{and}\quad \bc Y^{[n+1]}=\bc H(\bc Z,\bc Y^{[n]}),\quad n\ge0
        \end{equation}
        is convergent and the limit~$\bc S$ of this sequence has no zero coordinate.
\end{definition} 
 
Requiring the solution of a recursively defined combinatorial system to have no zero coordinate is a natural combinatorial condition  from the point of view of specification designers. In any case, Lemma~\ref{lem:0coord} shows that it is easy to detect. It is also easy to fix, by removing from the system the corresponding unknowns.

\begin{example} Here are a few examples of systems that are excluded by our definition although the iteration is convergent:
        \begin{itemize}
                \item[--] ${\mathcal Y}={\mathcal Y}$. In this case, the Jacobian matrix at 0 is not nilpotent (and the equation has an infinite number of solutions);
                \item[--] ${\mathcal Y}={\mathcal Y}+{\mathcal Z}{\mathcal Y}$. Again, the Jacobian matrix at 0 is not nilpotent (still 0 is its unique solution);
                \item[--] ${\mathcal Y}={\mathcal Z}{\mathcal Y}$. Here, the Jacobian matrix at~0 is nilpotent (it is~0), but this equation is not well-founded at~0 with our definition. 
\end{itemize}
\end{example}

We now state a nice and effective characterization of systems well-founded at~$\bs{0}$ (in  our previous article~\cite{PiSaSo08} the characterization was wrong, omitting the pathological cases of solutions with zero coordinates). The associated effective procedure is Algorithm~\ref{algo:wf0}.

\begin{Algorithm}{0.81\textwidth}   
\SetKwFunction{zerocoord}{\ref{algo:0-coord}}
\SetAlgoRefName{isWellFoundedAt0}
\caption{Characterization of well-founded systems at $\bs{0}$\label{algo:wf0}}
\DontPrintSemicolon
\Input{$\bc H=({\mathcal H}_{1:m})$: a vector of species such that $\bc H(\bs{0},\bs{0})=\bs{0}$.}
\Output{Answer to ``Is the system $\bc Y=\bc H(\bc Z,\bc Y)$ well-founded at $\bs{0}$?''}
\SetAlgoNoLine
\Begin{
\SetAlgoVlined
        Compute $\bc J:= {\bd\partial \bc H}/{\bd\partial\bc Y}(\bs{0},\bs{0})$\;
        \lIf{$\bc J^m = 0$}{\Return \zerocoord{$\bc H$}}
	\lElse{\Return NO}
}
\end{Algorithm}

\begin{theorem}[Characterization of well-founded systems at~$\bs{0}$]\label{th:carac-wf0}
 Let $\bc H=({\mathcal H}_{1:m})$ be a vector of species such that $\bc H(\bs{0},\bs{0})=\bs{0}$. The combinatorial system $\bc Y=\bc H(\bc Z,\bc Y)$ is well-founded at~$\bs{0}$ \emph{if and only if} the Jacobian matrix ${\bd\partial \bc H}/{\bd\partial\bc Y}(\bs{0},\bs{0})$ is nilpotent and the vector of species~$\bc Y^{[m]}$ defined by Equation~\eqref{eq:ite-def-wf0} has no zero coordinate.
\end{theorem}
\begin{proof}
One direction was proved along with the implicit species theorem and is a consequence of  Proposition~\ref{prop:is-cvg} and Lemma~\ref{lem:0coord}.

Conversely, if the system is well-founded at~$\bs{0}$, then Lemma~\ref{lem:0coord} gives the condition on~$\bc Y^{[m]}$. We now show the nilpotence of the Jacobian matrix by contradiction.
Let $\bs{\gamma}$ be an $\bc S$-structure such that $\gamma_i\neq 0$ for $i=1,\dots,m$ and let $n$ be the size of $\bs{\gamma}$. 
        Assume that the matrix  ${\bd\partial \bc H}/{\bd\partial\bc Y}(\bs{0},\bs{0})$ is not nilpotent.
        Thus, for all $q\in \mathbb{N}$, there exists a nonzero structure $\bs{\beta}_q$ in the species~$({\bd\partial \bc H}/{\bd\partial\bc Y}(\bs{0},\bs{0}))^q$. By construction, the size of $\bs{\beta}_q$ is $0$. Since none of the $\gamma_i$ is zero, there are infinitely many $\bc S$-structures of the form $\bs{\beta}_q\cdot \bs{\gamma}$, all of size $n$, which prevents the sequence~$(\bc Y^{[n]})_{n\in\N}$ from converging and leads to a contradiction.
\end{proof}

\section{Polynomial Species}\label{subseq:polspecies}
We now extend the implicit species theorem to cases with structures of size~0. Control over the number
of such structures is provided by polynomial species that we first present. This section and the next one can be skipped on first reading.

\subsection{Polynomial Species}\label{subsubsec:finite_species}

\begin{definition}
        A (possibly multisort) species of structures ${\mathcal F}$ is \emph{polynomial} if there exists $n\geq 0$ such that  ${\mathcal F}_{\geq n}=0$. In other words, there are only a finite number of ${\mathcal F}$-structures.
\end{definition}

The following results provide two effective characterizations of those systems that admit a polynomial solution; the first one applies to a system whose solution may have zero coordinates and the second one gives another characterization when the system is well-founded at~$\bs{0}$. While the second characterization seems to be clearer, the first one is needed in the next section, when $\bc H(\bs{0},\bs{0})\neq 0$. The corresponding decision procedure is Algorithm~\ref{algo:poly-Joyal}.

\begin{proposition}[Implicit Polynomial Species]\label{prop:polynomial_implicit}          Let~$\bc H=({\mathcal H}_{1:m})$ be a vector of species such that $\bc H(\bs{0},\bs{0})=\bs{0}$ and the Jacobian matrix ${\bd\partial \bc H}/{\bd\partial\bc Y}(\bs{0},\bs{0})$ is nilpotent.
          Let $(\bc Y^{[n]})_{n\in\N}$ be the sequence of species defined by 
\begin{equation}\label{eq:phi_poly}\tag{\ref{joyal-iteration}}
                  \bc Y^{[0]}=\bd{0}\quad\text{and}\quad \bc Y^{[n+1]}=\bc H(\bc Z,\bc Y^{[n]}),\quad n\ge0
\end{equation}
          The solution~$\bc S$ of the system~$\bc Y=\bc H(\bc Z,\bc Y)$ such that $\bc S(\bs{0})=\bs{0}$ is polynomial if and only if $\bc Y^{[m]}$ is polynomial and  $\bc Y^{[m]}=\bc Y^{[m+1]}$. 
 \end{proposition}

Note that under these conditions, the solution is given by~$\bc S=\bc Y^{[m]}$, or even more precisely, $\bc S=\bc Y^{[p]}$ where $p$ is the order of nilpotence of the Jacobian matrix.

\begin{proof}
        
        First, if $\bc Y^{[m]}$ is polynomial and $\bc Y^{[m]}=\bc Y^{[m+1]}$ then $\bc Y^{[m+1]}$ is also polynomial. That the limit is polynomial follows by induction.

        Conversely, we show that if the solution is polynomial, then the system has a particular form: it is triangular, and the right-hand side of each equation does not depend on the variable it defines. This implies that the solution is reached in at most~$m$ iterations.

If $\bc S$ has coordinates that are zero, they can be removed from the system without affecting the other coordinates. Thus, from now on, we consider that $\bc S$ has no zero coordinate. 
Let $k$ be the largest integer such that $\bc Y^{[k]}\neq\bc Y^{[k-1]}$. In particular, this means that $\bc Y^{[k]}=\bc S$.
By definition of $k$, the vector of species $\bc Y^{[k]}-\bc Y^{[k-1]}$ has at least one coordinate that is not $0$, say the $i$th. We show that the species $\bc H$ does not depend on ${\mathcal Y}_i$.
Differentiating the identity $\bc H(\bc Z,\bc Y^{[k-1]})=\bc H(\bc Z,\bc S)$ gives
\[
\frac{\bd\partial\bc H}{\bd\partial\bc Y}(\bc Z,\bc Y^{[k-1]})\cdot\bc Y^{[k-1]}\bs{'}+\frac{\bd\partial\bc H}{\bd\partial\bc Z}(\bc Z,\bc Y^{[k-1]})
=
\frac{\bd\partial\bc H}{\bd\partial\bc Y}(\bc Z,\bc S)\cdot\bc S\bs{'}+\frac{\bd\partial\bc H}{\bd\partial\bc Z}(\bc Z,\bc S).
\]
Since $\bc Y^{[k-1]}\subset\bc Y^{[k]}=\bc S$, the above identity implies equality of the first terms (resp. of the second terms), 
so that the following inclusions are equalities
\[
\frac{\bd\partial\bc H}{\bd\partial\bc Y}(\bc Z,\bc Y^{[k-1]})\cdot\bc Y^{[k-1]}\bs{'}\subset
\frac{\bd\partial\bc H}{\bd\partial\bc Y}(\bc Z,\bc S)\cdot\bc Y^{[k-1]}\bs{'}\subset
\frac{\bd\partial\bc H}{\bd\partial\bc Y}(\bc Z,\bc S)\cdot\bc S\bs{'}.
\]
Thus, $\bd\partial\bc H/\bd\partial\bc Y(\bc Z,\bc S)\cdot (\bc S\bs{'}-\bc Y^{[k-1]}\bs{'})=\bs{0}$. Since the $i$th coordinate of $(\bc S-\bc Y^{[k-1]})$ is not~0 while its value at~0 is~0, its derivative is not~0 either, so that $\bd\partial\bc H/\partial{\mathcal Y}_i(\bc Z,\bc S)=\bs{0}$. By Lemma~\ref{lemma:zero}, we conclude that $\bd\partial\bc H/\partial{\mathcal Y}_i=\bs{0}$ which indicates that $\bc H$ does not depend on ${\mathcal Y}_i$.
The same reasoning applies to all the nonzero coordinates of~$\bc Y^{[k]}-\bc Y^{[k-1]}$. We may assume, without loss of generality, that these are the last $m-p$ coordinates of the vector; therefore, the system can be split into two distinct blocks: 
\begin{itemize}
        \item an implicit strict subsystem~$({\mathcal Y}_{1:p})=\bc H_{1:p}(\bc Z,{\mathcal Y}_{1:p},0,\dots,0)$;
        \item a nonrecursive block that defines~$({\mathcal Y}_{p+1:m})$ as functions of $(\bc Z,{\mathcal Y}_{1:p})$. 
\end{itemize}
If $p=0$, then all the structures of the solution are produced by $\bc H(\bc Z,\bs{0})$ and thus $k=1$.  Otherwise, since the vector~$({\mathcal Y}_{1}^{[k]}-{\mathcal Y}_{1}^{[k-1]},\dots,{\mathcal Y}_p^{[k]}-{\mathcal Y}_p^{[k-1]})$ is equal to $\bs{0}$ and $({\mathcal Y}_{p+1:m}^{[k]})\neq({\mathcal Y}_{p+1:m}^{[k-1]})$, then $({\mathcal Y}_{1:p}^{[k-1]})\neq({\mathcal Y}_{1:p}^{[k-2]})$ by construction and thus the same reasoning can be applied to the implicit subsystem with $k$ replaced by $k'=k-1$ which is the largest integer such that $\bc Y^{[k']}\neq\bc Y^{[k'-1]}$.

At each step, $k$ is either $0$ or decreased by $1$ and the size of the implicit system is decreased by at least one; thus $k\leq m$.
\end{proof}

\begin{Algorithm}[t]{0.8\textwidth}   
\SetAlgoRefName{isPolynomial}
\caption{Detection of implicit polynomial species\label{algo:poly-Joyal}}
\DontPrintSemicolon
\Input{$\bc H=({\mathcal H}_{1:m})$: a vector of species such that $\bc H(\bs{0},\bs{0})=0$ and the Jacobian matrix ${\bd\partial \bc H}/{\bd\partial\bc Y}(\bs{0},\bs{0})$ is nilpotent.}
\Output{Answer to ``Is the solution of $\bc Y=\bc H(\bc Z,\bc Y)$ polynomial?''}
\SetAlgoNoLine
\Begin{
\SetAlgoVlined
Compute $\,\bc U:= \bc H^m(\bc Z,\bs{0})$ and  $\bc V:= \bc H^{m+1}(\bc Z,\bs{0})$\;
\lIf{($\,\bc U=\bc V$  {\em and} $\,\bc U$ is polynomial)}{\Return YES}
\lElse{\Return NO}
}
\end{Algorithm}

\begin{proposition}[Characterization of Implicit Polynomial Species]\label{prop:polynomial_implicit2}
         Let~$\bc Y=\bc H(\bc Z,\bc Y)$ be well founded at $\bs{0}$. 
         The solution~$\bc S$ of the system $\bc Y=\bc H(\bc Z,\bc Y)$ such that~$\bc S(\bs{0})=\bs{0}$ is polynomial if and only if the Jacobian matrix ${\bd\partial \bc H}/{\bd\partial\bc Y}(\bc Z,\bc Y)$ is nilpotent and the species $\bc H$ is polynomial.
\end{proposition}
\begin{proof}
For any~$m>0$, the sequence defined by~\eqref{eq:phi_poly} satisfies
\[
\bc Y^{[m+1]}-\bc Y^{[m]}\subset{\bd\partial\bc H}/{\bd\partial\bc Y}(\bc Z,\bc Y^{[m]}) \cdot(\bc Y^{[m]}-\bc Y^{[m-1]}),
\] 
which expresses the fact that any $\bc Y^{[m+1]}$-structure is an $\bc H$-structure of~$\bc Y^{[m]}$-structures and at least one of them has to be in the difference~$\bc Y^{[m]}-\bc Y^{[m-1]}$. Iterating and using the inclusion~$\bc Y^{[k]}\subset\bc Y^{[k+1]}$ for all~$k\ge0$, we obtain that
$\bc Y^{[m+1]}-\bc Y^{[m]}\subset\left({\bd\partial \bc H}/{\bd\partial\bc Y}(\bc Z,\bc{Y}^{[m]})\right)^m\cdot(\bc Y^{[1]}-\bc Y^{[0]}).$
If ${\bd\partial \bc H}/{\bd\partial\bc Y}(\bc Z,\bc Y)$ is nilpotent of order~$p\leq m$, then the right-hand side is~$\bs{0}$ so that~$\bc Y^{[m+1]}=\bc Y^{[m]}=\bc S$ in this case. If moreover~$\bc H$ is polynomial, then as a finite iteration of polynomials, this is a polynomial.

Conversely, if $\bc S$ is polynomial without zero coordinates, then~$\bc H$ has to be polynomial: if for any~$n$ there exists an~$\bc H$-structure of size at least~$n$, then $\bc H(\bc Z,\bc S)$ contains such an~$\bc S$-structure. Also, when~$\bc S$ is polynomial, the proof of the previous proposition shows that~$\bc H$ has a triangular structure from which the nilpotence of its Jacobian matrix is apparent.
\end{proof}


\subsection{Partially Polynomial Species}\label{subsub:partpol}

The concept of polynomiality can be refined in the case of multi-sort species. We start with 
$\bc F(\bc{Z}_1, \bc{Z}_2)$ a $(m_1+m_2)$-sort species. For any $\bc F$-structure~$\bs s$, we denote by $\left| \bs s \right|_1$ (resp.~$\left|\bs s \right|_2$) the size of the first (resp. second) tuple of sets in the underlying sets of~$\bs s$. 
We also let $\bc F_{=(k,n)}$ denote the subspecies of $\bc F$ such that, for any $\bs s \in \bc F_{=(k,n)}$, $\left|\bs s \right|_1=k$ and $\left|\bs s \right|_2=n$.
Also natural is to define the species~$\bc F_{\leq(k,n)}$ such that, for any structure~$\bs s \in \bc F_{\leq(k,n)}$, $\left|\bs s \right|_1\leq k$ and $\left|\bs s \right|_2\leq n$.
\begin{definition}\label{def:multi-polynomial}
        The multisort species~${\mathcal F}(\bc{Z}_1, \bc{Z}_2)$ is polynomial in the sorts~$\bc Z_1$ when, 
        for all~$n\geq 0$, the species ${\mathcal F}_{=(.,n)}=\sum_{k\geq 0}{\mathcal F}_{=(k,n)}$ is polynomial. 
\end{definition}

\begin{example}
        The species~$\Seq({\mathcal Z}_1+{\mathcal Z}_2)$ is not polynomial in ${\mathcal Z}_1$ or ${\mathcal Z}_2$, while the species~$\Seq({\mathcal Z}_1\cdot{\mathcal Z}_2)$, though not polynomial (in $\bc Z$), is polynomial  in ${\mathcal Z}_1$ and ${\mathcal Z}_2$.
\end{example}

The next question is to detect the partial polynomiality of the solutions directly from the system. 
\begin{example}
Only the first two of the following three equations
\[{\mathcal Y}={\mathcal Z}_1{\mathcal Z}_2\Seq({\mathcal Y}),\quad {\mathcal Y}={\mathcal Z}_1\Seq({\mathcal Y}{\mathcal Z}_2),\quad
{\mathcal Y}={\mathcal Z}_2+{\mathcal Y}{\mathcal Z}_1\]
have solutions that are polynomial in~${\mathcal Z}_1$.
                  \end{example}

\noindent Again, we give an effective characterization of those systems having a partially polynomial solutions.
 
\begin{proposition}[Implicit Partially Polynomial Species]\label{prop:IPFS}
        Let $\bc H=({\mathcal H}_{1:m})$ be a vector of species such that the system $\bc Y=\bc H({\mathcal Z}_1, \bc Z_2, \bc Y)$ is well founded at $\bs{0}$ and let~$\bc S({\mathcal Z}_1, \bc Z_2)$ be its solution such that $\bc S(\bs{0})=\bs{0}$.
                        The species~$\bc S({\mathcal Z}_1, \bc Z_2)$ is polynomial in ${\mathcal Z}_1$ if and only if 
        \begin{enumerate}
                \item the species $\bc S_0({\mathcal Z}_1):=\bc S({\mathcal Z}_1,\bs{0})$ is polynomial;
                \item the Jacobian matrix~$\bd\partial\bc H/\bd\partial\bc Y({\mathcal Z}_1, \bs{0}, \bc S_0({\mathcal Z}_1))$ is nilpotent;
                \item $\bc H$ is polynomial in ${\mathcal Z}_1$ and for all~$\,i\in\{1,\dots,m\}$, either the $i$th coordinate of $\bc S_0$ is~0 or $\bc H({\mathcal Z}_1, \bc Z_2, \bc Y)$ is polynomial in ${\mathcal Y}_i$.
        \end{enumerate}
\end{proposition}   
This Proposition is turned into Algorithm~\ref{algo:part-poly} to decide the partially polynomial character of an implicit species. The specialized system~$\bc Y=\bc H({\mathcal Z}_1, \bs{0}, \bc Y)$ can possibly define zero coordinates, thus we use Algorithm~\ref{algo:poly-Joyal} to check for the polynomial character of its solution. 
However, note that when this specialized system is well founded at~$\bs{0}$,  Proposition~\ref{prop:polynomial_implicit2} can be used instead; moreover, the second condition in Proposition~\ref{prop:IPFS} is a consequence of the first one by Proposition~\ref{prop:polynomial_implicit} and the test on $\bc{J\!}_0$ can be skipped.

\begin{example} Proposition~\ref{prop:IPFS} allows to conclude on the previous three equations:
  \begin{enumerate}
        \item when ${\mathcal H}={\mathcal Z}_1{\mathcal Z}_2\Seq({\mathcal Y})$, then ${\mathcal S}_0({\mathcal Z}_1)=0$, and the derivative  ${\mathcal Z}_1{\mathcal Z}_2\Seq({\mathcal Y})^2$ is $0$ at $({\mathcal Z}_1,0,0)$, thus the solution is polynomial in ${\mathcal Z}_1$;
        \item when ${\mathcal H}={\mathcal Z}_1\Seq({\mathcal Y}{\mathcal Z}_2)$, then: ${\mathcal S}_0({\mathcal Z}_1)=1+{\mathcal Z}_1\neq 0$; the specialized system~${\mathcal Y}={\mathcal Z}_1$ is well founded at~0; the species $\bc H$ is polynomial in $({\mathcal Z}_1,\bc Y)$, thus the solution is polynomial in ${\mathcal Z}_1$;
        \item when ${\mathcal H}={\mathcal Z}_2+{\mathcal Y}{\mathcal Z}_1$, the solution ${\mathcal S}={\mathcal Z}_2 \cdot \Seq({\mathcal Z}_1)$ is not polynomial in ${\mathcal Z}_1$. In that case, the derivative of ${\mathcal H}$ with respect to ${\mathcal Y}$ is ${\mathcal Z}_1$ which is not nilpotent at $({\mathcal Z}_1,0,0)$.
   \end{enumerate}     
\end{example}

\begin{Algorithm}[t]{0.79\textwidth}   
  \SetAlgoRefName{isPartiallyPolynomial}
\caption{Characterization of implicit partially polynomial species\label{algo:part-poly}}
\DontPrintSemicolon
\SetKwFunction{isPolynomial}{\ref{algo:poly-Joyal}}
\Input{$\bc H=({\mathcal H}_{1:m})$: a vector of species such that such that $\bc Y=\bc H({\mathcal Z}_1, \bc Z_2, \bc Y)$ is well founded at $\bs{0}$.}
\Output{Answer to ``Is the solution of $\bc Y=\bc H({\mathcal Z}_1, \bc Z_2, \bc Y)$ polynomial in ${\mathcal Z}_1$?''}
\SetAlgoNoLine
\Begin{
\SetAlgoVlined
\lIf{$\bc H$ is not polynomial in ${\mathcal Z}_1$}{\Return NO}\\
\Else{
\lIf{\isPolynomial{$\bc H({\mathcal Z}_1, \bs{0}, \bc Y)$}$=${\em NO}}{\Return NO}\\
\Else{Compute $\bc{J\!}_0:=\bd\partial\bc H/\bd\partial\bc Y({\mathcal Z}_1, \bs{0}, \bc S_0({\mathcal Z}_1))$\\
\lIf{$\bc{J\!}_0^m \neq 0$}{\Return NO}\\
\Else{
Compute $\bc S_0:=\bc H^m({\mathcal Z}_1, \bs{0}, \bc Y)$\;
\For{$i$ from 1 \KwTo $m$}{
        \lIf{$(\bc S_0)_i \neq 0$ {\em and }$\bc H$ is not  polynomial in ${\mathcal Y}_i$}{\Return NO}
}
\Return YES}}}
}
\end{Algorithm}

The proof relies repeatedly on the preservation of partial polynomiality by composition.
\begin{lemma}\label{lem:HS0-polynomial1} 
        Let $\bc F({\mathcal Z}_1, \bc Z_2,\bc Y)$ and $({\mathcal G}_{1:m})({\mathcal Z}_1, \bc Z_2)$ be multisort species such that $\bc G(0,\bs{0})=\bs{0}$.
        If 
        \begin{enumerate}
                \item $\bc F$ and $\bc G$ are polynomial in ${\mathcal Z}_1$ and
                \item for all~$\,i\in\{1,\dots,m\}$, either ${\mathcal G}_i({\mathcal Z}_1,\bs{0})=\bs{0}$ or $\bc F({\mathcal Z}_1, \bc Z_2, \bc Y)$ is polynomial in ${\mathcal Y}_i$,
        \end{enumerate}
        then the species $\bc F({\mathcal Z}_1, \bc Z_2, \bc G)$ is polynomial in ${\mathcal Z}_1$.
\end{lemma}
\begin{proof}
        We prove the result when ${\mathcal F}$ and ${\mathcal G}$ are single species and then, an induction on~$i$ gives that ${\mathcal F}_j({\mathcal Z}_1, \bc Z_2, \bc G_{1:i},\bc Y_{i+1\dotsc m})$ is polynomial in ${\mathcal Z}_1$ for each coordinate~$j$ of~$\bc F$.
        
Consider the subspecies of ${\mathcal F}({\mathcal Z}_1, \bc Z_2, {\mathcal G})$ whose structures are of size $(.,n)$.
        Any structure in this species is an ${\mathcal F}$-assembly of ${\mathcal Z}_1$-structures, $\bc Z_2$-structures, and  ${\mathcal G}$-structures. By definition, within the members of this assembly, at most $n$ are $\bc Z_2$-structures and none of the ${\mathcal G}$-structures is of size larger than $(.,n)$; thus, the following inclusion holds
        \[
        \left({\mathcal F}({\mathcal Z}_1, \bc Z_2, {\mathcal G})\right)_{=(.,n)}\subset {\mathcal F}_{\leq (.,n,.)}({\mathcal Z}_1, \bc Z_2,{\mathcal G}_{\leq(.,n)}).
        \]        
        If ${\mathcal F}({\mathcal Z}_1, \bc Z_2, {\mathcal Y})$ is polynomial in ${\mathcal Z}_1$ and ${\mathcal Y}$, then ${\mathcal F}_{\leq (.,n,.)}$ is polynomial, as is~${\mathcal G}_{\leq(.,n)}$; since the polynomial character is preserved by composition, ${\mathcal F}_{\leq (.,n,.)}({\mathcal Z}_1, \bc Z_2,{\mathcal G}_{\leq(.,n)})$ is also polynomial.       
        Otherwise, if ${\mathcal G}({\mathcal Z}_1,\bs{0})={0}$, then all the ${\mathcal G}$-structures are of size at least~$(.,1)$; thus ${\mathcal F}_{\leq (.,n,.)}({\mathcal Z}_1, \bc Z_2,{\mathcal G}_{\leq(.,n)})\subset{\mathcal F}_{\leq (.,n,n)}({\mathcal Z}_1, \bc Z_2,{\mathcal G}_{\leq(.,n)})$. This last species is polynomial, as the composition of the two polynomial species~${\mathcal F}_{\leq (.,n,n)}$ and~${\mathcal G}_{\leq(.,n)}$.       
        In both cases, for any $n\geq 0$, $\left({\mathcal F}({\mathcal Z}_1, \bc Z_2, {\mathcal G})\right)_{=(.,n)}$ is polynomial and the result follows.      
                                                                                                                        \end{proof}        

\begin{proof}[Proof of Proposition~\ref{prop:IPFS}]
We first establish that conditions 1. to 3. imply that $\bc S$ is polynomial in ${\mathcal Z}_1$, and for simplicity, the proof is carried out with $\bc Z_2$ reduced to a single sort.
      
      Assume that $\bc S({\mathcal Z}_1, \bc{Z}_2)$ is not polynomial in ${\mathcal Z}_1$ and let $n$ be the smallest size for which the species $\bc S_{=(.,n)}$ is not polynomial.   
      By definition, any~$\bc S$-structure of size $(.,n)$ is such that its $i$th coordinate is an ${\mathcal H}_i$-assembly of $\bc S$-structures, say $t_1, \dots, t_\ell$, such that $\left|t_j\right|_2\leq n$ for $j=1,\dots,\ell$.    
      If none of the $t_j$ is of size~$(.,n)$, then all of them belong to the species $\bc S_{<(.,n)}$ which is polynomial. Since $\bc S_{<(.,n)}({\mathcal Z}_1,0)=\bc S_{=(.,0)}=\bc S_0$, applying Lemma~\ref{lem:HS0-polynomial1} with $\bc F:=\bc H$ and $\bc G:=\bc S_{<(.,n)}$ shows that there are only a finite number of such decompositions.
      Otherwise, only one of the $t_j$ is of size~$(.,n)$ while all the other ones are of size~$(.,0)$, which means that they are $\bc S_0$-structures and so are $s_1,\dots,s_{i-1},s_{i+1},\dots,s_m$.
      This implies that the $\bc S$-structure is of the form $\bs{\alpha}\cdot\bs{\beta}$, with $\bs{\alpha}\in \bd\partial\bc H/\bd\partial\bc Y({\mathcal Z}_1, 0, \bc S_0({\mathcal Z}_1))$ and $\bs{\beta}\in\bc S$, with~$\left|\bs{\beta}\right|_2=n$. 
            By the second hypothesis, applying Lemma~\ref{lem:HS0-polynomial1} again, $\bc H({\mathcal Z}_1, 0, \bc S_0({\mathcal Z}_1))$ is polynomial in~${\mathcal Z}_1$ and so is $\bd\partial\bc H/\bd\partial\bc Y({\mathcal Z}_1, 0, \bc S_0({\mathcal Z}_1))$. If~$p$ is the order of nilpotence of this matrix, the reasoning above cannot be iterated more than $p$~times. Thus, there are only a finite number of $\bc S_{=(.,n)}$-structures that decompose in that way. Then $\bc S_{=(.,n)}$ is polynomial and $\bc S$ is polynomial in ${\mathcal Z}_1$.
      
      Conversely, assume that $\bc S$ is polynomial in ${\mathcal Z}_1$. Then, by inclusion,  $\bc S_0({\mathcal Z}_1)$ is polynomial in ${\mathcal Z}_1$ too.
      Assume now that the matrix $\bd\partial\bc H/\bd\partial\bc Y({\mathcal Z}_1, 0, \bc S_0({\mathcal Z}_1))$ is not nilpotent. 
            Then, for all$q\in \mathbb{N}$, there exists a nonzero structure $\bs{\delta}_q$ in the species~$(\bd\partial\bc H/\bd\partial\bc Y({\mathcal Z}_1, 0, \bc S_0({\mathcal Z}_1)))^q$; by construction, $|\bs{\delta}_q|_2=0$.
             Since the system we consider is well founded at $\bs{0}$, one can always find an $\bc S$-structure, say~$\bs{\gamma}$, such that $\gamma_i\neq 0$ for $i=1,\dots,m$; let $(.,{n})$ be the size of~$\bs{\gamma}$. 
            Then, there are infinitely many $\bc S$-structures of the form $\bs{\delta}_q\cdot \bs{\gamma}$, all of size $(.,{n})$, which prevents $\bc S$ from being polynomial in  ${\mathcal Z}_1$; the contradiction implies that the Jacobian matrix is nilpotent. 

      Regarding the third point, if $\bc H$ is not polynomial in ${\mathcal Z}_1$, then there exist infinitely many $\bc H$-structures of size~$(.,\ell,k)$ for some $\ell$ and $k$; and since the system is well founded at~$\bs{0}$, one can always find an $\bc S$-structure~$\bs\omega$ without zero coordinates to build infinitely many $\bc S$-structures from $\bc H$ and $\bs\omega$, their size being~$(.,\ell',k')$, with $\ell'$ and $k'$ depending on $\ell$, $k$ and the size of $\bs\omega$, which prevents the species $\bc S$ from being polynomial in ${\mathcal Z}_1$.
      Finally, assume that there exists~$i\in\{1,\dots,m\}$ such that the $i$th coordinate of $\bc S_0$ is nonzero and assume that the species~$\bc H({\mathcal Z}_1, \bc{Z}_2, \bc Y)$ is not polynomial in ${\mathcal Y}_i$.
      It means, on the one hand, that there exists a structure $\bs\omega$ in $\bc S$ such that $|\bs\omega|_2=(t_1+\dots+t_{i-1}+0+t_{i-1}+\dots+t_m)$, and on the other hand, that there exist infinitely many $\bc H$-structures of size~$(\ell_1,\ell_2,k_1+\dots+k_{i-1}+\cdot+k_{i+1}+\dots+k_m)$. Then, from~$\bs\omega$ and these $\bc H$-structures, it is possible to build infinitely many $\bc S$-structures of size~$(\cdot,\ell_2+k_1 t_1+\dots+k_{i-1} t_{i-1}+k_{i+1} t_{i+1}+\dots+k_m t_m)$, which is, again, a contradiction.
\end{proof}

\section{General Implicit Species Theorem}\label{subsec:gen_IST}

It is often the case that one defines a species by an equation or a system that does not satisfy the implicit species theorem directly. For instance, sequences ($\Seq$) can be defined by the implicit equation
\[{\mathcal Y}={\mathcal H}_{\mathcal L}({\mathcal Z},{\mathcal Y}):=1+{\mathcal Z}{\mathcal Y},\]
for which ${\mathcal H}_{\mathcal L}(0,0)=1\neq0$. Moreover, defining $\Seq$ as the limit of the iteration~${\mathcal Y}^{[n+1]}={\mathcal H}_{\mathcal L}({\mathcal Z},{\mathcal Y}^{[n]})$ with~${\mathcal Y}^{[0]}=0$ is not even possible at this stage, since the definition of composition in Section~\ref{species} demands that~${\mathcal Y}^{[n]}(0)=0$.

In the case of sequences, an easy way out is to define nonempty sequences as ${\mathcal U}=\Seq-1$, which is possible since $1\subset \Seq$. Setting~${\mathcal Y}=1+{\mathcal U}$ in the equation above gives a new equation~${\mathcal U}={\mathcal Z}({\mathcal U}+1)$ to which the implicit species theorem can be applied. More work is needed to make this idea work in general. For instance, the system
\[{\mathcal Y}_1=1+{\mathcal Z}{\mathcal Y}_1,\quad{\mathcal Y}_2=1+{\mathcal Y}_1^2\]
can be subjected to the implicit species theorem only after the translation~${\mathcal Y}_1=1+{\mathcal U}_1$, ${\mathcal Y}_2=1+1+{\mathcal U}_2$. Thus a first stage of the derivation consists in isolating the value of the solution species at~${\mathcal Z}=0$. This solution is in turn given by an implicit system, which has to have a polynomial solution in order to define only a finite number of structures of size~$n$. It turns out that this question can be solved in a unified manner, provided we first extend the definition of composition to a polynomial species composed with the species 1. Then it is possible to define a notion of well-foundedness for combinatorial systems allowing for structures of size 0, and  we finally obtain an extension of the Implicit Species Theorem to those combinatorial systems.

\subsection{General Composition}
While the composition of species~${\mathcal F}\circ{\mathcal G}$ is defined for arbitrary~${\mathcal F}$ when~${\mathcal G}(0)=0$, the composition with~${\mathcal G}=1$ is only defined when~${\mathcal F}$ is polynomial, so that the result makes sense as a species. (See Joyal's~\cite{Joyal86}\footnote{bearing in mind that here we consider what Joyal calls \emph{espèces finitaires}.}; see also~\cite[p.~111-112]{BeLaLe98}).

\begin{definition} Let ${\mathcal F}$ be  a polynomial species. The composition of ${\mathcal F}$ with 1 is defined as follows. 
\[
{\mathcal F}\circ 1[U]=\begin{cases}
\eset,&\quad\text{if $U\neq\eset$,}\\
\sum_{k\geq 0} {\mathcal F}[\{1,2,\dots,k\}]/\!\sim ,&\quad\text{otherwise}.
\end{cases}
\]
The sum in the definition is polynomial since~${\mathcal F}$ is polynomial and the equivalence classes are defined with respect to isomorphism of ${\mathcal F}$-structures (see \S\ref{parag:rel-species}).
\end{definition}

This definition extends to multisort species. We only give the statement for 2-sort species so as to avoid heavy notation.
If the species ${\mathcal F}({\mathcal Z},{\mathcal Y})$ is partially polynomial in~${\mathcal Z}$, then its composition with~1 is defined by
\[
{\mathcal F}(1,{\mathcal Y})[U,V]=\begin{cases}
\eset,&\quad\text{if $V\neq\eset$,}\\
\sum_{k\geq 0} {\mathcal F}[U,\{1,2,\dots,k\}]/\!\sim ,&\quad\text{otherwise}.
\end{cases} 
\]       
Again, the sum is polynomial since ${\mathcal F}$ is partially polynomial and the equivalence relation is now isomorphism for the second set: $s$ and $t$ in ${\mathcal F}[U,V]$ are equivalent if there exists a permutation~$\sigma\in{\mathcal P}_k$ such that ${\mathcal F}[\operatorname{Id},\sigma](s)=t$.

Finally, since sums can be viewed as multisort species, the composition is more generally defined for ${\mathcal F}\circ{\mathcal G}$  for polynomial~${\mathcal F}$ and arbitrary ${\mathcal G}$. For instance, if~${\mathcal G}(0)=1$, the polynomial ${\mathcal F}({\mathcal X}+{\mathcal Y})$ can be composed with~${\mathcal X}=1$ and then with~${\mathcal G}(0)-1$.

Many properties satisfied by the classical composition of species hold, and in particular the following, with the same proof as Lemma~\ref{lemma:zero}.
\begin{lemma}\label{lemma:zero1}
         Let $\bc G=({\mathcal G}_{1:m})$ be a vector of species such that ${\mathcal G}_i\neq 0$, for $i=1,\dots,m$. For any vector of species~$\bc F({\mathcal Y}_{1:m})$, if $\bc F({\mathcal G}_{1:m})=\bs{0}$, then $\bc F=\bs{0}$.       
\end{lemma}

\subsection{General Implicit Species}
This section  extends the definition of well-founded combinatorial systems to cases when~$\bc H(\bs{0},\bs{0})$ is  not necessarily~$\bs{0}$. It gives rise to Algorithm~\ref{algo:wf} to decide whether a system is well founded or not. This characterization then leads to our General Implicit Species Theorem.

\begin{definition}[Well-founded combinatorial system]\label{def:wf}
Let $\bc H$ be a vector of  species. The combinatorial system $\bc Y=\bc H(\bc Z,\bc Y)$ is said to be  \emph{well founded} when the iteration
\begin{equation}\label{eq:ite-def-wf}\tag{\ref{joyal-iteration}}
        \bc Y^{[0]}=\bd{0}\quad\text{and}\quad \bc Y^{[n+1]}=\bc H(\bc Z,\bc Y^{[n]}),\quad n\ge0
\end{equation}
is well defined, defines a convergent sequence and the limit~$\bc S$ of this sequence has no zero coordinate.
\end{definition}         
In this definition, \emph{`well defined'} means that the composition of species is actually defined, that is, for each sort~${\mathcal Y}_i$, either~$\bc H$ is polynomial in~${\mathcal Y}_i$ or ${\mathcal Y}^{[n]}_i(\bs{0})=0$ for all~$n$.

The restriction on zero coordinates is quite natural and already appears in the more specific framework of combinatorial specifications considered in~\cite{Zimmermann91}
\footnote{
Actually, the definition in~\cite{Zimmermann91} does not forbid zero coordinates. However, the corresponding procedure   to detect well-founded systems (Algo. B) rejects those with zero coordinates, which comes back to our definition.
}. 
This allows, in particular, to give a characterization of well-founded systems by necessary and sufficient conditions.

When  $\bc Y=\bc H(\bc Z,\bc Y)$ is a system such that $\bc H(\bs{0},\bs{0})\ne \bs{0}$, we define a companion system with a new sort ${\mathcal Z}_1$ marking the empty species, and show the relations between  iterations  on both systems;  finally the original system is well founded if and only if its companion system is well founded at 0, with a solution that is partially polynomial in ${\mathcal Z}_1$.
  
\def\Sun{\bc S_1}  
\begin{definition}If $\bc Y=\bc H(\bc Z,\bc Y)$ is a system such that $\bc H(\bs{0},\bs{0})\ne \bs{0}$, its \emph{companion system} is defined by 
  \begin{equation*}
	\bc Y=\bc K({\mathcal Z}_1, \bc Z,\bc Y), \quad\text{where}\quad  \bc K= \bc H(\bc Z,\bc Y)-\bc H(\bs{0},\bs{0})+{\mathcal Z}_1\bc H(\bs{0},\bs{0}).
\end{equation*}
\end{definition}

\begin{theorem}[Characterization of well-founded systems]\label{th:WF}
Let $\bc H=({\mathcal H}_{1:m})$ be a vector of species. The combinatorial system $\bc Y=\bc H(\bc Z,\bc Y)$ is well founded if and only if
\begin{enumerate}
	\item the companion system $\bc Y=\bc K({\mathcal Z}_1,\bc Z,\bc Y)$ is well founded at $\bs{0}$ and,
	\item  if the species~$\Sun({\mathcal Z}_1,\bc Z)$ is   the solution of $\bc Y=\bc K({\mathcal Z}_1,\bc Z,\bc Y)$ with~$\Sun(0,\bs{0})=\bs{0}$, then $\Sun({\mathcal Z}_1,\bc Z)$ is polynomial in ${\mathcal Z}_1$.
\end{enumerate}
In this case, the limit of~\eqref{eq:ite-def-wf} is~$\Sun(1,\bc Z)$. 
\end{theorem}       
\begin{proof} 
                                        
        Assume that conditions 1. and 2. are satisfied. The existence of the solution~$\Sun$ follows from the implicit species theorem, since $\bc Y=\bc K({\mathcal Z}_1,\bc Z,\bc Y)$ is well founded at~$\bs{0}$. Then~$\Sun$ is the limit of the sequence~$(\bc T^{[n]})$ defined by
    \begin{equation}\label{eq:ite-K}
            \bc T^{[0]}=\bd{0}\quad\text{and}\quad \bc T^{[n+1]}=\bc K({\mathcal Z}_1,{\mathcal Z},\bc T^{[n]}),\quad n\ge0.
    \end{equation}
For all $n\geq 0$, since $\bc T^{[n]}\subset \Sun$, the species $\bc T^{[n]}$ is polynomial in ${\mathcal Z}_1$ and can be composed with~1.
By induction, we now show that $\bc Y^{[n]}(\bc Z)=\bc T^{[n]}(1,\bc Z)$ for all $n\geq 0$. From there it follows that the iteration in Equation~\eqref{eq:ite-def-wf} is well defined.
The property is clear for~$n=0$. Assume that it holds for~$n$.
        Proposition~\ref{prop:IPFS} implies that for all~$i\in\{1,\dots,m\}$, either $\bc K({\mathcal Z}_1, \bc Z, \bc Y)$ is polynomial in ${\mathcal Y}_i$, or the $i$th coordinate of $\Sun({\mathcal Z}_1,\bs{0})$ is~0; this means that the $i$th coordinate of $\bc T^{[n]}(1,\bs{0})\subset{\bc S_1}(1,\bs{0})$ is~0 (applying Lemma~\ref{lemma:zero1}). Thus, the composition of $\bc K$ with~1 is possible in the following equation that proves the induction
\[
\bc T^{[n+1]}(1,\bc Z)=\bc K(1,{\mathcal Z},\bc T^{[n]}(1,\bc Z))=\bc H(\bc Z,\bc Y^{[n]})=\bc Y^{[n+1]},
\]
the second identity being given by the induction hypothesis. 
        As a consequence, $(\bc Y^{[n]})$ converges to the limit of $(\bc T^{[n]}(1,\bc Z))$, that is~$\bc S(\bc Z):=\Sun(1,\bc Z)$. Finally, the system $\bc Y=\bc K({\mathcal Z}_1,\bc Z,\bc Y)$ being well founded at $\bs{0}$, the species~$\Sun$ has no zero coordinate and by Lemma~\ref{lemma:zero1}, neither does $\bc S$. 

Conversely, assume that $\bc Y=\bc H(\bc Z,\bc Y)$ is well founded. If $\bc H(\bs{0},\bs{0})=\bs{0}$, then $\bc K=\bc H$ and the two properties are trivially satisfied; therefore, we only consider the case when $\bc H(\bs{0},\bs{0})\neq\bs{0}$. First, in order to check that $\bc Y=\bc K({\mathcal Z}_1,\bc Z,\bc Y)$ is well founded at $\bs{0}$, it is sufficient to check for the convergence of the sequence~$(\bc T^{[n]})_{n\in\N}$, 
the other properties being inherited from $\bc H$. Applying the second item of Lemma~\ref{lem:contaK} below, the convergence of~$(\bc T^{[n]})$ follows from that of~$(\bc Y^{[n]})$.
Let us now turn to the polynomiality of the solution.
For each~$k$, the convergence of~$\bc Y^{[n]}$ and Lemma~\ref{lem:contaK} imply that there exists~$n$ such that $\bc T_{=(.,k)}^{[n]}=(\bc S_1)_{=(.,k)}$ and moreover~$\bc T^{[n]}$ is polynomial in~${\mathcal Z}_1$. Thus for each~$k$, $(\bc S_1)_{=(.,k)}$ is polynomial in~${\mathcal Z}_1$, which means that~$\bc S_1$ itself is polynomial in~${\mathcal Z}_1$.

\begin{lemma}\label{lem:contaK}
      Let $\bc Y=\bc H(\bc Z,\bc Y)$ be well founded and $\bc Y=\bc K({\mathcal Z}_1, \bc Z,\bc Y)$ be its companion system. 
        Let~$(\bc Y^{[n]})_{n\in\N}$ and $(\bc T^{[n]})_{n\in\N}$ be the sequences defined by Equations~\eqref{eq:ite-def-wf} and~\eqref{eq:ite-K}. For all $n,k>0$,  
        \begin{enumerate}
                \item $\bc T^{[n]}$ is polynomial in ${\mathcal Z}_1$ and $\bc Y^{[n]}(\bc Z)=\bc T^{[n]}(1,\bc Z)$;
                \item $\bc Y^{[n]}=_k \bc Y^{[n+1]} \implies \bc T^{[n]}_{\leq(.,k)} = \bc T^{[n+1]}_{\leq(.,k)}\implies \bc T^{[n]}=_k \bc T^{[n+1]}$.
        \end{enumerate}
\end{lemma}
\begin{proof}
        The first point is proved by induction on $n$. The case~$n=0$ follows from the definition.
        By definition, $\bc K$ is polynomial in ${\mathcal Z}_1$.
For each $i\in\{1,\dots,m\}$ such that~$\bc T^{[n]}_i({\mathcal Z}_1,\bs{0})\neq0$, Lemma~\ref{lemma:zero1} shows that~$\bc T^{[n]}_i(1,\bs{0})\neq0$ and by the induction hypothesis this is~$\bc Y_i^{[n]}(0)$. Iteration~\eqref{eq:ite-def-wf} being well defined shows that in this case $\bc H$ is polynomial in~$\bc Y_i$ and therefore so is~$\bc K$. Thus, by  Lemma~\ref{lem:HS0-polynomial1}, the species~$\bc T^{[n+1]}$ is polynomial in ${\mathcal Z}_1$. Moreover, the composition of $\bc K$ with~1 being well defined, one has
\[
\bc T^{[n+1]}(1,\bc Z)=\bc K(1,\bc Z, \bc Y^{[n]})=\bc H(\bc Z, \bc Y^{[n]})=\bc Y^{[n+1]}.
\]   
        
        Again by induction on $n$, the sequence~$(\bc T^{[n]})$ is increasing and in particular $\bc T^{[n]}_{\leq(.,k)} \subset \bc T^{[n+1]}_{\leq(.,k)}$. Then, since $\bc Y^{[n]}=_k \bc Y^{[n+1]}$ one has $\bc T^{[n]}_{\leq(.,k)}(1,\bc Z) = \bc T^{[n+1]}_{\leq(.,k)}(1,\bc Z)$. Let $\bc D_{n,k}$ be the species defined by $\bc D_{n,k}({\mathcal Z}_1,\bc Z)=\bc T^{[n+1]}_{\leq(.,k)} - \bc T^{[n]}_{\leq(.,k)}$. By hypothesis $\bc D_{n,k}(1,\bc Z)=\bc Y^{[n+1]}-\bc Y^{[n]}=\bs{0}$ and thus, applying Lemma~\ref{lemma:zero1}, $\bc D_{n,k}({\mathcal Z}_1,\bc Z)=\bs{0}$, which rewrites as $\bc T^{[n]}_{\leq(.,k)} = \bc T^{[n+1]}_{\leq(.,k)}$. This implies in particular that
        \[
        \bc T^{[n]}_{\leq k}=\sum_{i=0}^k\bc T^{[n]}_{\leq(k-i,i)} = \sum_{i=0}^k\bc T^{[n+1]}_{\leq(k-i,i)} =  \bc T^{[n+1]}_{\leq k}.\qedhere
        \]
        \end{proof}
\renewcommand{\qedsymbol}{}
\end{proof}

\begin{Algorithm}[t]{0.76\textwidth}   
  \SetAlgoRefName{isWellFounded}
\caption{Characterization of well-founded systems \label{algo:wf}}
\DontPrintSemicolon
\SetKwFunction{isWellFoundedAtZero}{\ref{algo:wf0}}
\SetKwFunction{isPartiallyPolynomial}{\ref{algo:part-poly}}
\Input{$\bc H=({\mathcal H}_{1:m})$: a vector of species.}
\Output{Answer to ``Is the system $\bc Y=\bc H(\bc Z,\bc Y)$ well founded?''}
\SetAlgoNoLine
\Begin{
\SetAlgoVlined
Compute $\bc K:=\bc H(\bc Z,\bc Y)-\bc H(\bs{0},\bs{0})+{\mathcal Z}_1\bc H(\bs{0},\bs{0})$\;
\lIf{\isWellFoundedAtZero{$\bc K$}$=${\em NO}}{\Return NO}\\
\Else{
\lIf{\isPartiallyPolynomial{$\bc K$}$=${\em NO}}{\Return NO}\\
\lElse{\Return YES}
}}
\end{Algorithm}

\begin{example}
Here is how the theorem proves that the system ${\mathcal Y}_1=1+{\mathcal Z}{\mathcal Y}_1,~{\mathcal Y}_2=1+{\mathcal Y}_1^2$ from above is well founded. 
Its companion system $\bc Y=\bc K({\mathcal Z}_1, \bc Z,\bc Y)$ is defined by
\[
{\mathcal Y}_1={\mathcal Z}{\mathcal Y}_1+{\mathcal Z}_1,\quad{\mathcal Y}_2={\mathcal Y}_1^2+{\mathcal Z}_1.
\]
Using Theorem~\ref{th:carac-wf0}, the value of the Jacobian matrix~$\bd\partial\bc K/\bd\partial\bc Y=\left(\begin{smallmatrix}{\mathcal Z}&0\\ 2{\mathcal Y}_1&0\end{smallmatrix}\right)$ implies that $\bc Y=\bc K({\mathcal Z}_1, {\mathcal Z},\bc Y)$ is well founded at~$\bs{0}$.
In order to check that the solution $\Sun$ of the companion system is polynomial in ${\mathcal Z}_1$, we use Proposition~\ref{prop:IPFS}. The second condition follows from the nilpotence of       $\bd\partial\bc K/\bd\partial\bc Y({\mathcal Z}_1, 0,\bc Y)=\left(\begin{smallmatrix}0&0\\ 2{\mathcal Y}_1&0\end{smallmatrix}\right)$. Next, since $\bc K$ is polynomial, Proposition~\ref{prop:polynomial_implicit} shows that $\Sun({\mathcal Z}_1,0)$ is polynomial. In addition, $\bc K$ is polynomial in ${\mathcal Z}_1$ and $\bc Y$, which is the third condition for the polynomiality of~$\Sun$.

On the contrary, the system ${\mathcal Y}_1=1+{\mathcal Y}_2{\mathcal Y}_1,~{\mathcal Y}_2=1$ is not well founded. 
Its companion system $\bc Y=\bc K({\mathcal Z}_1, \bc Z,\bc Y)$ is defined by
\[
{\mathcal Y}_1={\mathcal Y}_2{\mathcal Y}_1+{\mathcal Z}_1,\quad{\mathcal Y}_2={\mathcal Z}_1.
\]
Once again, the well-foundedness of this system is easily checked but its solution is not polynomial in ${\mathcal Z}_1$ since the Jacobian matrix~$\bd\partial\bc K/\bd\partial\bc Y({\mathcal Z}_1, 0,\bc Y)=\left(\begin{smallmatrix}{\mathcal Y}_2&{\mathcal Y}_1\\ 0&0\end{smallmatrix}\right)$ is not nilpotent.
\end{example}

We can now state a general Implicit Species Theorem, concerning well-founded systems allowing for structures of size 0.

\begin{theorem}[General Implicit Species Theorem]\label{th:GIST}
Let $\bc H=({\mathcal H}_{1:m})$ be a vector of species, such that the system  $\bc Y=\bc H(\bc Z,\bc Y)$ is well founded. Then, this system admits a solution~$\bc S$ such that~${\bc S(\bs{0})=\bc H^{m}(\bs{0},\bs{0})}$, which is unique up to isomorphism. 
\end{theorem}     
\begin{proof}

By Theorem~\ref{th:WF}, a solution is given by $\bc S(\bc Z)=\Sun(1,\bc Z)$, with $\Sun$ the solution of the companion system of~$\bc Y=\bc H(\bc Z,\bc Y)$. It follows that $\bc S(\bs{0})=\Sun(1,\bs{0})=\bc K^{m}(1,\bs{0},\bs{0})=\bc H^{m}(\bs{0},\bs{0})$,
where the second equality is a consequence of Proposition~\ref{prop:polynomial_implicit}. 

Let us consider the shifted system $\bc Y=\bc H(\bc Z,\bc S(\bs{0})+\bc Y)-\bc S(\bs{0})$; note that the subtraction is well defined since $\bc S(\bs{0})= \bc H(\bs{0},\bc S(\bs{0}))\subset \bc H(\bc Z,\bc S(\bs{0}))\subset \bc H(\bc Z,\bc S(\bs{0})+\bc Y)$.
This shifted system satisfies the conditions of Joyal's Implicit Species Theorem~\ref{th:IST} and for any species~$\,\bc U$, solution of~$\bc Y=\bc H(\bc Z,\bc Y)$ such that~$\,\bc U(\bs{0})=\bc S(\bs{0})$, the species~$\,\bc U-\bc U(\bs{0})$ is a solution of the shifted system and is~$\bs{0}$ at~$\bs{0}$. By Theorem~\ref{th:IST}, this solution is isomorphic to~$\bc S-\bc S(\bs{0})$, so that $\,\bc U$ is isomorphic to~$\bc S$. 
\end{proof}

In this generalized setting, it is also possible to characterize systems with polynomial solutions.

\begin{proposition}[General Implicit Polynomial Species]\label{prop:GPI}           Let~$\bc H=({\mathcal H}_{1:m})$ be a vector of species such that $\bc Y=\bc H(\bc Z,\bc Y)$ is well founded. Let $(\bc Y^{[n]})_{n\in\N}$ be the sequence of species defined by Eq.~\eqref{eq:phi_poly} and let~$\bc S$ be the solution of the system~$\bc Y=\bc H(\bc Z,\bc Y)$ such that $\bc S(\bs{0})=\bc H^{m}(\bs{0},\bs{0})$. The following three properties are equivalent:
          \begin{enumerate}[i)]
                 \item the species~$\bc S$ is polynomial;
                 \item the species~$\bc Y^{[m]}$ is polynomial and  $\bc Y^{[m]}=\bc Y^{[m+1]}$;
                 \item the Jacobian matrix ${\bd\partial \bc H}/{\bd\partial\bc Y}(\bc Z,\bc Y)$ is nilpotent and the species $\bc H$ is polynomial.
          \end{enumerate} 
\end{proposition} 
\begin{proof} The first two properties are equivalent.
        Indeed, if~$\bc Y^{[m]}=\bc Y^{[m+1]}$, then by induction this species is~$\bc S$ and its polynomiality follows from that of~$\bc Y^{[m]}$. Conversely, if~$\bc S$ is polynomial, then~$\bc S-\bc S(\bs{0})$ is a polynomial solution of the shifted system~$\bc Y=\bc H(\bc Z,\bc S(\bs{0})+\bc Y)-\bc S(\bs{0})$, which is~$\bs{0}$ at~$\bs{0}$ and the conclusion is given by Proposition~\ref{prop:polynomial_implicit}. 
        
        We turn to the third property. Let us consider the companion system $\bc Y=\bc K({\mathcal Z}_1,\bc Z,\bc Y)$ and its solution~$\Sun$. 
        By definition, ${\bd\partial \bc H}/{\bd\partial\bc Y}(\bc Z,\bc Y)={\bd\partial \bc K}/{\bd\partial\bc Y}(\bc Z,\bc Y)$ and $\bc H$ is polynomial if and only if $\bc K$ is polynomial. 
        Thus, by Proposition~\ref{prop:polynomial_implicit}, in order to prove the equivalence of properties {\em i)} and {\em iii)}, it is sufficient to prove that $\bc S$ is polynomial if and only if $\Sun$ is polynomial. 
        Since $\bc S=\Sun(1,\bc Z)$, that the polynomiality of $\Sun({\mathcal Z}_1,\bc Z)$ implies the polynomiality of $\bc S$ follows by composition. 
        Conversely, if ${\Sun}(1,\bc Z)$ is polynomial, then ${\Sun}(1,\bc Z)_{\geq d}=0$ for some $d>0$. 
        Thus, by Lemma~\ref{lemma:zero1}, ${\Sun}({\mathcal Z}_1,\bc Z)_{\geq d}=0$ and $\Sun={\Sun}({\mathcal Z}_1,\bc Z)_{< d}$ is polynomial.
\end{proof}

\section{Newton's Iteration}\label{subsec:newton}
In this section we show how Newton's iteration can be lifted combinatorially for the construction of  the solutions of all well-founded systems $ \bc Y=\bc H(\bc Z, \bc Y)$ of combinatorial equations on species. 
Moreover, for an effective construction, the iteration can be computed on truncated species.
This construction has quadratic convergence in the following sense.

\begin{definition}
        The convergence of a sequence $(\bc F^{[n]})_{n\ge 0}$ to a (vector of) species $\bc F$ is  
        \emph{quadratic} if the contact doubles at each iteration; more precisely, if $\bc F^{[n]}$ has contact of order $k$
        with~$\bc F$, then $\bc F^{[n+1]}$ has contact of order~$2k+1$ with $\bc F$.        
\end{definition}
 
For the case of one equation, such a combinatorial Newton iteration has been introduced by Décoste, Labelle and Leroux~\cite{DeLaLe82}, who showed that the equation~${\mathcal Y}={\mathcal F}({\mathcal Z},{\mathcal Y})$ is solved by
\[
{\mathcal Y}^{[n+1]}={\mathcal Y}^{[n]}+\Seq\left(\frac{\partial{{\mathcal F}}}{\partial{{\mathcal Y}}}({\mathcal Z},{\mathcal Y}^{[n]})\right)\cdot({\mathcal F}({\mathcal Z},{\mathcal Y}^{[n]})-{\mathcal Y}^{[n]}),\qquad {\mathcal Y}^{[0]}=0.
\]

The main point here is that given the initial point~$0$, not only does the iteration converge, but the limit is the desired value. Our proof  relies on an extension of that combinatorial Newton iteration to the case of well-founded  systems. We  use the simple combinatorial interpretation of ``blooming'', due to Labelle~\cite{Labelle85b}, for the inverse of the Jacobian matrix, that  appears in Newton's iteration:
\[
\left(\Id-\frac{\bd\partial\bc H}{\bd\partial\bc Y}(\bc Z, \bc Y)\right)^{-1}=\sum_{k\ge 0}\left(\frac{\bd\partial\bc H}{\bd\partial\bc Y}(\bc Z, \bc Y)\right)^{k}.
\]

    Figure~\ref{fig:inv-jac} shows a graphical representation of a typical structure of the inverse of the Jacobian matrix, as a sequence of 
$\frac{\partial{{\mathcal H}}_i}{\partial{{\mathcal Y}}_j}$-structures, with consistent replacements of buds 
$\ds
\frac{\partial {{\mathcal H}}_i}{\partial{{\mathcal Y}}_{i_1}}\cdot
\frac{\partial{{\mathcal H}}_{i_1}}{\partial{{\mathcal Y}}_{i_2}}
\cdot\frac{\partial{{\mathcal H}}_{i_2}}{\partial{{\mathcal Y}}_{i_3}}\dotsm
$

\begin{figure} \centerline{\includegraphics[width=0.85\textwidth]{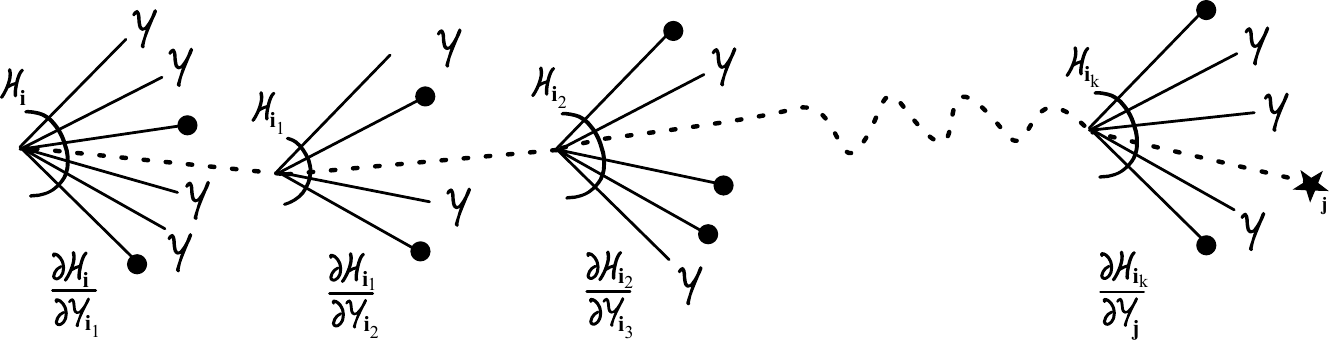}	}
\caption{A typical structure of the inverse of the Jacobian matrix of~$\bc H({\mathcal Z},\bc Y)$.}\label{fig:inv-jac}
\end{figure}

\begin{definition}
Let the system~$\bc Y=\bc H(\bc Z,\bc Y)$ be well founded; its \emph{combinatorial Newton operator} is defined by
\begin{equation}\label{eq:newton_operator}
\bc N_{\bc H}(\bc Z,\bc Y)
	=\bc Y+\left(\sum_{k\ge 0}\left(\frac{\bd\partial\bc H}{\bd\partial\bc Y}(\bc Z, \bc Y)\right)^{k}\right)\cdot\left(\bc H\left(\bc Z,\bc Y\right)-\bc Y\right).
\end{equation}
\end{definition}
Note that in general, $\bc N_{\bc H}$ is not a species but only a \emph{virtual} species (see~\cite[Ch.~2.5]{BeLaLe98}) because it is not necessarily the case that~$\bc Y\subset\bc H(\bc Z,\bc Y)$. However, we are only going to apply this operator to species for which this inclusion holds, so that only actual species will occur.

\subsection{Quadratic Convergence}
We now prove that Newton's iteration solves all systems to which the general implicit species theorem applies.
We thus rely on the existence of a unique solution (Thm.~\ref{th:GIST}), and show that  Newton's iteration provides quadratic convergence to it.
\begin{theorem}\label{th:newton}
Let $\bc Y=\bc H(\bc Z,\bc Y)$ be a well-founded system. The sequence \begin{equation}\label{ite:newt}\tag{${\mathcal N}$}
 \bc Y^{[0]}=\bs{0}\quad \mbox{and} \quad \bc Y^{[n+1]}=\bc N_{\bc H}(\bc Z,\bc Y^{[n]}),\quad n\ge0
\end{equation}
is well defined and converges quadratically to the solution~$\bc S$ of the system with~$\bc S(\bs{0})=\bc H^m(\bs{0},\bs{0})$.
\end{theorem} 

Note that when $\bc H(\bs 0,\bs 0)\neq \bs{0}$, the first iterations may not have contact~0 with $\bc{S}$. However, the theorem states that this happens eventually (indeed for $n\leq m$, see Lemma~\ref{lem:ini-contact}), and that from this point on, contact is doubled at each step.

\begin{lemma}Newton's iteration is well defined, that is $\bc Y^{[n]}\subset\bc H(\bc Z,\bc Y^{[n]})$ for all $n\in{\mathbb N}$.
\end{lemma}
\begin{proof}
The proof of this inclusion is by induction. For~$n=0$ this is a consequence of~$\bc Y^{[0]}$ being the empty species. If the property is satisfied for~$n$, then we use Taylor's formula truncated at the first order~\eqref{eq:taylor1} with~${\mathcal A}=\bc Y^{[n]}$ and~${\mathcal B}=\bc Y^{[n+1]}$. The required inclusion comes from the first summand of~$\bc N_{\bc H}(\bc Z, \bc Y^{[n]})$ being~$\bc Y^{[n]}$. Thus we get
\begin{align*}
	\bc H(\bc Z, \bc Y^{[n+1]})&\supset \bc H(\bc Z,\bc Y^{[n]})+\frac{\bd\partial\bc H}{\bd\partial\bc Y}(\bc Z,\bc Y^{[n]})\cdot (\bc Y^{[n+1]}-\bc Y^{[n]}),\\
	&\supset \bc Y^{[n]} + (\bc H(\bc Z,\bc Y^{[n]})-\bc Y^{[n]})+\frac{\bd\partial\bc H}{\bd\partial\bc Y}(\bc Z,\bc Y^{[n]})\cdot \sum_{k\ge 0}\left(\frac{\bd\partial\bc H}{\bd\partial\bc Y}(\bc Z, \bc{Y^{[n]}})\right)^{k}\!\!\cdot(\bc H(\bc Z,\bc Y^{[n]})-\bc Y^{[n]}),\\    
	&\supset \bc Y^{[n]} + \sum_{k\ge 0}\left(\frac{\bd\partial\bc H}{\bd\partial\bc Y}(\bc Z, \bc{Y^{[n]}})\right)^{k}\cdot(\bc H(\bc Z,\bc Y^{[n]})-\bc Y^{[n]})=\bc Y^{[n+1]},
\end{align*}
where in the second line we use the induction hypothesis to rewrite~$\bc H(\bc Z,\bc Y^{[n]})$ and the definition of the iteration to rewrite~$\bc Y^{[n+1]}$.
\end{proof}

We now observe that the species constructed by Newton's iteration do not overlap.
\begin{lemma}\label{lem:nonambig}
Let $\bc H(\bc Z,\bc Y)$ be a multisort species and $\,\bc U$ a subspecies of~$\bc H(\bc Z,\bc U)$. Then the species $\,\bc U$ and 
\begin{equation}\label{eq:nonambig}
\left(\frac{\bc\partial\bc H} {\bc\partial\bc Y}(\bc Z,\bc U)\right)^{k} \cdot\left(\bc H(\bc Z,\bc U)-\bc U\right),\quad k\in\N
\end{equation}
are pairwise disjoint.
\end{lemma}
\begin{proof}
The proof is by induction on $k$. When $k=0$, this is obvious. Otherwise, a structure $\bs\alpha$ of the species~$(\bc\partial{\bc H}/\bc\partial{\bc Y}(\bc Z,\bc U))^k\cdot(\bc H(\bc Z,\bc U)-\bc U)$ can be decomposed  as the product of two structures~$\bs\beta$ and~$\bs\delta$, with~$\bs\beta \in \bc\partial{\bc H}/\bc\partial{\bc Y}(\bc Z,\bc U)$ and $\bs\delta \in (\bc\partial{\bc H}/\bc\partial{\bc Y}(\bc Z,\bc U))^{k-1}\cdot(\bc H(\bc Z,\bc U)-\bc U)$.  This decomposition is unique and implies that $\bs\delta$ is not a $\,\bc U$-structure (by induction). Since $\bs\alpha$ is an $\bc H$-assembly which has~$\bs\delta$ as a member, the structure~$\bs\alpha$ does not belong to $\bc H(\bc Z,\bc U)$. Since the species~$\bc H(\bc Z,\bc U)$ contains~$\,\bc U$, then $\bs\alpha$ cannot be a $\,\bc U$-structure.                    
\end{proof} 
Next, we show that the application of Newton's operator improves the contact quadratically.
\begin{lemma}\label{lem:quad-cvg}
Let $\,\bc U$ be a subspecies of $\bc S$, solution of the well-founded system $\bc Y=\bc H(\bc Z,\bc Y)$.
If $\,\bc U$ is a subspecies of~$\bc H(\bc Z,\bc U)$ and  $\,\bc U=_k{\bc S}$, then $\bc N_{\bc H}(\bc Z,\bc U)\subset{\bc S}$ and $\bc N_{\bc H}(\bc Z,\bc U)=_{2k+1}{\bc S}$.
\end{lemma}
\begin{proof}
        First, recall that in the multisort case, the contact is with respect to the sum of sizes of the underlying sets.
                
Applying~$\bc H$ on both sides of $\,\bc U\subset\bc S$ implies that $\bc H(\bc Z,\bc U)\subset\bc S$; then a structure of the species ${\bc\partial\bc H/\bc\partial\bc Y(\bc Z,\bc U)(\bc H(\bc Z,\bc U)-\bc U)}$ is an $\bc H$-assembly of $\,\bc U$-structures and a unique $(\bc H(\bc Z,\bc U)-\bc U)$-structure that are all $\bc S$-structures. Thus,
\[
\frac{\bc\partial\bc H} {\bc\partial\bc Y}(\bc Z,\bc U) \cdot(\bc H(\bc Z,\bc U)-\bc U) \subset \bc H(\bc Z,\bc S)\subset\bc S.  
\]        
Then, an induction on~$k$ shows that all species of~\eqref{eq:nonambig} are subspecies of~$\bc S$. By Lemma~\ref{lem:nonambig}, they are distinct, so that their sum, namely~$\bc N_{\bc H}(\bc Z,\bc U)$, is also a subspecies of~$\bc S$.
          
Let now~$\bs{\alpha}$ be an $\bc S$-structure of size at most~$2k+1$. 
Since $\,\bc U$ is a subspecies of $\bc N_{\bc H}(\bc Z,\bc U)$, we only consider the case when $\bs{\alpha}$ is not a $\,\bc U$-structure.
By definition, $\bs{\alpha}$ is an $\bc H$-assembly of $\bc S$-structures. 
If all the $\bc S$-structures composing $\bs{\alpha}$ are of size at most~$k$, then they are $\,\bc U$-structures and $\bs{\alpha}$ belongs to $\bc H(\bc Z,\bc U)-\bc U\subset\bc N_{\bc H}(\bc Z,\bc U)$.
Otherwise, at most one of the~$\bc S$-structures composing~$\bs\alpha$ is of size larger than~$k$: two or more would give~$\bs\alpha$ a size larger than~$2k+1$. Thus, $\bs\alpha$ rewrites as $\bs\beta\times\bs\delta$ with $\bs\beta\in\bs\partial{\bc H}/\bs\partial{\bc Y}(\bc Z,\bc U)$ and $\bs\delta$ an $\bc S$-structure of size larger than~$k$ and at most~$2k+1$. Applying recursively the same reasoning to $\bs\delta$ up to exhaustion, one has $\bs\delta\in \left(\bs\partial{\bc H}/\bs\partial{\bc Y}(\bc Z,\bc U)\right)^{\ell} \left(\bc H(\bc Z,\bc U)-\bc U\right)$, for some $\ell>0$ and thus $\bs{\alpha} \in \bc N_{\bc H}(\bc Z,\bc U)$, which proves that $\bc N_{\bc H}(\bc Z,\bc U)=_{2k+1}{\bc S}$.
\end{proof}

The proof of Theorem~\ref{th:newton} is now concluded by induction from Lemma~\ref{lem:quad-cvg} and an initial contact provided by the following lemma.
\begin{lemma}\label{lem:ini-contact}Let~$\bc H=({\mathcal H}_{1:m})$ be a species such that the system~$\bc Y=\bc H(\bc Z,\bc Y)$ is well founded and let~$\bc S$ be its solution. Then there exists~$p\le m$ such that $\bc Y^{[p]}=\bc S(\bs{0})=\bc H^m(\bs{0},\bs{0})$, where $(\bc Y^{[n]})$ is the sequence defined by Eq.~\eqref{ite:newt}.
\end{lemma}
\begin{proof}
	If the system is well founded at~0, then~$\bc S(\bs{0})=\bs{0}$ so that~$p=0$ has the required property.
Otherwise, as in the proof of Theorem~\ref{th:GIST}, we consider the companion system ~$\bc Y=\bc K({\mathcal Z}_1,\bc Z,\bc Y)$ with~$\bc K({\mathcal Z}_1,\bc Z,\bc Y)=(\bc H(\bc Z,\bc Y)-\bc H(\bs{0},\bs{0}))+{\mathcal Z}_1\bc H(\bs{0},\bs{0})$. The   Newton operator associated to~$\bc K$ is
\begin{align*}
\bc N_{\bc K}({\mathcal Z}_1,\bc Z,\bc Y)&=
	\bc Y+\left(\sum_{k\ge 0}\left(\frac{\bd\partial\bc K}{\bd\partial\bc Y}({\mathcal Z}_1,\bc Z, \bc Y)\right)^{k}\right) \left(\bc K\left({\mathcal Z}_1,\bc Z,\bc Y\right)-\bc Y\right)\\
&=	\bc Y+\left(\sum_{k\ge 0}\left(\frac{\bd\partial\bc H}{\bd\partial\bc Y}(\bc Z, \bc Y)\right)^{k}\right) \left(\bc H\left(\bc Z,\bc Y\right)-\bc Y+({\mathcal Z}_1-1)\bc H(\bs{0},\bs{0})\right).
\end{align*}
Since the system $\bc Y =\bc H(\bc Z,\bc Y)$ is well founded, its companion system is well founded at~$\bs{0}$ and has a solution~$\Sun({\mathcal Z}_1,\bc Z)$ that is polynomial in~${\mathcal Z}_1$. Then by induction, the sequence defined by~$\bc T^{[0]}=\bs{0}$ and $\bc T^{[n+1]}=\bc N_{\bc K}({\mathcal Z}_1,\bc Z,\bc T^{[n]})$ consists of species that are polynomial in~${\mathcal Z}_1$ and satisfies $\bc T^{[n]}(1,\bc Z)=\bc Y^{[n]}(\bc Z)$ for all $n\ge0$. 

The species~$\Sun({\mathcal Z}_1,\bs{0})$ being polynomial, $\Sun({\mathcal Z}_1,\bs{0})=\bc S(\bs{0})=\bc H^m(\bs{0},\bs{0})$ by Proposition~\ref{prop:polynomial_implicit}.
Moreover, since~$\Sun({\mathcal Z}_1,\bs{0})$ is a subspecies of~$\Sun({\mathcal Z}_1,\bc Z)$, it is a subspecies of~$\bc T^{[p]}$ as soon as its contact with~$\Sun$ is large enough. For such a value of~$p$, $\bc T^{[p]}(1,\bs{0})=\Sun(1,\bs{0})=\bc S(\bs{0})=\bc Y^{[p]}(\bs{0})$ and the conclusion of the lemma holds. In order to bound the value of~$p$, we observe that for any~$n$, $\bc K(\bc Z,\bc T^{[n]})\subset\bc T^{[n+1]}$ and conclude that $p\le m$ by Proposition~\ref{prop:polynomial_implicit}.
\end{proof}

\begin{example}\label{ex:catalan_spec}
For Catalan trees, Newton's iteration reads:
\[
{\mathcal Y}^{[n+1]}={\color{blue}{\mathcal Y}^{[n]}}+{\color{red}\Seq({\mathcal Z}\cdot\Seq({\mathcal Y}^{[n]})^2)}\cdot({\mathcal Z}\cdot\Seq({\mathcal Y}^{[n]})-{\mathcal Y}^{[n]}).
\]
The first iterates are as displayed by Figure~\ref{fig:first_catalan}.    
Rectangles enclose structures of a given size, when for this size, all the Catalan trees have been produced. This is to be compared with the successive species obtained by a simple fixed point iteration (see Example~\ref{ex:catalan}). For instance, Newton's iteration produces all the trees of size~5 at its second step, while the previous iteration does not generate all of them before the fifth step.
\end{example}

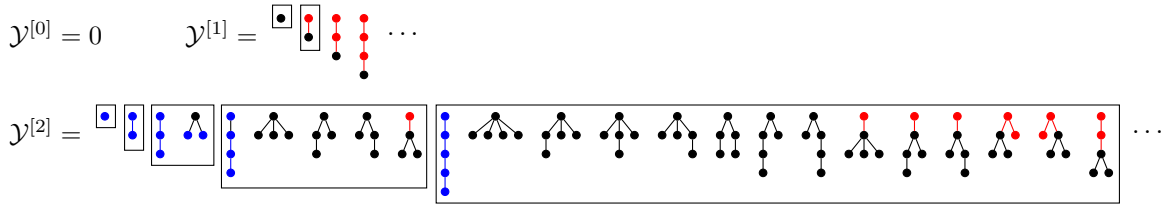
\begin{figure}[ht]
\noindent\begin{tikzpicture} 	
	\tikzstyle{level 1}=[sibling distance=2mm]        
	\node[rectangle,minimum width=2cm,minimum height=4mm,anchor=west]{};
	\matrix[mat, inner sep=2mm, row sep=0mm, column sep=2.5mm]{    
	\node[rectangle]{${\mathcal Y}^{[0]}=0$}; \&   
	\node[yshift=2mm][rond,color=white] {} 	child[color=white]{\nrp child{\nrp child{node[rond]{}}}} ;\&
	\\};           
\end{tikzpicture}  
\begin{tikzpicture} 	
	\tikzstyle{level 1}=[sibling distance=2mm]         
	\node[rectangle,minimum width=16cm,minimum height=2mm,anchor=west]{};
	\matrix[mat, inner sep=2mm, row sep=0mm, column sep=2.5mm]{
	\node[rectangle]{${\mathcal Y}^{[1]}=$}; \& 
	\node[yshift=2mm][rond]{}; \&  
	\node[yshift=2mm][rond,color=red] {} 	child[color=red]{node[rond,color=black]{}} ;\&  
	\node[yshift=2mm][rond,color=red] {} 	child[color=red]{\nrp child{node[rond,color=black]{}}} ;\& 
	\node[yshift=2mm][rond,color=red] {} 	child[color=red]{\nrp child{\nrp child{node[rond,color=black]{}}}} ;\&   
	\node{$\dots$};
	\\}; 
\node[rectangle, draw, draw opacity=1,fill opacity=0, minimum width=0.25cm, minimum height=0.3cm] at (1.48,0.4){};  
\node[rectangle,draw, draw opacity=1,fill opacity=0, minimum width=0.25cm, minimum height=0.6cm] at (1.85,0.25){};  
\end{tikzpicture}\\    
\begin{tikzpicture}  
	\tikzstyle{level 1}=[sibling distance=2mm]
	\tikzstyle{level 2}=[sibling distance=2mm] 
	\tikzstyle{level 3}=[sibling distance=2mm]
	\tikzstyle{level 4}=[sibling distance=2mm]      
	\node[rectangle,minimum width=16cm,minimum height=2mm,anchor=west]{};
	\matrix[mat,inner sep=2mm,row sep=3mm, column sep=2.5mm]{ 
	\node[rectangle]{${\mathcal Y}^{[2]}=$}; \& 
	\node[yshift=2mm][rond,color=blue]{}; \&  
	\node[yshift=2mm][rond,color=blue] {} 	child[color=blue]{\nrp} ;\&  
	\node[yshift=2mm][rond,color=blue] {} 	child[color=blue]{\nrp child{\nrp}} ;\&
	\node[yshift=2mm][rond] {}  child{node[color=blue][rond]{}} child{node[color=blue][rond]{}} ;\& 
	\node[yshift=2mm][rond,color=blue] {} 	child{node[rond][color=blue] {} 
						child[color=blue]{\nrp 
						child[color=blue]{\nrp}}} ; \&
	\node[yshift=2mm][rond] {} child{\nrp} 
				   	child{\nrp}
				   	child{\nrp} ; \&
        \node[yshift=2mm][rond] {} child{ \nrp  
        				child{\nrp}} 
				child{\nrp};\&
        \node[yshift=2mm][rond] {} child{ \nrp} 
				child{\nrp
					child{\nrp}};\& 
        \node[yshift=2mm][rond,color=red] {} child[color=red]{node[rond][color=black]{} 
		   		child[color=black]{\nrp}
		   		child[color=black]{\nrp}} ; \&  
        \node[yshift=2mm][rond,color=blue] {} 	
				child[color=blue]{\nrp 
					child{\nrp 
						child{\nrp
							child{\nrp}}}} ; \&

        \node[yshift=2mm][rond] {} child{\nrp} 
			   	child{\nrp}
				child{\nrp}
			   	child{\nrp} ; \& 
        \node[yshift=2mm][rond] {} child{\nrp 
		   			child{\nrp}}
				child{\nrp}
		   		child{\nrp} ; \& 
        \node[yshift=2mm][rond] {} child{\nrp}
     				child{\nrp
					  	child{\nrp}}
		      		child{\nrp} ; \&
        \node[yshift=2mm][rond] {} child{\nrp}
			   	child{\nrp}
			        child{\nrp
				   	child{\nrp}} ; \&  
        \node[yshift=2mm][rond] {} child{\nrp
					child{\nrp}}
				child{\nrp  	
					child{\nrp}};\&
        \node[yshift=2mm][rond] {} child{\nrp 
				child{\nrp 
					child{\nrp}}}
			        child{\nrp} ;\&
        \node[yshift=2mm][rond] {} child{\nrp}
        			child{\nrp 
        				child{\nrp 
        					child{\nrp}}} ;\&  
        \node[yshift=2mm][rond,color=red] {} 	child[color=red]{node[rond][color=black] {}  
						child[color=black]{\nrp}  
						child[color=black]{\nrp} 	
						child[color=black]{\nrp}} ; \&  
        \node[yshift=2mm][rond,color=red] {} 	child[color=red]{node[rond][color=black] {} 
						child[color=black]{\nrp 
							child[color=black]{\nrp}}
							child[color=black]{\nrp}} ;\&  
        \node[yshift=2mm][rond,color=red] {} 	child[color=red]{node[rond][color=black] {}					
						child[color=black]{\nrp} 
						child[color=black]{\nrp 
							child[color=black]{\nrp}}} ;\& 
        \node[yshift=2mm][rond,color=red] {} 	child[color=red]{node[color=black][rond] {}  
							child[color=black]{\nrp} 
							child[color=black]{\nrp}} 
						child[color=red]{\nrp};\& 
        \node[yshift=2mm][rond,color=red] {} child[color=red]{\nrp} 
					        child[color=red]{ node[rond][color=black]{} 
					 		child[color=black]{\nrp}      	
					  		child[color=black]{\nrp}};\& 
        \node[yshift=2mm][rond,color=red] {} child[color=red]{\nrp 
					        child[color=red]{ node[rond][color=black]{}
						child[color=black]{\nrp} 
						child[color=black]{\nrp}}} ; \&  
        \node[inner sep=0mm]{$\dots$};  
	\\};    
\node[rectangle,draw, draw opacity=1,fill opacity=0, minimum width=0.25cm, minimum height=0.3cm] at (1.48,0.5){};   
\node[rectangle,draw, draw opacity=1,fill opacity=0, minimum width=0.25cm, minimum height=0.6cm] at (1.85,0.35){};
\node[rectangle,draw, draw opacity=1,fill opacity=0, minimum width=0.82cm, minimum height=0.8cm] at (2.49,0.25){};  
\node[rectangle,draw, draw opacity=1,fill opacity=0, minimum width=2.7cm, minimum height=1.1cm] at (4.35,0.1){};  
\node[rectangle,draw, draw opacity=1,fill opacity=0, minimum width=8.99cm, minimum height=1.3cm] at (10.32,0){};
\end{tikzpicture}  
\caption{First iterates of Newton's iteration on Catalan trees\label{fig:first_catalan}}
\end{figure}

\begin{example} For series-parallel graphs (see Example~\ref{ex:spg}), we give Newton's iteration only for the part concerning ${\mathcal S}:={\cal Y}_1$ and~${\mathcal P}$, the graphs themselves are obtained by adding~${\mathcal Z}:={\cal Y}_2$ to those iterates.
\[
\begin{pmatrix}{\mathcal S}^{[n+1]}\\ {\mathcal P}^{[n+1]}\end{pmatrix}
=
\begin{pmatrix}{\color{blue}{\mathcal S}^{[n]}}\\{\color{blue} {\mathcal P}^{[n]}}\end{pmatrix}	
	+  \left({\color{red}\sum_{k\ge0}{\begin{pmatrix}0&\Seq^2({\mathcal Z}+{\mathcal P}^{[n]})-1\\ \Set_{\ge 1}({\mathcal Z}+{\mathcal S}^{[n]})& 0\end{pmatrix}^k}}\right)
	\begin{pmatrix}\Seq_{\ge 2}({\mathcal Z}+{\mathcal P}^{[n]})-{\mathcal S}^{[n]}\\ \Set_{\ge 2}({\mathcal Z}+{\mathcal S}^{[n]})-{\mathcal P}^{[n]}\end{pmatrix}.
\]
\begin{figure}[ht]
\centerline{\includegraphics[width=\textwidth]{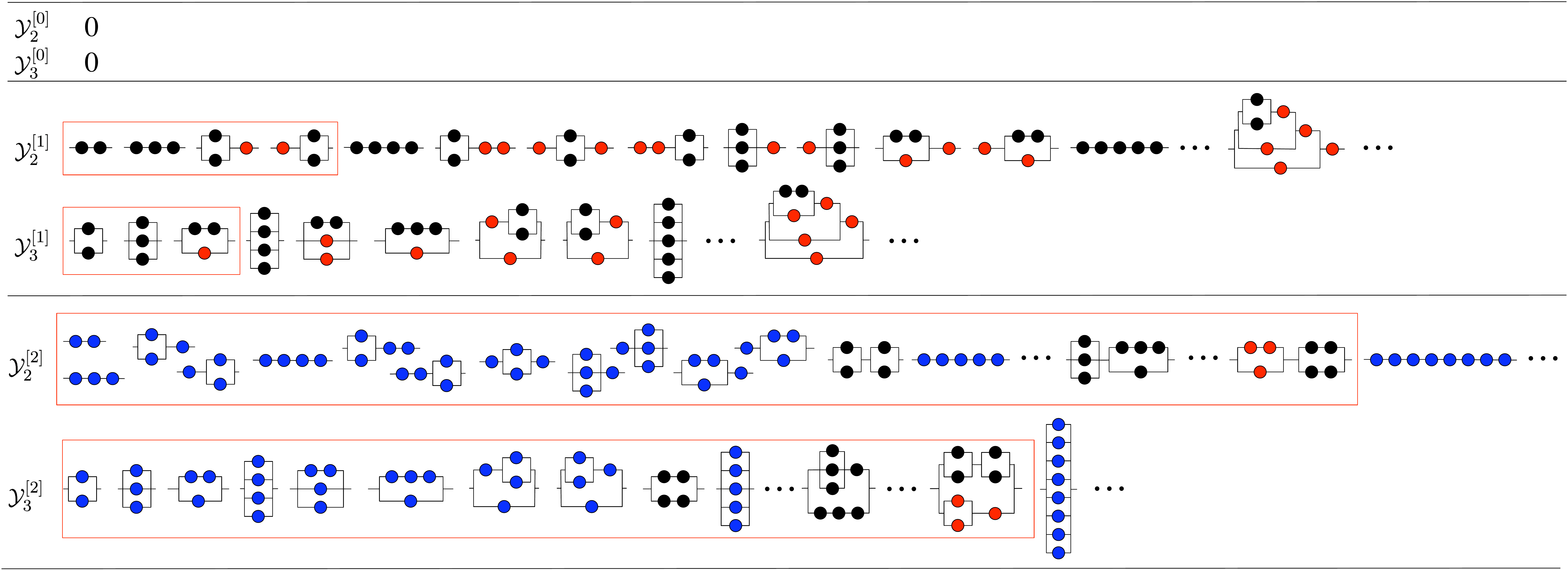}}
\caption{First iterates of Newton's iteration on series-parallel graphs\label{fig:first_sp}}
\end{figure}
The first few iterates are given in Figure~\ref{fig:first_sp}.
\end{example}

\subsection{Effective Computation}\label{sec:newton_operator}

Newton's iteration, as stated in~Theorem~\ref{th:newton} gives an infinite sequence that converges to the solution of the system. Moreover, Newton's operator also  involves an infinite sum.  
In order to  effectively compute structures by  Newton's iteration, we give an alternative to~Theorem~\ref{th:newton} that involves truncated species. 

These truncated iterations are lifted to iterations over power series in Section~\ref{sec:series}, so as to get efficient algorithms. 

\begin{theorem}[Truncated Newton iteration]\label{coro:newton_trunc}
Let $\bc Y=\bc H(\bc Z,\bc Y)$, be a well-founded system and $\bc S$ its solution from Theorem~\ref{th:GIST}.
Let  $(\bc Y^{[n]})_{n\ge 0}$ be the sequence defined by
\[
 \bc Y^{[0]}=\bc S(\bs{0})\quad \mbox{and} \quad  \bc Y^{[n+1]}=\left(\bc N_{\bc H}(\bc Z,\bc Y^{[n]})\right)_{\leq 2^{n+1}-1}.
\]
Then, for all $n\ge 0$, $\bc Y^{[n]}=_{2^{n}-1} \bc S$.
\end{theorem}
Recall that~$\bc S(\bs{0})=\bc S_{=0}$ itself is~$\bs{0}$ when the system is well founded at~0 and can be computed by iterating~$\bc H(\bs{0},\bc Y)$ from~$\bs{0}$ at most $m$ times otherwise; in particular $\bc S(\bs{0})=\bc{H}^m(\bs{0},\bs{0})$ (see Lemma~\ref{lem:ini-contact}).
\begin{proof}
  For any~$k\geq0$, one has the inclusion $\bc S_{\leq k}\subseteq \bc H(\bc Z,\bc S_{\leq k})$, since any $\bc S$-structure of size at most~$k$ is an $\bc H$-assembly of $\bc S$-structures of size at most~$k$.
  Then, Lemma~\ref{lem:quad-cvg} applies with  $k=2^n-1$ and by induction~$\,\bc U=\bc Y^{[n]}=\bc S_{\leq k}$.
  \end{proof}

\paragraph{Computation of Newton's operator} Newton's operator involves an infinite sum that we also compute using  Newton's iteration.
The inverse of a matrix~$A$ can be computed by Newton's iteration:
${B^{[n+1]}=B^{[n]}-B^{[n]}(AB^{[n]}-\operatorname{Id})}$ converges quadratically to the inverse of~$A$.
This idea goes back at least to Schulz~\cite{Schulz33}. 
This iteration also applies to matrices of species since all the operations involved have combinatorial interpretations. Thus, the inverse of the matrix $\Id-\bjac{n}$ occurring in Newton's operator is obtained by:
\begin{equation}\label{eq:U}
\bc U^{[i+1]}=\bc U^{[i]}+\bc U^{[i]}\cdot\left(\bjac{n}\cdot\bc U^{[i]}-(\bc U^{[i]}-\Id)\right).
\end{equation}
The proof of quadratic convergence is a simpler variant of the previous one.

\subsection{Optimized Newton Iteration}\label{subsec:optimized_Newton}
We go one step further and improve the efficiency of the algorithm by lifting to the combinatorial setting a technique that saves a constant factor in the speed of Newton's iteration. The idea is to avoid spending too much time in the computation of the inverse of the Jacobian matrix. The optimization consists in computing only one iteration of the inverse by Eq.~\eqref{eq:U} at each step of the main Newton iteration. This optimized version of Newton's iteration is briefly  explained below, and can also  be found in~\cite{Pivoteau08}.

\begin{proposition}\label{prop:newt_opt}
	Let $\bc Y=\bc H(\bc{Z},\bc Y)$ be a well-founded system. The sequence~$(\bc Y^{[n]})_{n\ge 0}$ defined by 
\begin{align}
    \bc U^{[n+1]}&=\bc U^{[n]}+\bc U^{[n]}\cdot\left(\jac{n}\cdot\bc U^{[n]}+\Id-\bc U^{[n]}\right)_{\le 2^n}, \label{eq:newt_U*}\\
        \bc Y^{[n+1]}&= \bc Y^{[n]}+\bc U^{[n+1]}\cdot\left(\bc H(\bc{Z}, \bc Y^{[n]})-\bc Y^{[n]}\right)_{\le 2^{n+1}-1}, \label{eq:newt_Y*}
\end{align}
with $\bc Y^{[0]}=\bd{0}$ and $\,\bc U^{[0]}=\Id$, converges quadratically to the species~$\bc S$, solution of~$\bc Y=\bc H(\bc{Z},\bc Y)$. 
\end{proposition}

\begin{proof}[Main steps of the proof]
A complete proof is given in~\cite{Pivoteau08}. We only give here a sketch of it. 
The first step is to show that this iteration is well defined, i.e., all subtraction signs correspond to inclusions. This is done by an induction on~$n$ showing that $\,\bc U^{[n]}\subset \Id + \jac{n}\cdot\bc U^{[n]}$ and~$\bc Y^{[n]}\subset \bc H(\bc{Z},\bc Y^{[n]})$. 
The second step is to ensure that there is no ambiguity on the $\,\bc U^{[n+1]}$-structures. This comes from an induction showing that any sequence of structures of the species $\jac{n}-\jac{n-1}$ is uniquely derived from Equation~\eqref{eq:newt_U*}. 
Finally, quadratic convergence is given by an induction proving that  $\bc Y^{[n]} =_k \bc Y^{[n+1]}$ and $\,\bc U^{[n]} =_{\lfloor k/2\rfloor} \bc U^{[n+1]}$ are sufficient to ensure that $\bc Y^{[n+1]} =_{2k+1} \bc Y^{[n+2]}$ and $\,\bc U^{[n+1]} =_{k}\bc U^{[n+2]}$.\qedhere
\end{proof}

\subsection{Characterization of the Iterates}\label{sec:Strahler}   
In the case of the simple iteration of Section~\ref{seqspecies}, the ${\mathcal H}$-rooted trees are produced by increasing height. We now give a simple combinatorial interpretation of Newton's iteration using a generalization for ${\mathcal H}$-rooted trees of the classical Strahler number (see, e.g.,~\cite{Viennot1990}). 
This is not used in the sequel.

\begin{definition}The \emph{Strahler number} of a nonempty ${\mathcal H}$-rooted tree $\gamma$, denoted by $\Sh(\gamma)$, is defined recursively. The structure $\gamma$ decomposes as an ${\mathcal H}$-assembly with members  $\gamma_{1:q}$. If all the members of $\gamma$ are singletons, i.e., $\gamma_1=\dots=\gamma_q={\mathcal Z}$, then $\Sh(\gamma)=1$. Otherwise, let $M=\max_{1\le i \le q}(\Sh(\gamma_i))$ be the maximum Strahler number of one of $\gamma$'s members, then:
\[
\Sh(\gamma)=
\begin{cases}
  	M+1&\mbox{if $\exists i,j$, such that $i\neq j$ and  $\Sh(\gamma_i)=\Sh(\gamma_j)=M$,}\\	
  	M&\mbox{otherwise.}\\
\end{cases}
\]  
\end{definition}
\begin{figure}\begin{center}
 \includegraphics[width=10cm]{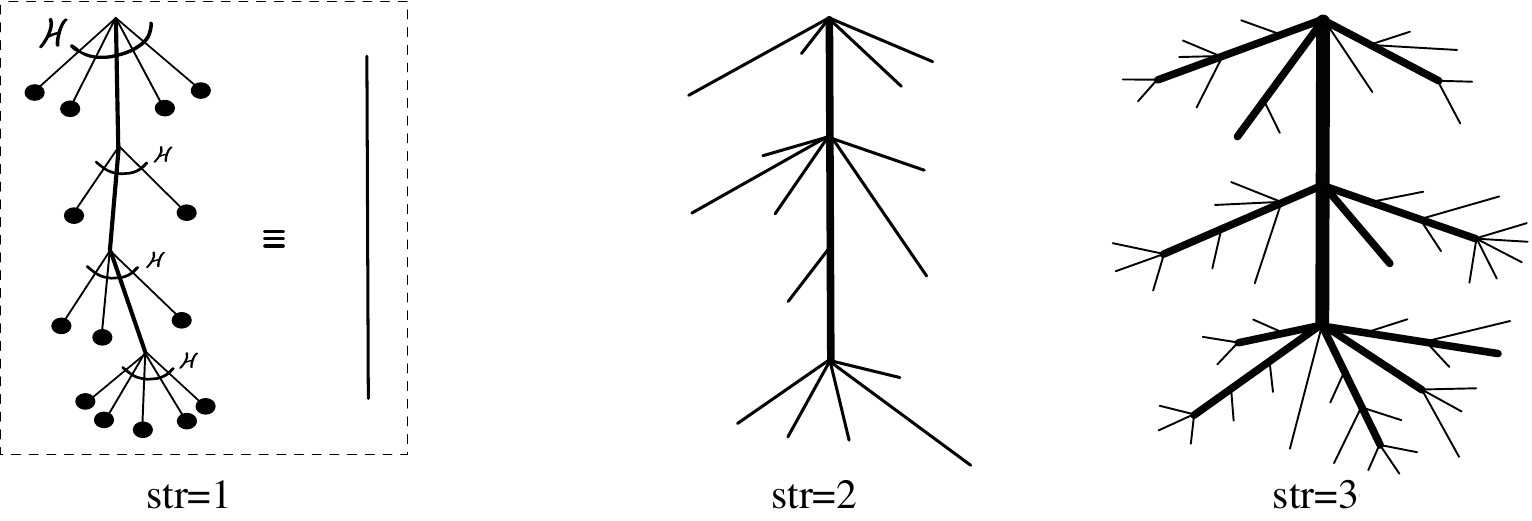}	
\end{center}  
\caption{Typical ${\mathcal H}$-rooted trees with Strahler number 1,2 or 3. In the second and third pictures, trees having Strahler number equal to 1 are depicted by simple lines, as suggested in the first picture.} 
\end{figure}

\begin{proposition}Let~${\mathcal H}$ be a multisort species such that ${\mathcal Y} = {\mathcal H} ({\mathcal Z}, {\mathcal Y})$ is well founded.
Let $({\mathcal Y}^{[{k}]})_{{k}\in\N}$ be the sequence of species produced by Newton's iteration starting from~${\mathcal Y}^{[0]}=0$. Then, for any ${\mathcal H}$-rooted tree $\gamma$:
\[
\gamma \in {\mathcal Y}_i^{[{k}]}-{\mathcal Y}_i^{[{k}-1]}\Longleftrightarrow\Sh(\gamma)={k}.
\]	
\end{proposition}
\begin{proof}
By induction on ${k}$. The only ${\mathcal H}$-rooted trees with Strahler number equal to~1 are the~${\mathcal H}$-assemblies whose members are atoms, that is the ${\mathcal Y}^{[1]}$-structures. Thus $\gamma \in {\mathcal Y}^{[1]}\Leftrightarrow \Sh(\gamma)= 1$.

        If $\gamma\in {\mathcal Y}^{[{k}+1]}-{\mathcal Y}^{[{k}]}$, then~$\gamma$ rewrites as a sequence~ $\beta_1\dotsm\beta_\ell\cdot\delta$, with $\beta_j\in  \partial {\mathcal H}/\partial {\mathcal Y}({\mathcal Z}, {\mathcal Y}^{[{k}]})$, $j=1,\dots,\ell$ and $\delta \in {\mathcal H}({\mathcal Z}, {\mathcal Y}^{[{k}]})-{\mathcal Y}^{[{k}]}$. The structure $\delta$ is exclusively composed of ${\mathcal Y}^{[{k}]}$-structures; by the induction hypothesis, their Strahler number is at most~${k}$;  thus, $\Sh(\delta)\leq {k}+1$. Then $\beta_\ell\cdot\delta$ is an  ${\mathcal H}$-assembly with at most one member (the structure~$\delta$) whose Strahler number is~$\le {k}+1$ and the other ones (that are ${\mathcal Y}^{[{k}]}$-structures) whose Strahler number is~$\le {k}$. Therefore,~$\Sh((\beta_\ell\cdot\delta))\leq {k}+1$. Iterating this reasoning on the $\beta_j$'s until $j=1$, we get $\Sh(\gamma_i)\leq{k}+1$. Since $\gamma\notin {\mathcal Y}^{[{k}]}$, the induction hypothesis implies that $\Sh(\gamma_i)> {k}$, for $i=1,\dots,m$, which gives $\Sh(\gamma)={k}+1$.
 
        If $\Sh(\gamma)= {k}+1$, then $\gamma$ is an ${\mathcal H}$-assembly whose members have Strahler number at most~${k}$ and possibly one equal to~${k}+1$. If none of them has Strahler number equal to~${k}+1$, then $\gamma$ is an ${\mathcal H}({\mathcal Z}, {\mathcal Y}^{[{k}]})$-structure. Otherwise, $\gamma$ is a structure of the form $\beta\cdot\delta$ with $\beta \in \partial {\mathcal H}/ \partial {\mathcal Y}({\mathcal Z}, {\mathcal Y}^{[{k}]})$ and $\Sh(\delta)={k}+1$. Iterating this until $\gamma$ is exhausted, we obtain that $\gamma\in {\mathcal Y}^{[{k}+1]}$. Since $\Sh(\gamma)>{k}$, the structure~$\gamma$ does not belong to ${\mathcal Y}^{[{k}]}$,  and thus $\gamma\in {\mathcal Y}^{[{k}+1]}-{\mathcal Y}^{[{k}]}$. 
\end{proof}         

This result can be extended to any well-founded system~$\bc Y = \bc H (\bc Z, \bc Y)$, bearing in mind that  $\bc H$-rooted trees can be viewed as derivation trees of the combinatorial structures described by such a system.

\section{Special Classes of Species}\label{sec:special}
\subsection{Constructible Species}\label{subsec:constructible}

Even if the combinatorial properties stated in this part are applicable to all species, we introduce a restriction on species in order to give complexity results in Section~\ref{sec:series}. The restriction we impose is guided by the framework of constructible combinatorial classes~\cite{FlSe09}, which we cast here into the species language.

\begin{definition}\label{def:constructible} 
        A \emph{constructible species} is inductively defined as either:
        \begin{enumerate}                \item one of the \emph{basic species} (see~Table~\ref{tab:sum_esp_sg}) in $\{ 1, {\mathcal Z}, +, \cdot, \Seq, \Cyc, \Set, {\mathcal Y}_1,{\mathcal Y}_2,\dotsc \}$;
                \item any basic species with a cardinality constraint that is a finite union of intervals;
                \item a composition of constructible species;
                \item the solution of a well-founded system~$\bc Y=\bc H({\mathcal Z},\bc Y)$, such that each coordinate of $\bc H$ is constructible.
        \end{enumerate}
		We  call \emph{iterative constructible}, the species defined by the first three rules only (i.e., without recursivity).      
\end{definition}
\begin{example}\label{ex:constructible} All the examples in this article are constructible.
\end{example}

\subsection{Flat Species}\label{subsec:flat}
Flat species will prove useful in the next part, thanks to the nice properties of their generating series.

An ${\mathcal F}$-structure $s$ on $U$ is \emph{asymmetric} if it does not have any internal symmetry. More precisely, for any permutation $\sigma$ of $U$ different from the identity, ${\mathcal F}[\sigma](s)\neq s$.
\begin{definition} Let ${\mathcal F}$ be a species of structures. The \emph{flat part} of ${\mathcal F}$ is the subspecies $\overline{{\mathcal F}}$ of ${\mathcal F}$ defined, for any finite set $U$, by $\overline{{\mathcal F}}=\{s \in {\mathcal F}[U] \mid s \mbox{ is asymmetric} \}$.        
\end{definition}
This extends to multisort species, asymmetry being with respect to each sort.

\begin{example} The most basic cases are: $\overline{\Seq}=\Seq$, $\overline{\Cyc}={\mathcal Z}$, $\overline{\Set}=1+{\mathcal Z}$.
\end{example}

A \emph{flat} species is a species that is isomorphic to its flat part. For instance, the species of sequences is flat, whereas cycles and sets are not.

\begin{lemma}\label{lemma:composition_flat}   Let ${\bc F}$ be a $k$-sort flat species and ${\bc G}={\cal G}_{1:k}$ a flat species, then the species ${\bc F} \circ {\bc G}$ is flat.
\end{lemma}
\begin{proof}         According to~\cite[Prop.~4, p.~323]{BeLaLe98} and~\cite{Labelle92}, for any species ${\bc F}$ and ${\bc G}$ (not necessarily flat), ${\overline{{\bc F}}\circ \overline{{\bc G}} \subset \overline{{\bc F}\circ {\bc G}}}$. Thus, the flatness of~${\bc F}$ and ${\bc G}$ gives the following relations 
        \[
        {\bc F} \circ {\bc G}=\overline{{\bc F}}\circ \overline{{\bc G}} \subset \overline{{\bc F}\circ {\bc G}} \subset {\bc F}\circ {\bc G}
        \]
        that lead to ${\bc F} \circ {\bc G}=\overline{{\bc F}\circ {\bc G}}$.
\end{proof}

\begin{proposition}[Implicit Flat Species]\label{prop:flat}         
	Let~$\bc F({\bc Z},\bc Y)$ be a flat species,  if $\bc Y=\bc F({\bc Z},\bc Y)$ is well founded, then its solution~$\bc S$ given by Theorem~\ref{th:GIST} is flat.
\end{proposition}
\begin{proof}
     By induction on the sequence $(\bc Y{}^{[n]})_{n\ge 0}$ defined by $\bc Y{}^{[n+1]}=\bc F({\bc Z},\bc Y{}^{[n]})$, using Lemma~\ref{lemma:composition_flat}.
\end{proof}

\part{Computation}\label{part:computation}
We now turn to the computational part of our work and show that the combinatorial iterations transfer  to  both  levels of generating series and numerical evaluation.
Newton's  combinatorial iteration directly transfers to the level of formal power series, the quadratic convergence being in terms of valuation. The strength of Newton’s iteration in this context is that it makes it possible to compute the first $n$ terms of the generating series in $O(n\log n)$ arithmetic operations. Taking into account the combinatorial origin of the coefficients, we also show that the bit complexity is likewise quasi-optimal. 
	When interpreted numerically, for a value of the variable inside the disk of convergence of the generating series, we show that the same iteration computes the values of the power series,  in both cases of ordinary and exponential generating series.

\section{Formal Power Series}\label{sec:series}      
In this section, we transfer Newton's iteration on a combinatorial system to iterations over generating series. This extends earlier work of Labelle's~\cite{Labelle86} to systems and to equations with~1. We also discuss several possible truncation orders.
The quadratic rate of convergence of Newton's iteration leads to efficient enumeration algorithms, whose complexity we show to be quasi-optimal.

\subsection{Generating Series}\label{sec:genseries}
There are several formal power series associated to a species~${\mathcal F}$. The exponential generating series~$F(z)$ encodes the numbers of ${\mathcal F}$-structures on the sets~$\{1,\dots,n\}$. This is called \emph{labeled} enumeration. The ordinary generating series, denoted by~$\tilde{F}(z)$, is used for \emph{unlabeled} enumeration; it encodes the numbers of isomorphism classes of ${\mathcal F}$-structures. Finally, a third kind of series, the cycle index series~$Z_{{\mathcal F}}$ is a more general tool that gathers the information of both exponential and ordinary generating series.

\subsubsection{Definitions}

\begin{definition}
        The \emph{exponential generating series} of a species of structures ${\mathcal F}$ is the formal power series
        \[
        F(z)=\sum_{n=0}^\infty f_n \frac{z^n}{n!}
        \]
where $f_n = |{\mathcal F}[\{1,\dots,n\}]|$ is the number of ${\mathcal F}$-structures on a set of size~$n$ (also called \emph{labeled} ${\mathcal F}$-structures of size~$n$).
\end{definition}

\begin{definition}\label{def:ogf}
        The \emph{ordinary generating series} of a species of structures ${\mathcal F}$ is the formal power series
        \[
        \tilde{F}(z)=\sum_{n=0}^\infty \tilde{f}_n z^n
        \]
where $\tilde{f}_n$ is the number of unlabeled ${\mathcal F}$-structures of size $n$.
\end{definition}           
In the case of asymmetric structures, the number of labeled structures of size $n$ coincides with the number of unlabeled structures of size $n$ multiplied by $n!$, so that for flat species (see Section~\ref{subsec:flat}), exponential and ordinary generating series coincide.

\begin{definition}
        The \emph{cycle index series} of a species of structures ${\mathcal F}$ is a formal power series in an infinite number of variables, defined by:
\begin{equation}\label{eq:cycle-index}
Z_{{\mathcal F}}(z_1,z_2,z_3,\dotsc) = \sum_{n\ge0}{\frac{1}{n!}\left(\sum_{\sigma\in{\mathcal P}_n}{\operatorname{fix}{\mathcal F}[\sigma]z_1^{\sigma_1}z_2^{\sigma_2}\dotsm}\right)},
\end{equation}
where $\sigma_i$ is the number of cycles of length $i$ in the cycle decomposition of the permutation $\sigma$  and $\operatorname{fix}{\mathcal F}[\sigma]$ is the number of ${\mathcal F}$-structures on $\{1,\dots,n\}$ fixed by ${\mathcal F}[\sigma]$.
\end{definition}

\begin{table}
\begin{small}
\begin{center}
$\begin{array}{|@{}l@{}|l|l|l@{}|l|l|l|l|l|l|l|l|}
\hline                                           
&0&1&{\mathcal Z}& {\mathcal F}+{\mathcal G}&  {\mathcal F}\cdot{\mathcal G}& \Seq&\Set\phantom{\bigl|}&\Cyc\\  
\hline
\begin{array}{l} \mbox{cycle}\\ \mbox{index}\\ \mbox{series}\end{array} &0&1&z_1\phantom{\Biggl|}&Z_{{\mathcal F}}+Z_{{\mathcal G}}&Z_{{\mathcal F}}\cdot Z_{{\mathcal G}} &\ds\frac{1}{1-z_1} &\ds\exp\left(\sum_{k>0}\frac{z_k}k\right)\!\!&\ds\sum_{k>0}\frac{\varphi(k)}k\log\frac{1}{1-z_k}\\
\hline
\end{array}$  
\end{center}    
\end{small}
\caption{Cycle index series.\label{tab:cycle_index} ($\varphi$ is Euler's totient function)}
\end{table}

Table~\ref{tab:cycle_index} presents the cycle index series of the basic constructible species.
The following fundamental result relates these three generating series.
\begin{property}{\cite[Th.~8 p.~18 and (33) p.~112]{BeLaLe98}}\label{prop:ogfegf}
A species of structures ${\mathcal F}$ has exponential generating series $Z_{{\mathcal F}}(z,0,0,\dotsc)$ and ordinary generating series  $Z_{{\mathcal F}}(z,z^2,z^3,\dotsc)$.
\end{property}

Table \ref{tab:sum_esp_sg} presents ordinary and exponential generating series associated with basic constructible species (see~\cite{FlSe09} for more details.)

\begin{table}[ht]
\begin{small} 
\begin{center}
$ \begin{array}{|l|l|@{}c@{}|@{}c@{}|l@{}|l|} 
\hline
   	\mbox{species} & \mbox{operator} & ~{\mathcal G} =0~ &  ~{\mathcal G} = 1~ &  \mbox{exponential g.s.}  & \mbox{ordinary generating series}\\[2pt]  
\hline  
\multicolumn{6}{|c@{}|}{\phantom{\Bigl|}\mbox{basic constructible species}}\\
\hline
   	\mbox{Disjoint union} & {\mathcal F}+ {\mathcal G} & {\mathcal F} & {\mathcal F}+ 1 & F(z)+G(z) & F(z)+G(z)\\[2pt]
	\mbox{Cartesian product} & {\mathcal F}\cdot {\mathcal G} & 0 & {\mathcal F} & F(z)\cdot G(z) & F(z)\cdot G(z)\\[2pt]    
	\mbox{Sequence} &  \Seq({\mathcal G}) & 1 & - & (1-G(z))^{-1} & (1-G(z))^{-1} \\[2pt]
\hline       
	\mbox{Cycle} & \Cyc({\mathcal G}) & 0 & - & \ds\log\frac1{1-G(z)} & \ds\sum_{k>0}\frac{\varphi(k)}k\log\frac{1}{1-G(z^k)}\\[2pt]         	        
	\mbox{Set}  & \Set({\mathcal G}) & 1 & -  & \exp(G(z)) & \ds\exp\left(\sum_{k>0}\frac{G(z^k)}k\right)\\[2pt]
\hline
\multicolumn{6}{|c@{}|}{\phantom{\Bigl|}\mbox{basic constructible species with a fixed cardinality}}\\
\hline
\mbox{$\ell$-tuple, $\ell>0$} &  \Seq_\ell({\mathcal G}) & 0 & 1 & G^\ell(z) & G^\ell(z)\\[2pt]
	\mbox{Cycle, card $\ell>0$}\!\!& \Cyc_\ell({\mathcal G}) & 0 & 1 &  \frac{1}{\ell} G^\ell(z) & \ds [u^\ell] \sum_{k= 1}^{\infty}\frac{\varphi(k)}{k}\log\frac{1}{1-u^kG(z^k)}\\[2pt] 
	\mbox{Set of card $\ell>0$}  & \Set_\ell({\mathcal G}) & 0 & 1 & \frac{1}{\ell!}G^\ell(z) & [u^\ell]\exp\left(uG(z)+\frac{u^2}{2}G(z^2)+\frac{u^3}{3}G(z^3)+\dotsb\right)\\[2pt]
\hline
\end{array}
$ 
\end{center}
\end{small}
\caption{Generating series associated with basic constructible species. {\footnotesize (The character ``$-$'' stands for ``undefined''; other cardinality constraints are obtained by subtraction ($\ge \ell$) or finite unions ($\le \ell$).)} \label{tab:sum_esp_sg}	
}
\end{table}

\subsubsection{Derivative} The cycle index series of the derivative of a species ${{\mathcal F}}$ is given by
\[Z_{{\mathcal F}'}(z_1,z_2,z_3,\dotsc)= \frac{\partial }{\partial z_1} Z_{{\mathcal F}}(z_1,z_2,z_3,\dotsc).
\]
This leads to a simple formula for exponential generating  series: $F'(z) = dF(z)/dz$. But there is no such relation  for  ordinary generating series: the computation of the ordinary generating series  for a derivative species goes by the cycle index series of the derivative: 
\begin{equation}\label{Fprime}
\widetilde{F^{\prime}}(z)=  Z_{{\mathcal F'}}(z,z^2,z^3,\dotsc).
\end{equation}
In a way, this is why Newton's iteration for unlabeled enumeration is complicated: \emph{the sole knowledge of the system of equations over ordinary generating series is not sufficient to derive the Newton operator, which requires the derivative.} (See the example of unlabelled Cayley trees in the introduction.)
\subsubsection{Composition}\label{subsub:composition}
The next fundamental result in the theory of species allows to compute cycle index series by composition.
\begin{property}{\cite[Th. 2, p. 43]{BeLaLe98}}\label{prop:plethysm}
Let ${\mathcal F}$ and ${\mathcal G}$ be two species of structures and assume that ${\mathcal G}(0)=0$. The cycle index series of the species ${\mathcal F}\circ{\mathcal G}$ is 
\[
Z_{{\mathcal F}\circ{\mathcal G}}(z_1,z_2,z_3,\dotsc)=
Z_{{\mathcal F}}(Z_{{\mathcal G}}(z_1,z_2,z_3,\dotsc),
Z_{{\mathcal G}}(z_2,z_4,z_6,\dotsc), Z_{{\mathcal G}}(z_3,z_6,z_9,\dotsc),\dotsc).
\] 
This formula also holds when ${\mathcal F}$ is a polynomial species and~${\mathcal G}=1$.
This operation is sometimes called the \emph{plethystic substitution} of $Z_{\mathcal G}$ in $Z_{\mathcal F}$. \end{property}

In view of Property~\ref{prop:ogfegf}, we deduce that for exponential generating series, the composition of species translates into the composition of series.

In the case of ordinary generating series, the composition is more intricate:  the ordinary generating series of the composition of species is  defined using operators that were studied by Pólya~\cite{PoRe87}.
Properties~\ref{prop:ogfegf} and \ref{prop:plethysm} thus lead us to the following.
\begin{definition}\label{coro:ogf_comp} The \emph{Pólya operator} ${\Phi}_{{\mathcal F}}$ of a species ${\mathcal F}$ is defined by
\begin{equation*}
	{\Phi}_{{\mathcal F}}:\tilde{{G}}(z)\mapsto Z_{{{\mathcal F}}}(\tilde{{{G}}}(z),\tilde{{{G}}}(z^2),\tilde{{{G}}}(z^3),\dotsc) = Z_{{\mathcal F}\circ{\mathcal G}}(z,z^2,z^3,\dotsc)
\end{equation*}
\end{definition}
It associates the ordinary generating series of~${\mathcal F}\circ{\mathcal G}$ to that of~${\mathcal G}$.
As a special case, the ordinary generating series of a species ${\mathcal F}$ is given by $\tilde{{F}}(z)={\Phi}_{{\mathcal F}}(z)$. Moreover, these operators compose by
${\Phi}_{{\mathcal F}\circ{\mathcal G}}={\Phi}_{{\mathcal F}}\circ{\Phi}_{{\mathcal G}}$.

As implicit species are built by composition, the previous results ensure that systems defining implicit species translate into systems of functional equations on ordinary or exponential generating series.

\begin{example}\label{ex:series_system}  
The ordinary and exponential generating series~$T(z)$ of Catalan trees (see Example~\ref{ex:catalan})
are both defined by the functional equation~$T(z)=z/(1-T(z))$ since the isomorphism type of this species is the species itself. 
The species of Cayley trees is defined by  ${\mathcal G} ={\mathcal Z} \cdot \Set({\mathcal G})$.
Its exponential generating series is defined by $G(z)=z\exp(G(z))$ whereas its ordinary generating series is defined by  $\tilde{G}(z)=z\exp\left(\tilde{G}(z)+\frac12\tilde{G}(z^2)+\dotsb\right)$, using the $\Set$ operator given in Table~\ref{tab:sum_esp_sg}.
\end{example}

\subsubsection{Generating Series for Multisort Species}
The definitions of ordinary, exponential and cycle index series extend to multisort species. One variable is used per sort, or one infinity of variables in the case of cycle index series. There is no technical difficulty but the notation becomes messier. The ordinary and exponential series are still deduced by simple specializations of the cycle index series. The plethystic substitution extends to the multisort context, which leads to the definition of Pólya operators.
We refer to~\cite[p.~106-107]{BeLaLe98} for details. From now on, we use exponential and ordinary generating series, cycle index series and Pólya operators both in the unisort or multisort context.

\subsection{Convergence of Iterations on Power Series}\label{section:Newton-series} 
The iterations on species of the previous part translate into iterations on generating series; the resulting iterates then converge to the expected generating series since they are the generating series of the successive species produced by the combinatorial iteration.

Recall that the \emph{valuation} of a power series~$S(z)$, denoted by $\operatorname{val}(S(z))$, is the exponent of the first nonzero coefficient of the series. A metric is classically deduced by defining 
the \emph{distance} between two power series by $d(F(z),G(z))=2^{-\operatorname{val}(F(z)-G(z))}$; the notion of convergence follows.

\begin{lemma}[Transfer, series part]\label{lem:trans1}  
	Let $\bc Y=\bc H({\mathcal Z},\bc Y)$ be well founded and~$\bc F$ be a vector of species with exponential generating series (resp. Pólya operator)~$\bd{F}$.
	If 
	\[ 
	\bc Y^{[n+1]} = \bc F({\mathcal Z},\bc Y^{[n]}), \quad \mbox{ with } \bc Y^{[0]}=\bd{0},
	\] 
	is an increasing sequence of species converging to the solution~$\bc S$ of the system, then the sequence
	\[
	\bd{Y}^{[n+1]}(z)= \bd{F}(z, \bd{Y}^{[n]}(z)), \quad \mbox{ with } \bd{Y}^{[0]}(z)=\bd{0},
	\]
    converges to the vector of exponential (resp. ordinary) generating series~$\bd{S}(z)$ (resp.~${\bd{\tilde S}}(z)$) of the species~$\bc S$.		
\end{lemma}     
\begin{proof}
Convergence of sequences of species translates in terms of valuation at the level of generating series into  $\operatorname{val}\left(\bs Y^{[n]}(z)-\bs S(z)\right)\rightarrow \infty$, which gives the convergence of generating series. 	
\end{proof}

This lemma gives a simple algorithm to compute the first coefficients of generating series.   
Indeed, the fixed point iteration induced by any well-founded combinatorial system (choosing $\bc F=\bc H$ in the previous lemma) is an automatic process to derive its associated counting series. 
But, as we have seen earlier, this iteration may be slow, meaning that in the worst case, several steps are needed to get one more correct coefficient. We now consider the faster convergence provided by Newton's iteration.
\def\NH{\bd{N}\!_{\bc H}}  
\def\TNH{{\bd{\Phi_{{\mathcal N}_{\mathcal H}}}}}  

\begin{definition}
Let $\bc Y=\bc H({\mathcal Z},\bc Y)$ be a well-founded system.
The \emph{Newton operator for exponential generating series} is defined by 
\begin{equation} 
\NH(z,\bd{Y}(z))={\bd Y}(z)+\left(\Id-{\bd\partial\bd{H}}/{\bd\partial\bd{Y}}(z,\bd{Y}(z))\right)^{-1}\cdot \left(\bd{H}(z,\bd{Y}(z))-\bd{Y}(z)\right),
\end{equation} 
where $\bd{H}$ is the vector of exponential generating series of~$\bc H$.

The \emph{Newton operator for ordinary generating series} is defined by
\begin{equation} 
{\TNH}(z,\bd{Y}(z))={\bd Y}(z)+\left(\Id-\bs{\Phi}_{\bd\partial\bc H/\bd\partial\bc Y}(z,\bd{Y}(z))\right)^{-1}\cdot \left(\bs{\Phi}_{\bc H}(z,\bd{Y}(z))-\bd{Y}(z)\right).
\end{equation} 
\end{definition}
This last operator is one of the fundamental reasons why we resort to species theory. In the case of exponential generating series, Newton's iteration can be deduced from the equation~$\bs Y(z)=\bs H(z,\bs Y(z))$ over power series. In the case of \emph{ordinary} generating series, the analogous equation is~$\bs Y(z)=\bs{\Phi}_{\bc H}(\bs Y(z))$, but Newton's iteration also uses~$\bs{\Phi}_{\bd\partial\bc H/\bd\partial\bc Y}$ that we recover from the combinatorial origin of the system. 
         
We now show that the iterations defined using these operators converge as expected.
\begin{definition}
        The convergence of a sequence $(\bs F^{[n]}(z))_{n\ge 0}$ to a vector of series $\bs F(z)$ is  
        \emph{quadratic} when the distance is at least squared at each step. In other words, the number of matching coefficients doubles.
        \end{definition}

\begin{theorem}\label{th:newt_sg}
    Let $\bc Y=\bc H({\mathcal Z},\bc Y)$ be well founded and let~$\bc S$ denote its solution.
	Newton's iterations defined~by
	\begin{equation*} 
	\bd Y^{[n+1]}(z)=\NH\left(z,\bd{Y}^{[n]}(z)\right),\qquad \bd Y^{[n+1]}(z)=\TNH\left(z,\bd{Y}^{[n]}(z)\right)
	\end{equation*}  
with $\bd Y^{[0]}(z)=\bd{0}$, converge quadratically, respectively to~$\bd{S}(z)$ and~${\bd{\tilde S}}(z)$.
\end{theorem}
\begin{proof}
        This is an application of Lemma~\ref{lem:trans1} to the combinatorial Newton iteration of Theorem~\ref{th:newton}.
\end{proof}    
\begin{example}\label{ex:catalan_series} The ordinary and exponential generating series of Catalan trees satisfy $T(z)=z/(1-T(z))$.
	 	The corresponding Newton iteration is
	\[T^{[n+1]}=T^{[n]}+\frac{zV^{[n]}-T^{[n]}}{1-z(V^{[n]})^2},\qquad V^{[n]}=\frac{1}{1-T^{[n]}},\]
with~$T^{[0]}=0$. The first steps are:
\begin{align*}
T^{[1]}&=\mathbf{z+z^2}+z^3+z^4+z^5+z^6+z^7+z^8+z^9+z^{10}+\dotsb,\\
T^{[2]}&=\mathbf{z+z^2+2z^3+5z^4+14z^5+42z^6}+131z^7+417z^8+1341z^9+4334z^{10}+\dotsb,\\
T^{[3]}&=\mathbf{z+z^2+2z^3+5z^4+14z^5+42z^6+132z^7+429z^8+1430z^9+4862z^{10}}+\dotsb.
\end{align*}
Boldfaced terms are those matching with the solution.	
\end{example}  

\begin{example}\label{ex:unorderedrooted_series} The ordinary generating series~$\tilde{G}(z)$ of Cayley trees 	satisfies
\[\tilde{G}(z)=z\exp\!\left(\tilde{G}(z)+\frac12\tilde{G}(z^2)+\dotsb\right).\]
The species of Cayley trees is defined by~${\mathcal G}={\mathcal Z}\cdot\Set({\mathcal G})$. The combinatorial Newton operator corresponding to ${\mathcal H}({\mathcal G})= {\mathcal Z}\cdot\Set({\mathcal G})$ is therefore 
\[{\mathcal N}_{\mathcal H}({\mathcal Y})={\mathcal Y}+\Seq({\mathcal B})\cdot({\mathcal B}-{\mathcal Y})\quad\text{with}\quad{\mathcal B}={\mathcal Z}\cdot\Set({\mathcal Y}).\]
The associated Pólya operator is given by
\[\Phi_{{\mathcal N}_{\mathcal H}}:Y(z)\mapsto Y(z)+\frac{B(z)-Y(z)}{1-B(z)}\quad\text{with}\quad B(z)=z\exp\!\left(Y(z)+\frac12Y(z^2)+\dotsb\right).\]
The first few steps give:
\begin{align*}
\tilde{G}^{[1]}&=\mathbf{z+z^2}+z^3+z^4+z^5+z^6+z^7+z^8+z^9+z^{10}+\dotsb,\\ 
\tilde{G}^{[2]}&=\mathbf{z+z^2+2z^3+4z^4+9z^5+20z^6}+47z^7+110z^8+261z^9+620z^{10}+\dotsb,\\
\tilde{G}^{[3]}&=\mathbf{z+z^2+2z^3+4z^4+9z^5+20z^6+48z^7+115z^8+286z^9+719z^{10}}+\dotsb.
\end{align*}
\end{example}

Computing all those series up to the desired order is actually unnecessary. It is possible to truncate the series at each step of the iteration.
We use the notation $f(z)  \bmod z^N$ to represent the series $f(z)$ truncated at its $N$th coefficient.
\begin{proposition}\label{th:newt_sg_trunc}
    Let $\bc Y=\bc H({\mathcal Z},\bc Y)$ be a well-founded system.
	Newton's iterations defined~by
	\begin{equation}\label{eq:newt_sg_trunc}
	\bd Y^{[n+1]}(z)=\NH\left(z,\bd{Y}^{[n]}(z)\right)\bmod z^N,\qquad \bd Y^{[n+1]}(z)=\TNH\left(z,\bd{Y}^{[n]}(z)\right)\bmod z^N
	\end{equation}  
with $\bd Y^{[0]}(z)=\bc{S}(\bd{0})$ and $N=2^{n+1}$, converge quadratically, respectively to~$\bd{S}(z)$ and~${\bd{\tilde S}}(z)$.
\end{proposition}    
\begin{proof}
        This is an application of Lemma~\ref{lem:trans1} to the combinatorial Newton iteration of Proposition~\ref{coro:newton_trunc}.
\end{proof}
\begin{example}\label{ex:truncated_catalan_series} The truncated iteration for Catalan trees becomes
	\[T^{[n+1]}=T^{[n]}+\frac{zV^{[n]}-T^{[n]}}{1-z(V^{[n]})^2}\bmod z^{2^{n+1}},\qquad V^{[n]}=\frac{1}{1-T^{[n]}}\bmod z^{2^{n+1}},\]
with~$t^{[0]}=0$. The first few terms are:
\begin{align*}
T^{[1]}&=z,\\
T^{[2]}&=z+z^2+2z^3,\\
T^{[3]}&=z+z^2+2z^3+5z^4+14z^5+42z^6+132z^7,\\
T^{[4]}&=z+z^2+2z^3+5z^4+14z^5+42z^6+132z^7+429z^8+1430z^9+4862z^{10}+16796z^{11}\\
&\quad+58786z^{12}+208012z^{13}+742900z^{14}+2674440z^{15}.
\end{align*}	
Even if it seems slower than the iteration of Example~\ref{ex:catalan_series}, only correct coefficients are computed and the convergence is still quadratic. The main improvement is that the computations can be done using no more precision than what is needed, which leads to a faster algorithm.
\end{example}

\paragraph{Optimized Newton iteration}
The optimization of Newton's iteration described in section \ref{subsec:optimized_Newton} for the computation of species can also be adapted to  series: it   computes in parallel the inverse of the Jacobian matrix and the iterates of the solution, so that it only needs  one iteration for the matrix inversion at each iteration of the solution.

Algorithm \ref{algo:newtonSeries} describes the computation of exponential (or ordinary) generating series at precision $N$ by the optimized Newton iteration. Given a vector of species $\bc H=({\mathcal H}_{1:m})$, such that the system ${\bc Y=\bc H({\mathcal Z},\bc Y)}$ is well founded, together  with  an integer $N$, the algorithm computes $\bs S(z) \bmod z^N$, where $\bs S(z)$ is the series solution of the system such that $\bs S(0) =\bs H^m(0,\bs{0})$. For this purpose, it also computes $\left(\Id-{\bs\partial\bs H}/{\bs\partial\bs Y}(z,\bs S(z))\right)^{-1}\bmod z^{\lceil N/2\rceil}$,  the inverse  of the Jacobian matrix, at precision $\lceil N/2\rceil$. 
This is done recursively, starting from $N$ the desired number of terms. By Theorem~\ref{th:newt_sg}, one Newton step from $\lceil N/2\rceil$ terms computes the result and, by the reasoning in the proof of Proposition~\ref{prop:newt_opt}, $\lceil N/2\rceil$ terms of the inverse of the Jacobian matrix are sufficient. The end of the recursion, when $N=1$, is obtained by the initial $\bc{S}(\bs{0})=\bc{H}^m(\bs{0},\bs{0})$ and $\,\bc{U}^{-1}(\bs{0})=\Id$.

\SetKwBlock{Function}{Function}{end}
\SetKwFunction{recSeries}{recSeries}
\SetKwFunction{sH}{\large sH}
\SetKwFunction{sJ}{\large sJ}
\SetKwComment{Comment}{// }{}

\begin{Algorithm}[tb]{0.85\textwidth}   
\SetAlgoRefName{newtonSeries}
\caption{Computation of generating series with a given precision\label{algo:newtonSeries}}
\DontPrintSemicolon
\Input{A type: EGS (exponential g.s.) or OGS (ordinary g.s.)}
\Input{A vector of species $\bc H=({\mathcal H}_{1:m})$, such that $\bc Y=\bc H({\mathcal Z},\bc Y)$ is well founded} 
\Input{An integer $N$}
\Output{The first $N$ terms of the exponential (or ordinary) generating series $\bs S(z)$ of the solution of  $\bc Y=\bc H({\mathcal Z},\bc Y)$}
\smallskip

\SetAlgoNoLine
\Begin{
\SetAlgoVlined
        Compute the Jacobian matrix $\bc J({\mathcal Z}, \bc Y):={\bd\partial\bc H}/{\bd\partial\bc Y}({\mathcal Z}, \bc Y)$\;
        
        \uIf{the type is {\em EGS}}{
                Set up a procedure \sH: $(\bs Y(z), N)\mapsto\bs H(z,\bs Y(z))\bmod z^N$\;
                Set up a procedure \sJ: $(\bs Y(z), N)\mapsto\bs J(z,\bs Y(z))\bmod z^N$\;
                \Comment*[r]{ $\bs H$ and $\bs J$ are the e.g.s. of $\bc H$ and $\bc J$}
        }
        \Else(\textit{(the type is {\em OGS})}){       
                Set up a procedure \sH: $(\bs Y(z), N)\mapsto\bs{\Phi}_{\bc H}(z,\bs Y(z))\bmod z^N$\;     
                Set up a procedure \sJ: $(\bs Y(z), N)\mapsto\bs{\Phi}_{\bc J}(z,\bs Y(z))\bmod z^N$\;
                \Comment*[r]{ $\bs{\Phi}_{\bc H}$ and $\bs{\Phi}_{\bc J}$ are the Pólya operators of $\bc H$ and $\bc J$}
        }
        
        $\bs U$, $\bs Y$:= \recSeries($N$)\;
        \Return $\bs Y$
        }
 
\medskip

\Function(\recSeries){
\Input{An integer $N$}
\Output{$\bs U$: $\left(\Id-\bs J(z,\bs S(z))\right)^{-1}\bmod z^{\lceil N/2\rceil}$}
\Output{$\bs Y$: $\bs S(z) \bmod z^N$}
\SetAlgoNoLine
\Begin{
\SetAlgoVlined
       \lIf{$N = 1$}{\Return  $\bs Y:= \bs H^m(0,\bs{0})$ and  $\bs U:=\Id$}\\
        \Else{
                $\bs U$, $\bs Y:=$ \recSeries($\lceil N/2\rceil$)\;
                $\bs U:= \bs U + \bs U\cdot(\sJ(\bs Y,{\lceil N/2\rceil})\cdot \bs U +\Id-\bs U) \bmod z^{\lceil N/2\rceil}$\;
                $\bs Y:= \bs Y+ \bs U\cdot(\sH(\bs Y,N) -\bs Y) \bmod z^N$ \;
                \Return $\bs U$, $\bs Y$ 
                }
        }
}
  
\end{Algorithm}

The difference between the two types of generating series lies on the way to compute the series for $\bc H$ and $\bd\partial\bc H/\bd\partial\bc Y$. The computation of their exponential generating series is straightforward using rules such as those given in Table~\ref{tab:sum_esp_sg}, while for ordinary generating series, the Pólya operators are needed. Since the precision is given in input, these operators can be truncated in order to make the computation tractable. In both cases, the inner recursive procedure is the same.

\begin{example}\label{ex:newton_cayley}Here are the iterates computed by Algorithm~\ref{algo:newtonSeries} to get the exponential and ordinary generating series of Cayley trees (${\mathcal Y}={\mathcal Z}\cdot\Set({\mathcal Y})$) with precision~$N=10$. 

\medskip

\noindent\begin{tabular}{|p{1.1cm}@{}p{0.475\textwidth}|p{1.1cm}@{}p{0.33\textwidth}|}
\hline
&Exponential generating series&&Ordinary generating series\\
\hline
        $U^{[1]}=$&$1\phantom{\Bigl(}$ & $U^{[1]}=$&$1$ \\
        $Y^{[1]}=$&$z$ & $Y^{[1]}=$&$z$\\
        \hline
        $U^{[2]}=$&$1+z\phantom{\Bigl(}$
        &$U^{[2]}=$&$1+z$\\
        $Y^{[2]}=$&$z+\frac{2}{2}{z}^{2}$
        &$Y^{[2]}=$&$z+{z}^{2}$\\[4pt]
        \hline
        $U^{[3]}=$&$1+z+\frac{4}{2}{z}^{2}\phantom{\Bigl(}$
        &$U^{[3]}=$&$1+z+2{z}^{2}$\\
        $Y^{[3]}=$&$ z+\frac{2}{2}{z}^{2} +\frac{9}{6}{z}^{3} +     \frac{64}{24}{z}^{4} +   \frac{625}{120}{z}^{5}$
        &$Y^{[3]}=$&$1+z+2{z}^{2}$\\[4pt]
        \hline
        $U^{[4]}=$&$1+z+\frac{4}{2}{z}^{2}+\frac{27}{6}{z}^{3} +     \frac{256}{24}{z}^{4} +   \frac{3125}{120}{z}^{5}\phantom{xx}$
        &$U^{[4]}=$&$1+z+2{z}^{2}+5{z}^{3}+13{z}^{4}+35{z}^{5}\phantom{\Bigl(}$\\[8pt]  
        $Y^{[4]}=$&$z+\frac{2}{2}{z}^{2} +\frac{9}{6}{z}^{3} +     \frac{64}{24}{z}^{4} +   \frac{625}{120}{z}^{5}+  \frac{7776}{720}{z}^{6}+\frac{117649}{5040}{z}^{7} +   \frac{2097152}{40320}{z}^{8} +  \frac{43046721}{362880}{z}^{9} +\frac{1000000000}{3628800}{z}^{10}\phantom{\Bigl(}$
        &
$Y^{[4]}=$&$z+{z}^{2}+2{z}^{3}+4{z}^{4}+9{z}^{5}+20{z}^{6}+48{z}^{7}+115{z}^{8}+286{z}^{9}+719{z}^{10}\phantom{\Bigl(}$  \\          
        \hline
\end{tabular}

\end{example}

\subsection{Arithmetic Complexity of Enumeration}\label{sec:arith_compl}
We now turn to the complexity analysis of the computation of the generating series up to precision $N$, or equivalently of the computation of enumeration sequences for structures up to size $N$. This means that we are computing with 
truncated power series of the form
\[
f_N(z)  = \sum_{n=0}^N f_nz^n = f(z) \bmod z^{N+1}.
\]             
We first deal with the arithmetic complexity (number of arithmetic operations), while the next section deals with the number of bit operations.

We show that the arithmetic complexity of the computation of the generating series of the species solution of~$\bc Y= \bc H ({\mathcal Z}, \bc Y)$ with precision~$N$ can be expressed in terms of a
\emph{arithmetic complexity of the species}~$\bc H$. In  the framework of constructible species, this complexity turns out to be quasi-optimal in~$N$.

\begin{definition}\label{def:speciesComplexity}  A multisort species $\bc H({\mathcal Z},\bc Y)$ is of \emph{exponential arithmetic complexity $C_e$} (resp. \emph{ordinary arithmetic complexity $C_o$)} when, for any species ${\bc U}({\mathcal Z})$, it is possible to compute the exponential (resp. ordinary) generating series of $\bc H({\mathcal Z},\bc U({\mathcal Z}))$ with precision~$N$ from the exponential (resp. ordinary) generating series of $\,\bc U({\mathcal Z})$ in $C_e(N)$ (resp. $C_o(N)$) arithmetic operations.
\end{definition}

When the type of series has no influence on the complexity result, we omit the subscript, and denote the arithmetic complexity by $C$ instead of $C$ or $C_o$.

\begin{lemma}\label{lemma:complexitycomposition}
Let  $\bc F$ and  $\bc G$ be species with respective arithmetic complexity
$C_1$ and $C_2$, then the composition $\bc F \circ \bc G$ has  arithmetic complexity $C_1+C_2$.	
\end{lemma}
\begin{proof}
For any species ${\bc U}$,	the computation of the generating series of $\bc G(\bc U)$ with precision $N$ has arithmetic complexity $C_2(N)$, and  then then  
computation of the generating series of $\bc F(\bc G(\bc U))$ with precision $N$ requires 
$C_1(N)$ more arithmetic operations.
\end{proof}

\noindent An important illustration comes from constructible species.
\begin{proposition}\label{prop:constructible} Iterative unisort constructible species have arithmetic complexity in $O(\M)$, both in the ordinary and exponential frameworks. So have their derivatives.	
\end{proposition}
Here, we use the classical notation $\M$ to denote a \emph{multiplication function}: $\M(N)$ is an upper bound on the number of arithmetic operations needed in the product of two polynomials of degree at most $N$. In particular~$\M(N)=O(N\log N)$ with algorithms based on the fast Fourier transform. However it is more convenient to state the results in terms of~$\M(N)$; replacing this function by the arithmetic complexity of the underlying multiplication in an actual implementation gives a more accurate estimate.
We use the usual conventions in this context: $\M(N_1+N_2)\le \M(N_1)+\M(N_2)$, and $N\log N = O(\M(N))$, see~\cite{GaGe99} for more information.

\begin{proof}
We are interested in constructible species  obtained by composition of the basic operators presented in Table~\ref{tab:sum_esp_sg}.
We first state the arithmetic complexity for these basic operators. 

For exponential  generating series, the Newton iteration method allows for computing rapidly on power series: calculating the \emph{reciprocal} of a power series, or its  \emph{logarithm}, \emph{power} and \emph{exponential} up to~$z^N$ can be achieved  with arithmetic complexity  $O(\M(N))$. These results are classical~\cite{GaGe99}; we present explicit algorithms in Section~\ref{sec:extend}.
By the previous lemma, the arithmetic complexity of  the composition of two such operators  is also in $O(\M(N))$.

Ordinary generating series are expressed using P\'olya operators. Those in Table~\ref{tab:sum_esp_sg} can be expressed in terms of a power series of the form                                                     
\[H(z)=\sum_{k>0}{c_kG(z^k)}\]
(taking $\log(1/(1-G(z)))$ for $G(z)$ in the case of cycles). The computation of~$G(z^k)$ with precision~$N$ only uses $G(z)$ with precision $\lfloor N/k\rfloor$ and no arithmetic operation, so that the whole sum requires at most $\sum_k N/k\sim N \log N$ arithmetic operations, which is again in~$O(\M(N))$. The final exponential for sets has arithmetic complexity~$O(\M(N))$ too.

The constructions with a cardinality constraint depending on an integer~$\ell$ lead to polynomials in $G(z),G(z^2),\dots,G(z^\ell)$, whose expansion at precision~$N$ are thus also of arithmetic complexity~$O(\M(N))$, the implied constant in the $O()$ term depending on the actual polynomial.            

By Table~\ref{tab:deriv}, all derivatives of constructible species are constructible, whence the last part of the result.
\end{proof}

\noindent We now compute the arithmetic complexity of implicitly defined species. 
	\begin{theorem}\label{prop:complexityFixedPoint} 
Let ${\bc Y}= {\bc H} ({{\mathcal Z}}, {\bc Y})$ be well-founded and assume that~$\bc H$ and $\bc\partial\bc H/\bc\partial\bc Y$ have arithmetic complexity $C$. Then the species ${\bc S}$ solution of the system has arithmetic complexity $O(C + \M)$. 
	\end{theorem}
The following consequence is now immediate from Proposition~\ref{prop:constructible}.
\begin{corollary}\label{prop:combSystem} 
	Constructible species  have \emph{arithmetic complexity} $O(\M)$,  both in the ordinary and exponential frameworks.	
\end{corollary}

\paragraph{Proof of theorem~\ref{prop:complexityFixedPoint}}
We first deal with the computation of the truncated generating series of the species~$\bc S$ itself.
	Newton's iteration on combinatorial systems is a ``divide and conquer'' algorithm that computes generating series up to order~$N$. 	Denoting by $T(N)$ the arithmetic complexity of Newton's iterations from Eq.~\eqref{eq:newt_sg_trunc}
				for computing the  first $N=2^n$ terms of the solution series of $\bc Y=\bc H({\mathcal Z}, \bc Y)$, we have
\begin{equation} \label{eq:newt_cplx}
T(N) = T({N}/{2}) + 2C(N) +R(N) + K\M(N)+ O(N),
\end{equation}
where $R(N)$ is the  arithmetic complexity of computing the inverse of the matrix at precision $N$
and $K$ counts the number of products of power series involved in the product of this inverse by the last vector.
Finally, series vector addition and subtraction at precision $N$ require $O(N)$ arithmetic operations.
The term~$R(N)$ can itself be replaced by~$O(\M(N))$, the inverse of a matrix of power series being  itself computed by Newton's iteration (see~\S\ref{sec:newton_operator}). With the hypothesis that $N\log N=O(\M(N))$, the arithmetic complexity of expanding the generating series of~$\bc S$ to the order~$N$ is thus~$O(C(N)+\M(N))$.

\def\hatH{\hat{\bc H\;}\!\!}
For a species~$\,{\mathcal U}$ of arithmetic complexity~$\hat{C}$, the species~$\bc S(\,{\mathcal U}({\mathcal Z}))$ is solution to the system $\bc Y=\hatH({{\mathcal Z}},\bc Y)$, with $\hatH({{\mathcal Z}},\bc Y)=\bc H({\,{\mathcal U}}({{\mathcal Z}}),\bc Y))$. The hypothesis on~$\bc H$ implies that~$\hatH$ and $\bc\partial\hatH/\bc\partial\bc Y({{\mathcal Z}},\bc Y)=\bc\partial\bc H/\bc\partial\bc Y({\,{\mathcal U}}({\mathcal Z}),\bc Y)$ both have arithmetic complexity~$C+\hat{C}$. Thus by the previous argument, the solution to this system can be computed within the desired complexity, which concludes the proof of the theorem.
\qed

\paragraph{Optimized Newton iteration} All the steps of the reasoning apply to the optimized iteration used in Algorithm~\ref{algo:newtonSeries}:
\begin{enumerate}
	\item Truncation of species (the analogue of Corollary~\ref{coro:newton_trunc}), by truncating~$\,\bc U^{[n+1]}$ at size~$2^n$ and $\bc Y^{[n+1]}$ at size~$2^{n+1}$;
	\item Translation into exponential and ordinary generating series (analogues of Theorem~\ref{th:newt_sg} and its truncated variant Proposition~\ref{th:newt_sg_trunc});
	\item Arithmetic complexity estimate for implicit species: in the proof of Theorem~\ref{prop:complexityFixedPoint}, the cost~$R(N)$ of computing the inverse of the matrix in Equation~\eqref{eq:newt_cplx} is replaced by the smaller cost of one matrix product, which has an impact on the constant hidden in the $O()$ estimate of that theorem.
\end{enumerate}

\subsection{Bit Complexity of Enumeration}\label{sec:bit_compl}
In the preceding section, we only computed the number of arithmetic operations. While this leads to complexity estimates that measure correctly the time needed by the computations when the coefficients are e.g., floating point numbers, this measure is less pertinent when dealing with exact integer or rational coefficients, whose size grows with the precision of the series.
We now estimate the bit complexity of Newton's iteration on power series taking into account the cost of the operations on the integer or rational coefficients. This complexity is
expressed using a function  $\M_{\mathbb{Z}}(N)$, which represents an upper bound on the number of binary operations for multiplying two integers of size at most $N$ bits. Using the fast Fourier transform one has $\M_{\mathbb{Z}}(N)=O(N\log N\log\log N)$ but expressing the complexity in terms of $\M_{\mathbb{Z}}(N)$ allows for a better understanding of the actual time needed by an implementation based on a given arithmetic library. As we did before with the function~$\M$, we use the usual assumption~$\M_{\mathbb{Z}}(N_1+N_2)\le \M_{\mathbb{Z}}(N_1)+\M_{\mathbb{Z}}(N_2)$ and refer to~\cite{GaGe99} for more information.

\subsubsection{Ordinary Generating Series} 
\begin{lemma}\label{lemma:coeff-ord}
The product of \emph{analytic} ordinary generating series at precision~$N$  has bit complexity  $O(\M_{\mathbb{Z}}(N)\times \M(N))$.
\end{lemma}
\begin{proof}
This is a consequence of the exponential growth formula for coefficients of a power series
$F(z)$ with radius of convergence $ \rho>0$: $ \limsup([z^{N}] F(z))^{1/N}=1/\rho$.
Since the coefficients are integers, their numbers of bits grow at most like~$N\log_2 1/\rho$, so that all the coefficients up to the $N$th one have size bounded by $KN$ for some~$K>0$. The computation of the product then requires $O(\M(N))$ operations, each of bit complexity bounded by~$O(\M_{\mathbb{Z}}(KN))=O(\M_{\mathbb{Z}}(N))$, whence the result.
\end{proof}
Analytic series with integer coefficients have radius of convergence at most~1. When this radius is smaller than~1, the bit size of the truncated power series is quadratic in~$N$, while the above complexity estimate is~$O(N^2\log^2N\log\log N)$ if fast Fourier transform is used. Thus the bit complexity is linear in the size of the input and output, up to logarithmic factors.

\begin{proposition}\label{prop:bit-ordinary}
Let ${\bc Y}= {\bc H} ({{\mathcal Z}}, {\bc Y})$ be  well-founded and assume that~$\bc H$ and $\bc\partial\bc H/\bc\partial\bc Y$ have ordinary arithmetic complexity $C_o$ and that the ordinary generating series ${\bd{\tilde S}}$ of the solution of the system is analytic at~0. Then the computation of~${\bd{\tilde S}}$ at precision~$N$ has bit complexity $O (\M_{\mathbb{Z}}(N) \times (C(N) + \M(N)))$. 
\end{proposition}
\begin{proof}
In Newton's iteration,  all operations at precision $N$ add an extra factor in  $O (\M_{\mathbb{Z}}(N))$, so that	
the recurrence relation (\ref{eq:newt_cplx}) for bit complexity now turns into 
\[B(N) = B({N}/{2}) + \M_{\mathbb{Z}}(N)\left(2C_o(N) +R(N) + K\M(N)+ O(N)\right),\]
whose solution is bounded by~$O(\M_{\mathbb{Z}}(N)(C_o(N)+\M(N)))$.
\end{proof}
\begin{corollary}\label{coro:sgo-constructible}
The computation of the ordinary generating series of constructible species with precision~$N$ 
has bit complexity $O(\M(N)\times \M_{\mathbb{Z}}(N))$.
\end{corollary}
\begin{proof} The bound~$C_o(N)=O(\M(N))$ is given by Proposition~\ref{prop:constructible}. The analyticity of the solution is proved in Theorem~\ref{cor:ogf} below.
\end{proof}
As above, when the radius of convergence is not equal to~1, this complexity is quasi-optimal: it is linear in the size of the output, up to logarithmic factors.
		
\subsubsection{Exponential Generating Series}
Exponential generating series have rational coefficients and a bit of care is needed in order to multiply them efficiently. Still, we obtain a quasi-optimal method.
\begin{lemma}\label{lemma:coeff-exp}
	The product of \emph{analytic} exponential generating series at precision~$N$  has bit complexity  $O(\M(N)\times \M_{\mathbb{Z}}(N \log N))$.
	\end{lemma}
\begin{proof}
When the radius of convergence of an exponential generating series $f(z)=\sum_{n\ge0} f_n z^n /n!$ is finite and nonzero, the coefficient $f_n$ has bit size bounded by $O(n \log n)$, thus the bit size of the coefficients in the truncated  series      $f_N(z)=\sum_0^N f_n z^n /n!$ is in $O(N\log N)$. When the radius of convergence is infinite, $f_n/n!$ tends to 0 faster than $n\log \rho$ for any $\rho$, thus the coefficients in $f_N(z)$ are still in $O(N\log N)$.

In order to exploit the integrality of the coefficients~$f_n$ in $f_N(z)$, we change the representation, and
consider the truncated series
 \[F_N(z) = \sum_{n=0}^N N! f_n z^n /n! = N! f_N(z),\] 
in which the bit size of the coefficients is now bounded by $O(N \log N)$. 
The product of series is computed by 
 \[(FG)_N(z) = \frac1{N!} F_N(z)G_N(z).\] 
This transformation adds an extra term in $O(N \M_{\mathbb{Z}}(N\log N))$ in the bit complexity of the product, since each of the $N$ terms must be divided by $N!$; this extra term is absorbed by the cost of the product of the series $F_N(z)$ and $G_N(z)$, which is in $O(\M(N)\M_{\mathbb{Z}}(N\log N))$.
\end{proof}

\begin{proposition}\label{prop:bit-exponential}
Let ${\bc Y}= {\bc H} ({{\mathcal Z}}, {\bc Y})$ be a well-founded system and assume that~$\bc H$ and $\bc\partial\bc H/\bc\partial\bc Y$ have exponential arithmetic complexity $C_e(N)$ and that the exponential generating series ${\bs S}$ of the solution of the system is analytic at~0. Then the computation of the series~${\bs S}$ at precision~$N$ has bit complexity $O (\M_{\mathbb{Z}}(N\log N) \times (C_e(N) + \M(N)))$.
\end{proposition}
\begin{proof}
Similar to that of Theorem~\ref{prop:bit-ordinary}.
\end{proof}

\begin{corollary}\label{coro:sge-constructible}
The computation of the exponential generating series of constructible species with precision~$N$ 
has bit complexity $O(\M(N)\times \M_{\mathbb{Z}}(N\log N))$.
\end{corollary}
\begin{proof}
Similar to that of Corollary~\ref{coro:sgo-constructible}.
\end{proof}

For the optimized Newton iteration, the bit complexity estimate follows the same lines,  and, as in the case of arithmetic complexity, the constant in the $O()$ estimate is smaller.

\section{Analytic Species}\label{sec:num}   
Random generation by Boltzmann sampling~\cite{DuFlLoSc04,FlFuPi07} depends on so-called \emph{oracles} giving numerical values of generating series inside their disk of convergence.
In this section, we transfer the results of the previous one to provide numerical Newton iterations that converge to these values. Classically, sufficient conditions for the convergence of Newton's iteration include a starting point close enough to the root. In this combinatorial context however, our iteration manages to capture the combinatorial origin of the equations and converges unconditionally when started at the origin.

\subsection{Basic Properties}
We depart slightly from the general framework of species theory to concentrate on cases when the series converge in a neighborhood of~0. This is motivated by Theorem~\ref{cor:ogf} below showing that all generating series coming from implicit constructible species have a nonzero radius of convergence.

\begin{definition}\label{def:analyticSpecies}A species $\bc H(\bc Z)$ is called \emph{analytic} if its exponential generating series~$\bs H(\bs z)$ is analytic in the neighborhood of $\bs{0}$.
\end{definition}
\par \noindent A subspecies of an analytic species is itself analytic, by absolute convergence.
Also, if a species $\bc H(\bc Z,\bc Y)$ is analytic, its Jacobian matrix is  analytic too.
\begin{lemma}\label{lemma:expo_gf}Iterative constructible species are analytic.
\end{lemma}
\begin{proof}This follows from the analyticity of the generating series in Column~5 of Table~\ref{tab:sum_esp_sg}, the analyticity of the composition of analytic series at~0 and that of polynomials with analytic series.
\end{proof}
\begin{proposition}[Implicit analytic species]\label{prop:analytic} Let $\bc Y=\bc H(\bc Z,\bc Y)$ be a well-founded system, and~$\bc S$ its solution. If the species $\hatH(\bc Z,\bc U):=\bc H(\bc Z,\bc S(\bs{0})+\bc U)$ is analytic, then $\bc S$ is an \emph{analytic} species.
\end{proposition}
Note that when the system is well founded at~0, the condition simplifies to~${\bc H}(\bc Z,\bc Y)$ being analytic.
\begin{proof}
First we observe that the species~$\hatH(\bc Z,\bc U)$ is well defined.
Let $\Sun({\mathcal Z}_1,\bc Z)$ be defined as in Theorem~\ref{th:GIST}.
This species is polynomial in~${\mathcal Z}_1$, so that Prop.~\ref{prop:IPFS} implies that for all $i=1,\dots,m$ ($m$ the number of coordinates of~$\bc H$), either the $i$th coordinate of~$\Sun({\mathcal Z}_1,\bs{0})=\bc S(0)$ is 0 or $\bc H(\bc Z,\bc Y)$ is polynomial in~${\mathcal Y}_i$. By Lemma~\ref{lem:HS0-polynomial1}, this in turns implies that $\bc H(\bc Z,\bc S(\bs{0})+\bc U)$ is polynomial in~${\mathcal Z}_1$, so that its composition with~${\mathcal Z}_1=1$ is defined.

Next, the proof is a simple consequence of the implicit function theorem for analytic functions (see \textit{e.g.}, \cite[Ch.~IV]{Cartan95}). The necessary conditions are fulfilled: by hypothesis $\hat{\bs H\,}\!(\bs z,\bs u)$ is analytic and so is its Jacobian, moreover the matrix~$(\Id-\bs\partial\bs H/\bs\partial\bs u)(0,0)$ is invertible, by the nilpotence of $\bs\partial\bc H/\bs\partial\bc Y(\bs{0},\Sun(1,\bs{0}))$, itself a consequence of Proposition~\ref{prop:IPFS}.
\end{proof}

\subsection{Dominant Species and their Generating Series} 

The algorithms introduced in later sections use bounds on series associated to species. In order to define relevant  bounds for implicit subspecies, we introduce the notion of dominant species that is consistent with well-founded systems.
The domination of species translates
 into \emph{majorant series}, that play an important role in the design of our numerical oracle in Section~\ref{section:numerical-evaluation}.
\begin{definition}
For any formal power series~$F$ and~$G$ with nonnegative coefficients, we say that $G$ is a \emph{majorant series} for $F$ and write $F\lhd G$ if for all $\bs n\ge\bs{0}$, their coefficients satisfy $[\bs z^\bs n]F(\bs z)\le[\bs z^\bs n]G(\bs z)$. The notation $\bs F\lhd\bs G$ for vectors or matrices means that the property holds entry by entry.
\end{definition}

\begin{definition}
  A multisort species $\bc F(\bc Z,\bc Y)$, is \emph{dominated} by the species $\bc G(\bc Z,\bc Y)$ if
\begin{enumerate}	\item $\bc G$ is flat (see~\S\ref{subsec:flat});
	\item the ordinary generating series obey ${\bd{\tilde F}}\lhd{\bd{\tilde G}}$;
	\item for any $(\bs n,\bs k)$ and any coordinate~$i$, $[\bs z^{\bs n}\bs y^{\bs k}]{\bd{\tilde F}}_i=0\Rightarrow[\bs z^{\bs n}\bs y^{\bs k}]{\bd{\tilde G}}_i=0.$
		\end{enumerate} 
This is denoted by $\bc F(\bc Z,\bc Y) \lhd \bc G(\bc Z,\bc Y)$.
\end{definition} 
The last condition ensures that dominant species retain some of the  characteristics of those they dominate.

\begin{example} \label{ex:domSet}
Sets are dominated by sequences: sequences are flat and both ordinary generating series are equal (see Table~\ref{tab:sum_esp_sg}).
\end{example}

\begin{example}\label{ex:domCyc} 
Cycles are dominated by nonempty sequences:      
Table~\ref{tab:sum_esp_sg} gives the ordinary generating series of cycles as $\sum{\phi(k)/k\log(1/(1-z^k))}$, but this rewrites as $z/(1-z)$ so that again the ordinary generating series are identical.
\end{example} 

While the definition of dominance is in terms of ordinary generating series, the exponential generating series also follow the same inequality:
\begin{lemma}\label{lemma:egf}
Let $\bc F$ and $\bc G$ be two multisort species such that $\bc G$ is a dominant species for $\bc F$, then $\bs G(\bs{z})$ is a majorant series for $\bs F(\bs{z})$.       
\end{lemma}
\begin{proof} 
This is the unisort case, which extends to multisort species.

Let~$f_n$ be the number of labeled ${\mathcal F}$-structures on~$\{1,\dots,n\}$ and~$\tilde{f}_n$ the number of unlabeled such structures, and define similarly~$g_n$ and~$\tilde{g}_n$ for~${\mathcal G}$. By dominance, these are related by~$\tilde{f}_n\le\tilde{g}_n$. The number of labeled structures~$f_n$ is bounded by $n!\tilde{f}_n$, while flatness of~${\mathcal G}$ implies~$g_n=n!\tilde{g}_n$. Thus the proof is summarized by
\[f_n\le n!\tilde{f}_n\le n!\tilde{g}_n=g_n.\]
\end{proof}

Dominance passes through systems of equations.
\begin{proposition}[Implicit Dominant Species]\label{prop:dom}
   If the species~$\bc F(\bc Z,\bc Y)$ is dominated by $\bc G(\bc Z,\bc Y)$ and the system $\bc Y=\bc F(\bc Z,\bc Y)$ is well founded, then $\bc V=\bc G(\bc Z,\bc V)$ is well founded and the solution of~$\bc Y=\bc F(\bc Z,\bc Y)$ is dominated by the solution of~${\bc V}=\bc G(\bc Z,\bc V)$.
\end{proposition}  
\begin{example}\label{ex:dom2} By example~\ref{ex:domSet}, $\Set$ is dominated by $\Seq$, thus the species ${\mathcal G}$ of Cayley trees defined by ${\mathcal G} ={\mathcal Z} \cdot \Set({\mathcal G})$ is dominated by the species of Catalan trees defined by ${\mathcal T} ={\mathcal Z} \cdot \Seq({\mathcal T})$. 
This transfers to both their exponential and ordinary generating series: 	\[
	T(z)=z+2\frac{z^2}{2!}+12\frac{z^3}{3!}+120\frac{z^4}{4!}+\dotsb \mbox{ is a majorant series for } G(z)=z+2\frac{z^2}{2!}+9\frac{z^3}{3!}+64\frac{z^4}{4!}+\dotsb \mbox{ and }
	\] 
	\[
	\tilde{T}(z)=z+{z}^{2}+2{z}^{3}+5{z}^{4}+14{z}^{5}+\dotsb\mbox{ is a majorant series for } \tilde{G}(z)=z+{z}^{2}+2{z}^{3}+4{z}^{4}+9{z}^{5}+\dotsb.
	\]
\end{example}
The proof relies on the preservation of dominance by composition.
\begin{lemma}\label{lem:dom}
   Let $\bc F$, $\bc G$, $\bc A$ and $\bc B$ be multisort species such that $\bc G$ and $\bc B$ are flat. If $\bc F(\bc Z,\bc Y)$ is dominated by $\bc G(\bc Z,\bc Y)$, $\bc A$ is dominated by $\bc B$ and the composition $\bc F(\bc Z,\bc A)$ is defined, then the composition $\bc G(\bc Z,\bc B)$ is defined and $\bc F(\bc Z,\bc A)$ is dominated by $\bc G(\bc Z,\bc B)$. 
\end{lemma}

\begin{proof}
Recall that the composition $\bc F(\bc Z,\bc A)$ is defined when $\bc F$ is polynomial with respect to the coordinates for which $\bc A(\bs{0})$ is not~0. By the last part of the definition of dominant species, those are exactly the coordinates for which $\bc B(\bs{0})$ is not~0 and then $\bc G$ is polynomial with respect to them too.
	
Next, we observe that as a composition of flat species, $\bc G(\bc Z,\bc B)$ is flat too (by Lemma~\ref{lemma:composition_flat}). 
The condition on ordinary generating series is given by the following chain of equalities and inequalities in the unisort case:
\begin{align*}
	Z_{{\mathcal F}}(\tilde{ A}(z),\tilde{ A}(z^2),\dotsc)&\lhd Z_{{\mathcal F}}(\tilde{ B}(z),\tilde{ B}(z^2),\dotsc)\\
	&\lhd Z_{{\mathcal F}}(\tilde{ B}(z),\tilde{ B}(z)^2,\dotsc)=\tilde{ F}(\tilde{ B}(z))\\
	&\lhd\tilde{ G}(\tilde{ B}(z))=Z_{{\mathcal G}}(\tilde{ B}(z),\tilde{ B}(z^2),\dotsc).
\end{align*}
The first inequality comes from the domination of~${\mathcal A}$ by~${\mathcal B}$ and the positivity of the coefficients of~$Z_{\bc F}$; the second one comes from  ordinary generating series having nonnegative integer coefficients; the third one is a consequence of the domination of~${\mathcal F}$ by~${\mathcal G}$. The same reasoning applies to the multisort case.

The last property follows from Corollary~\ref{coro:zero}.
\end{proof}

\begin{proof}[Proof of Proposition~\ref{prop:dom}]
We first deal with the case when~$\bc F(\bs{0},\bs{0})=\bs{0}$, which is also the value of~$\bc G(\bs{0},\bs{0})$ by domination. 
Since $\bs\partial\bc F/\bs\partial\bc Y$ is dominated by~$\bs\partial\bc G/\bs\partial\bc Y$, so are their values at~$\bs{0}$, and the nilpotence of~$\bs\partial\bc G/\bs\partial\bc Y$ at~$\bs{0}$ follows from Lemma~\ref{lem:dom} with~$\bc F=\bc G=\bc Y^m$, where~$m$ is the dimension of the system~$\bc F$. 
Thus both sequences $(\bc Y{}^{[n]})_{n\ge 0}$ and $(\bc V{}^{[n]})_{n\ge 0}$ defined by $\bc Y{}^{[n+1]}=\bc F(\bc Z,\bc Y{}^{[n]})$ and $\bc V{}^{[n+1]}=\bc G(\bc Z,\bc V{}^{[n]})$ converge. By induction using Lemma~\ref{lem:dom} again, for all $n\ge 0$, $\bc Y{}^{[n]}$ is dominated by $\bc V{}^{[n]}$. The property on limits follows from considering any fixed size.

If $\bc F(\bs{0},\bs{0})\neq\bs{0}$ then we first consider the companion system for~$\bc F$, dominated by the companion system for~$\bc G$. These systems are well founded at~$\bs{0}$ and thus by the previous argument, the solution~$\bc S_0({{\mathcal Z}}_1)$ of the first one is dominated by the solution~$\bc T_0({{\mathcal Z}}_1)$ of the second one. In particular, $\bc S_0$ being polynomial implies that $\bc T_0$ is polynomial too, and the nilpotence of~$\bs\partial\bc G/\bs\partial\bc Y(0,\bc T_0({{\mathcal Z}}_1))$ follows again from Lemma~\ref{lem:dom}. The last condition showing that~$\bc V=\bc G(\bc Z,\bc V)$ is well founded is immediate as well. The domination of the solutions follows from the same argument as above.
\end{proof}

\subsection{Constructible Species are Analytic}
\begin{theorem}\label{cor:ogf}If $\bc Y=\bc H({\mathcal Z},\bc Y)$ is  well founded, where ${\bc H}$ is a constructible species, then it defines a species whose exponential \emph{and} ordinary generating series are analytic in the neighborhood of the origin.
\end{theorem}
The first step of the proof is to find good dominant species.
\begin{lemma}\label{lemma:flat-constructible} 
        Any constructible species is dominated by a flat constructible species.
\end{lemma}
\begin{proof}
The proof is by induction on the definition of the constructible species.
First, all the basic species are either flat ($1, {\mathcal Z}, +, \cdot $, $\Seq$ and the ${{\mathcal Y}}_i$), or dominated by flat constructible ones: $\Set$ is dominated by~$\Seq$ (Example~\ref{ex:domSet}) and $\Cyc$ is dominated by~$\Seq_{>0}$ (Example~\ref{ex:domCyc}). Any basic species ${\mathcal F}$ with a cardinality constraint is a subspecies of ${\mathcal F}$ and flatness is preserved by inclusion. Composition preserves flatness by Lemma~\ref{lemma:composition_flat}. Finally, implicit species are obtained by Proposition~\ref{prop:dom}.
\end{proof}

\begin{lemma}Flat constructible species are analytic.
\end{lemma}
These species correspond to context-free languages, so that this lemma is the classical fact that algebraic generating series are analytic.
\begin{proof}
Flat constructible species are obtained by removing~$\Set$ and~$\Cyc$ from the basic species used in Definition~\ref{def:constructible}. 
The iterative constructible species are thus clearly analytic. 
Adding a new variable~${\mathcal Y}$ to the system for each~$\Seq(\,{\mathcal U})$, and the corresponding equation~${\mathcal Y}=1+{\mathcal Y}{\,\mathcal U}$, shows that a well-founded system of flat constructible species can be rewritten as a well-founded system with \emph{polynomial}~$\bc H$. Thus the corresponding~$\hat{\bs{H}\,}\!$ in Prop.~\ref{prop:analytic} is analytic, so that the conclusion of the Proposition holds and the recursive constructible flat species are analytic as well.
\end{proof}

Theorem~\ref{cor:ogf} is now a consequence of these two lemmas, the definition of dominant species for the ordinary case, and Lemma~\ref{lemma:egf} for the exponential case.

\subsection{Numerical Evaluation}\label{section:numerical-evaluation}
A simple way to compute numerical values of the generating series inside their disk of convergence is to first compute sufficiently many terms of the power series (e.g., by Newton's iteration, using $2^n$ terms at the $n$th iteration, following Theorem~\ref{th:newt_sg}) and then evaluate the series numerically. While this method is quite efficient close to the origin, it might require a large number of coefficients for values closer to the circle of convergence of the series. We now consider faster ways, first for exponential generating series, then in the case of ordinary generating series. The latter is more involved, except of course for flat species.

\subsubsection{Exponential Generating Series}
A consequence of the nice behavior of exponential generating series under composition is that Newton's iteration can be used numerically in a straightforward way.
\begin{lemma}[Transfer, numerical part, exponential generating series]\label{lem:trans2}  
Let $\bc Y=\bc H({\mathcal Z},\bc Y)$ be a well-founded system. Let also $\bc F$ be an analytic species such that
\begin{equation}\label{iter_exp_esp}
	\bc Y^{[n+1]} = \bc F({\mathcal Z},\bc Y^{[n]}), \quad \mbox{ with } \bc Y^{[0]}=\bd{0},
\end{equation}
defines an increasing sequence of species converging to the solution~$\bc S$ of the system. Then the exponential generating series $\bd S({z})$ of~$\bc S$ has positive radius of convergence $\rho$ and for all $\alpha$ such that $|\alpha|<\rho$, the sequence
\begin{equation}\label{iter_exp_gf}
	\bs y^{[n+1]}= \bd{F}(\alpha, \bs y^{[n]}), \quad \mbox{ with } \bs y^{[0]}=\bd{0},
\end{equation}
    converges to~$\bd{S}(\alpha)$.		
\end{lemma}   
\begin{proof}  
The species $\bc S$ is analytic by proposition~\ref{prop:analytic}.	The point is to show that for all $\alpha$ such that $0\le |\alpha|<\rho$,  $\bs Y^{[n]}(\alpha)$ converges to $\bs S(\alpha)$, $\bs y^{[n]}$ is well defined and $\bs y^{[n]}=\bs Y^{[n]}(\alpha)$ (the evaluation of $\bs Y^{[n]}$ at $\alpha$ is equal to the value obtained by numerical iteration).

	By Lemma~\ref{tail-dominant} below, the monotonicity and convergence of the combinatorial sequence~${\bc Y^{[n]}}$ imply that the~$\bs Y^{[n]}$'s are analytic for~$|z|<\rho$, and that $\bs Y^{[n]}(\alpha)$ converges to~$\bs S(\alpha)$. 
	Let~$r$ be such that $|\alpha|\le r<\rho$. Assuming $\bs F(z,\bs S)$ to be analytic in a polydisk $|(z,\bs S)|\le(r,\bs S(r))$ with component-wise inequality, the vector~$\bs F(\alpha,\bs Y^{[n]}(\alpha))$ is well defined, and thus by induction 
\[
\bs y^{[n+1]}=\bs F(\alpha,\bs y^{[n]})=\bs F(\alpha,\bs Y^{[n]}(\alpha))=\bs Y^{[n+1]}(\alpha).
\]
We now prove the required analyticity of~$\bs F(z,\bs S)$.
Let $\bs F(z,\bs S)=\sum{\bs f_{i,\bs j}z^iS_1^{j_1}\dotsm S_m^{j_m}}$ and $\bs S(z)=\sum{\bs c_kz^k}$.
For each coordinate $h\in\{1,\dots,m\}$, extracting the coefficient of~$z^k$ ($k=0,\dots,N$) in the identity $\bs F(z,\bs S(z))=\bs S(z)$  leads to an inequality of the form
$$\sum_{i+j_1\ell_1+\dots+j_m\ell_m\le N}\!\!\!\!{{f}_{h,i,\bs j}r^i
(\sum_{k_1=0}^{\ell_1}{{c}_{1,k}r^{k_1}})^{j_1}
\dotsm
(\sum_{k_m=0}^{\ell_m}{{c}_{m,k}r^{k_m}})^{j_m}}
=\sum_{k\le N}{{c}_{h,k}r^k}\le {S}_h(r),\qquad N\in{\mathbb{N}},$$
where first indices denote coordinates.
The coefficients being positive, the first sum converges to ${S}_h(r)$ as $N\rightarrow\infty$. This proves the convergence of $F_h(z,\bs S)$ for $|(z,\bs S)|\le (r,\bs S(r))$ and therefore that of~$\bs F(z,\bs S)$ which 
concludes the proof. 	
\end{proof}

A first consequence of this lemma is that when~$\bs H$ is analytic, the iteration~\eqref{iter_exp_gf} with~$\bs F=\bs H$ computes values of~$\bs S$. More interesting is its use with the Newton operator for~$\bc F$.
\begin{theorem}[Newton oracle for values of exponential generating series]\label{th:newtonnum}
Let~$\bc Y=\bc H({\mathcal Z},\bc Y)$
be a well-founded system with $\bc H$ an analytic species. Let $\alpha$ be inside the disk of convergence of the exponential generating series~$\bs S(z)$ of the solution species~$\bc S$. Then the following iteration converges to the solution~$\bs S(\alpha)$ of the system:
\[\bs y^{[n+1]}=\bs y^{[n]}+\left(\Id-\frac{\bs\partial\bs H}{\bs\partial\bs Y}(\alpha,\bs y^{[n]})\right)^{-1}\cdot(\bs H(\alpha,\bs y^{[n]})-\bs y^{[n]}),\qquad\bs y^{[0]}=\bs{0}.\]
\end{theorem}
\begin{proof}
This is a consequence of the transfer lemma, since the Newton operator is an analytic species.   
\end{proof}

\begin{example}\label{ex:num-cayley} The exponential generating series of Cayley trees satisfies~$G(z)=z\exp(G(z))$ (see Ex.~\ref{ex:series_system}). Its value at $\alpha$ is given by the limit of the simple Newton iteration:
	\[g^{[n+1]}=g^{[n]}+\frac{\alpha e^{g^{[n]}}-g^{[n]}}{1-\alpha e^{g^{[n]}}}\]
with initial value~$g^{[n]}=0$. For $\alpha=1/10$, the first few values of the series of Cayley trees are:
\begin{align*}
&\mathbf{0.111}11111111111111111,\\
&\mathbf{0.1118325}2640942066782,\\
&\mathbf{0.111832559158962}89731,\\
&\mathbf{0.11183255915896296483}. 
&\end{align*}
Boldfaced digits indicate those that match the digits of the actual value.
\end{example}

\subsubsection{Ordinary Generating Series}
By Theorem~\ref{cor:ogf}, the ordinary generating series~${\bd{\tilde S}}(z)$ of a species~$\bc S$, solution of the well-founded system~$\bc Y=\bc H({\mathcal Z},\bc Y)$, with a constructible~$\bc H$, is analytic in the neighborhood of the origin. As above, while evaluating~${\bd{\tilde S}}(\alpha)$ for $\alpha$ inside the disk of convergence of~${\bd{\tilde S}}(z)$ can be achieved through numerical evaluation of power series obtained by Newton's iteration, this requires the computation of a large number of coefficients and a more direct numerical algorithm is preferable. If $\bc H$ is flat, then the technique of the previous section applies. 

Otherwise, 
we use a semi-numerical algorithm combining two techniques: computation of generating series truncated at a small order that can be evaluated at small enough values and iterative evaluation at powers of the point of interest. 

First, we observe that Pólya operators act naturally on sequences: if $\alpha$ lies inside the disk of convergence of the ordinary generating series~${\bd{\tilde H}}(z):=\bs{\Phi}_{\bc F}({\bd{\tilde G}}(z))$, then by Prop.~\ref{prop:ogfegf} and~\ref{prop:plethysm}, the sequence $({\bd{\tilde H}}(\alpha^k))_{k\in\N^\star}$ can be computed from the sequence~$({\bd{\tilde G}}(\alpha^k))_{k\in\N^\star}$. Associated to the Pólya operator $\bs{\Phi}_{\bc F}: {\bd{\tilde G}}(z) \mapsto {\bd{\tilde H}}(z)$,  we use the sequence transformer~$\bs{\Psi}_{\bc F}$
\[
\bs{\Psi}_{\bc F}: \left(\alpha , ({\bd{\tilde G}}(\alpha^k))_{k\in\N^\star}\right) \mapsto  ({\bd{\tilde H}}(\alpha^k))_{k\in\N^\star},
\]
which computes the sequence~${\bd{\tilde H}}(\alpha)$, ${\bd{\tilde H}}(\alpha^2)$, ${\bd{\tilde H}}(\alpha^3)$, \dots from $\alpha$ and~${\bd{\tilde G}}(\alpha)$, ${\bd{\tilde G}}(\alpha^2)$, ${\bd{\tilde G}}(\alpha^3)$, \dots

We now give an analogue of Lemma~\ref{lem:trans2} transferring combinatorial convergence into numerical convergence of sequences. \begin{lemma}[Transfer, numerical part, ordinary generating series]\label{lem:trans3}  
Let $\bc Y=\bc H({\mathcal Z},\bc Y)$ be a well-founded system. Let also $\bc F$ be an analytic species such that
\[
	\bc Y^{[n+1]} = \bc F({\mathcal Z},\bc Y^{[n]}), \quad \mbox{ with } \bc Y^{[0]}=\bd{0}
\]
defines an increasing sequence of species converging to the solution~$\bc S$ of the system. Assume that the ordinary generating series ${\bd{\tilde S}}$ of the species~$\bc S$ has positive radius of convergence $\rho$. 
For all $\alpha$ such that $|\alpha|<\rho$, the sequence of sequences 
\[
(\bs y^{[n+1]}_k)_{k\in\N^\star}= \bs{\Psi}_{\bc F}\left(\alpha, (\bs y^{[n]}_k)_{k\in\N^\star}\right), \quad \mbox{ with } (\bs y^{[0]}_k)_{k\in\N^\star}=(\bd{0}, \bd{0}, \dotsc),
\]
converges to the sequence~$({\bd{\tilde S}}(\alpha^k))_{k\in\N^\star}$, in the sense that $\max_k\|{\bd{\tilde S}}(\alpha^k)-\bs y^{[n]}_k\|\rightarrow0$ as $n\rightarrow\infty$.
\end{lemma}   
This lemma applies in particular to a constructible species~$\bc H$; the convergence of~${\bd{\tilde S}}(z)$ in a neighborhood of the origin is then granted by Theorem~\ref{cor:ogf}. 
\begin{proof}
The convergence of the series~${\bd{\tilde Y}}^{[n]}(z)$ in $|z|<\rho$, as well as the convergence of the sequence~${\bd{\tilde Y}}^{[n]}(\alpha)$ to~${\bd{\tilde S}}(\alpha)$ follow from the same argument as in Lemma~\ref{lem:trans2}. 

We have to show that for all $\alpha$ such that $|\alpha|<\rho$, the sequence~$\bs{\Psi}_{\bc F}(\alpha, ({\bd{\tilde Y}}^{[n]}(\alpha^k))_{k\in\N^\star})$ is well defined, for then by induction~$(\bs y^{[n+1]}_k)_{k\in\N^\star}=({\bd{\tilde Y}}^{[n+1]}(\alpha^k))_{k\in\N^\star}$.
As in the case of exponential generating series, this is obtained by a positivity argument. Indeed, the cycle index series of~$\bc F$ from Eq.~\eqref{eq:cycle-index} has nonnegative coefficients, 
so that the evaluation of~$\bs{\Phi}_{\bc F}(z,{\bd{\tilde Y}}^{[n]}(z))$ is bounded termwise by the convergent evaluation of $\bs{\Phi}_{\bc F}(z,{\bd{\tilde S}}(z))={\bd{\tilde S}}(z)$, which shows that the iteration is well defined. 
\end{proof}

\begin{theorem}[Newton oracle for values of ordinary generating series]\label{th:newt_num_ord} Let $\bc Y=\bc H({\mathcal Z},\bc Y)$ be a well-founded system with $\bc H$ a constructible species. Let~$\alpha$ be inside the disk of convergence of the ordinary generating series~${\bd{\tilde S}}(z)$ of the solution. The following iteration converges to the sequence~$({\bd{\tilde S}}(\alpha^k))_{k\in\N^\star}$:
	\[(\bs y^{[n+1]}_k)_{k\in\N^\star}=\bs{\Psi}_{\bc N_{\bc H}}\left(\alpha, (\bs y^{[n]}_k)_{k\in\N^\star}\right),\qquad (\bs y^{[0]}_k)_{k\in\N^\star}=(\bd{0},\bd{0},\dotsc) ,\]
where~$\bc N_{\bc H}$	is defined by Equation~\eqref{eq:newton_operator}.
\end{theorem}

\begin{example}\label{ex:catalan_num}
For flat species, the iteration is the same as in the previous section. For instance, the species of Catalan trees satisfies~${\mathcal T}={\mathcal Z}\Seq({\mathcal T})$. The combinatorial Newton operator is given in Example~\ref{ex:catalan_spec}. The corresponding Pólya operator is
\[Y(z)\mapsto Y(z)+\frac{zV(z)-Y(z)}{1-zV(z)^2},\qquad\text{with}\quad V(z)=\frac{1}{1-Y(z)}.\]
The associated sequence transformer is thus simply
\[(\alpha,(u_k)_{k\in\N^\star})\mapsto \left(u_k+\frac{\alpha^kv_k-u_k}{1-\alpha^kv_k^2}\right)_{k\in\N^\star},\quad\text{with}\quad v_k=\frac{1}{1-u_k},\]
and thus the value of~$T(\alpha)$ is given by the same iteration as in the labeled case:
\[y^{[n+1]}=y^{[n]}+\frac{\alpha v^{[n]}-y^{[n]}}{1-\alpha (v^{[n]})^2},\qquad v^{[n]}=\frac{1}{1-y^{[n]}},\]
with~$y^{[0]}=0$. For $\alpha=1/10$, the first few values of the series of Catalan trees are:
\begin{align*}
&\mathbf{0.11}1111111111111111111,\\
&\mathbf{0.112701}252236135957066,\\
&\mathbf{0.1127016653792}30322595,\\ 
&\mathbf{0.112701665379258311482}.
\end{align*}
\end{example}

\begin{example}\label{numer_Co} The ordinary generating series~$\tilde{G}$ of Cayley trees satisfies
\[
\tilde{G}(z)=z\exp(\tilde{G}(z)+\frac12\tilde{G}(z^2)+\dotsb)=z+z^2+2z^3+4z^4+9z^5+\dotsb.
\]
The Pólya operator for its Newton iterator was given in Example~\ref{ex:unorderedrooted_series} as
\[\Phi_{{\mathcal N}_{\mathcal H}}:Y(z)\mapsto Y(z)+\frac{B(z)-Y(z)}{1-B(z)}\quad\text{with}\quad B(z)=z\exp(Y(z)+\frac12Y(z^2)+\dotsb).\]
This translates directly into a Newton iteration on sequences converging to values of~$\tilde{G}$ at powers of $\alpha$ by the corollary above. 
In the table below $\alpha=0.3$, and  we show the evaluation of  the first three values of the sequence $(\tilde{{Y}}^{[n]}(\alpha^k))_{k\in\N^\star}$, up to the fifth iteration. 
\medskip
\begin{center}
\begin{tabular}{llll}                  
	$n$&$Y^{[n]}(0.3)$&$Y^{[n]}(0.3^2)$&$Y^{[n]}(0.3^3)$\\   \hline
	0&{\bf 0}&{\bf 0}&{\bf 0}\\
	1&{\bf 0.}42857142857142857& {\bf 0.09}8901098901098901 & {\bf 0.0277}49229188078108\\
	2&{\bf 0.5}4831147767352699& {\bf 0.099887}063059789885 & {\bf 0.0277706291}77403332\\
	3&{\bf 0.557}09091792164806& {\bf 0.09988716235706}3362 & {\bf 0.027770629192262428}\\
	4&{\bf 0.55713917}646743778& {\bf 0.099887162357064255} & {\bf 0.027770629192262428}\\
	5&{\bf 0.55713917793231456}& {\bf 0.099887162357064255} & {\bf 0.027770629192262428}\\
	\hline
\end{tabular}
\end{center}
\end{example}
The rest of this section describes ways to compute this iteration in practice. 

\paragraph{Hybrid Method}

We now have two convergent iterations at our disposal: one on truncated power series and one on infinite sequences of values at powers of~$\alpha$. Our method below consists in combining these iterations in parallel so as to get rid of infinite sequences, using the iteration of Theorem~\ref{th:newt_num_ord} with
power series to evaluate values at high exponents of~$\alpha^i$, where they converge fast.
It depends on a threshold~$K$ to switch the use of power series. This is illustrated in Figure~\ref{fig-hybrid-algo}.

\begin{figure}
\begin{center}
\begin{tikzpicture} [scale=3,line cap=round]
\def\aa{0.8}
        \draw (0,0) circle (.88);
        \draw[color=black!20, fill=black!20] (0,0) circle (0.29);
        \draw[->] (-1,0) -- (1,0) node[right] {}; 
        \draw[->] (0,-1) -- (0,1) node[above] {};
        \foreach \x in {1,...,40}
                \pgfmathpow{\aa}{\x}
                \draw[xshift= \pgfmathresult cm] (0pt,1pt) -- (0pt,-1pt) node[below,fill=white] {};
        \draw[->] (0,0) -- (xyz polar cs:angle=45,radius=0.88) node [midway,above]{$\rho$};
        \draw (0.34, 0.11) node[rectangle, scale=.7,inner sep=0.5pt]{$\alpha^{K+1}$};
        \draw (0.57,-0.11) node[scale=.7]{$\alpha^K\;\dotsc\;\;\alpha^2\;\;\;\alpha$};
\end{tikzpicture}
\end{center}
\caption{Hybrid method: inside the gray disk, the evaluation uses a truncated power series\label{fig-hybrid-algo}}
\end{figure}
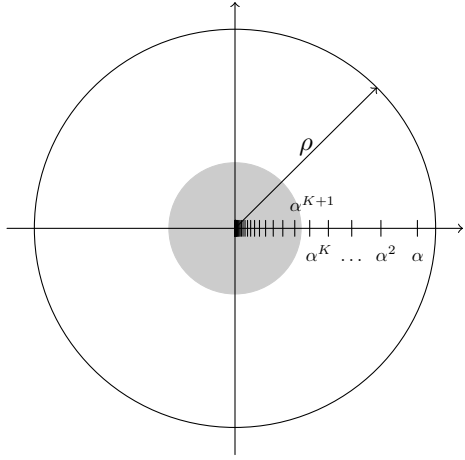

The method applies to a species $\bc F$, with Pólya operator~$\bs{\Phi}_{\bc F}$ and associated sequence  transformer~$\bs{\Psi}_{\bc F}$. The ordinary generating series ${{\bd{\tilde F}}(z)}$ has radius of convergence $\rho$, and the numerical evaluation concerns~$\alpha<\rho$. The input of the method consists of
\begin{enumerate}[i.]\setlength{\itemsep}{0pt}
	\item a $K$-tuple  $(\bs{{y}}^{[n]}_1,\dots,\bs y^{[n]}_K)$ of approximations of 
	$(\bs Y^{[n]}(\alpha),\dots,\bs Y^{[n]}(\alpha^K))$; 
	\item a truncation $\bs T_m^{[n]}(z):=\bs Y^{[n]}(z)\bmod z^m$ of the Taylor series of~$\bs Y^{[n]}$ at order $m$ (which is used to evaluate values of $(\bs y^{[n]}_k)$ for indices ${k> K}$);
	\item a truncation order $M$ on series;
	\item a precision~$\epsilon>0$.
\end{enumerate}
The idea is to compute
\[(\bs y^{[n+1]}_k)_{k\in\N^\star}\simeq\bs{\Psi}_{\bc N_{\bc H}}\left(\alpha, (\bs y^{[n]}_1,\dots,\bs y^{[n]}_K,\bs T_m^{[n]}(\alpha^{K+1}),\bs T_m^{[n]}(\alpha^{K+2}),\dotsc)\right).\]
In practice, we cannot apply $\bs{\Psi}_{\bc N_{\bc H}}$ to an infinite sequence, thus we need to determine the number $L$ of nonzero terms that are necessary to compute the sequence $(\bs y^{[n+1]}_k)$ at precision~$\epsilon$.
The output of the method is the vector $(\bs y^{[n+1]}_{1:K})$, and the truncated series $\bs Y^{[n+1]}(z)\bmod z^M$. When using a Pólya operator given by Newton's iteration, this method is an algorithm if we use~$m=2^n$ and $M=2^{n+1}$ by Proposition~\ref{th:newt_sg_trunc}. These bounds and the choice of~$K$ and~$L$ are discussed below. 

\medskip
\noindent{\bf Method:} Numerical evaluation of ordinary generating series.\\
{\bf Input:} 
$(\bs y^{[n]}_{1:K})$;
 $\bs T^{[n]}_m(z):=\bs Y^{[n]}(z)\bmod z^m$; $M\in\mathbb N$; $\epsilon>0$\\
{\bf Output:} an approximation of $(\bs y^{[n+1]}_{1:K})$; $\bs Y^{[n+1]}(z)\bmod z^{M}$.
\begin{enumerate}
	\item Compute a bound~$L$ such that 
\[\left\|\bs{\Psi}_{\bc F}\left(\alpha, (\bs y^{[n]}_k)_{k\in\N^\star})\right)
-\bs{\Psi}_{\bc F}\left(\alpha, (\bs y^{[n]}_1,\dots,\bs y^{[n]}_L,0,0,\dotsc)\right)\right\|<\epsilon;\]
		\item Compute $(\bs y^{[n+1]}_{1:K})$ the first $K$ values of the sequence \[\bs{\Psi}_{\bc F}\left(\alpha, (\bs y^{[n]}_1,\dots,\bs y^{[n]}_K,\bs T^{[n]}_m(\alpha^{K+1}),\dots,\bs T^{[n]}_m(\alpha^{L}),0,0,\dotsc)\right);\]
\item Compute $\bs Y^{[n+1]}(z):={\Phi}_{\bc F}(\bs Y^{[n]}(z)\bmod z^m)\bmod z^M$;
\item {\bf Return} $\bs Y^{[n+1]}(z)\bmod z^M$ and $(\bs y^{[n+1]}_{1:K})$.
\end{enumerate}

\paragraph{Numerical Evaluation of Pólya operators}
The first step of this hybrid method is to find a bound for the truncation of the Pólya operators. We now turn to the computation of such a bound.
We deal with Sets and Cycles. Other cases can be treated in a similar way.
\begin{lemma}\label{numer_polya}Let ${\mathcal Y}$ be an analytic species such that~${\mathcal Y}(0)=0$. Let~$0\le\alpha<1$ be inside the disk of convergence of the ordinary generating series~${\tilde{Y}}$ of ${\mathcal Y}$. Let $\tilde{S}(z)$ (resp. $\tilde{C}(z)$) be the ordinary generating series of $\Set({\mathcal Y})$ (resp. $\Cyc({\mathcal Y})$). Then $\tilde{S}(z)$ converges at~$\alpha$ and for any~$L>0$, the following inequalities hold
\[s_L\le \tilde{S}(\alpha)\le s_L\exp\!\left(\frac{\tilde{Y}(\alpha^L)}{\alpha^L}\sum_{i\ge L}{\frac{\alpha^i}{i}}\right)\qquad\text{with}\qquad s_L=\exp\!\left(\sum_{i=1}^{L-1}{\frac{\tilde{Y}(\alpha^i)}{i}}\right).\]
If moreover $\tilde{Y}(\alpha)<1$, then $\tilde{C}(z)$ converges at~$\alpha$ and for any $L>0$
\[c_L\le \tilde{C}(\alpha)\le c_L+\frac{\tilde{Y}(\alpha^L)}{\alpha^L}\frac{\alpha^L}{1-\alpha}\qquad\text{with}\qquad c_L=\sum_{i=1}^{L-1}{\frac{\phi(i)}{i}\log\frac{1}{1-\tilde{Y}(z^i)}}.\]
\end{lemma}
The argument of the exponential in the right-hand side of the first inequality tends to~0 as $L$ tends to infinity, so that these inequalities can be used to compute $\tilde{S}(\alpha)$ and~$\tilde{C}(\alpha)$ with arbitrary precision, provided one can compute values of $\tilde{Y}$ at the powers of~$\alpha$.
\begin{proof}
Table~\ref{tab:sum_esp_sg} gives
\[\log(\tilde{S}(\alpha)/s_L)=\sum_{i\ge L}{\frac{\tilde{Y}(\alpha^i)}{i}}.\]
Since ${\mathcal Y}(0)=0$, the generating series $\tilde{Y}(z)$ satisfies $\tilde{Y}(0)=0$ and therefore $\tilde{Y}(z)/z$ itself is a power series with positive coefficients. It follows that for any $\beta$ such that $0\le\beta\le\alpha^L$, $\tilde{Y}(\beta)/\beta\le \delta$, with ${\delta:=\tilde{Y}(\alpha^L)/\alpha^L}$. Using this inequality for $\beta=\alpha^i$, $i=L, L+1,\dotsc$ gives the result for $\tilde{S}(\alpha)$.

In the case of $\Cyc$, the same reasoning leads to the sequence of inequalities
\begin{align*}
0\le \tilde{C}(\alpha)-c_L&=\sum_{i\ge L}{\frac{\phi(i)}{i}\log\frac{1}{1-\tilde{Y}(\alpha^i)}}\le\sum_{i\ge L}{\frac{\phi(i)}{i}\log\frac{1}{1-\delta \alpha^i}},\\
&\le\sum_{i\ge L,j>0}{\frac{\phi(i)}{i}\delta\frac{\alpha^{ij}}{j}}
\le\sum_{n\ge L}{\delta \left(\sum_{i\mid n}{\phi(i)}\right)\frac{\alpha^n}{n}}\le \delta \frac{\alpha^L}{1-\alpha},
\end{align*}
where we have used the classical identity $\sum_{i\mid n}\phi(i)=n$.
\end{proof}

\begin{example} In order to compute the value of 
	\[\tilde{S}(z)=\exp(\tilde{Y}(z)+\frac12\tilde{Y}(z^2)+\dotsb)\]
 at $\alpha=.3$ when $\tilde{Y}(z)=z+z^2$, we compute the first few values of 
\[\exp\left(\frac{\tilde{Y}(\alpha^L)}{\alpha^L}(\ln\frac{1}{1-\alpha}-\sum_{i<L}\frac{\alpha^i}{i})\right)-1,\] which gives
\[0.59, 0.064, 0.012, 0.0027, 0.00065, 0.00016,0.000042, 0.000011, 0.0000030, 0.00000081\]
showing that an error smaller than~$10^{-6}$ is achieved as soon as~$L=10$.
By monotonicity, this bound is also sufficient for the evaluation at the smaller~$\alpha^i$, $i>1$.

This is used to compute the numerical iteration for Cayley trees; starting from the row~$n=2$ in Example~\ref{numer_Co}, where the truncation of~$Y^{[2]}(z)$ is~$z+z^2$, the hybrid method then proceeds as follows: the values of~$z+z^2$ at $0.3^i$ for $i=4,\dots,10$ are first computed and appended to the estimates of~$Y^{[2]}(0.3^i)$, $i=1,\dots,3$, which leads to the 10-tuple of values 
\[{c}=(0.5483115,0.0998871,0.0277706,0.008165,0.0024359,\dots,0.0000059).\]
These values are then used to compute the values of $\bs{\Psi}_{\bc F}\left(0.3, c,0,0,\dotsc)\right)$ at the second step of the hybrid method. In particular, the necessary values of~$B(\alpha^i)$ are obtained by
\begin{align*}
	B(\alpha)&\simeq\alpha\exp(c_1+c_2/2+\dots+c_{10}/10)\\
	B(\alpha^2)&\simeq\alpha^2\exp(c_2+c_4/2+\dots+c_{10}/5)\\
	B(\alpha^3)&\simeq\alpha^3\exp(c_3+c_6/2+c_9/3)\\
	B(\alpha^4)&\simeq\alpha^4\exp(c_4+c_8/2)\\
	B(\alpha^5)&\simeq\alpha^5\exp(c_5+c_{10}/2)\\
	B(\alpha^i)&\simeq\alpha^i\exp(c_i),\quad i=6,\dots,10.
\end{align*}
\end{example}

\paragraph{Truncation orders}
In order to compute values of~$\bs Y^{[n]}(z)$, we use a truncation of this power series. 
Although the order $2^n$ is valid in a Newton iteration for~$\bs Y^{[n]}(z)$, it is desirable for efficiency reasons to use fewer terms if possible. We now discuss the choice of the order of truncation in terms of the desired accuracy. 

This is achieved using dominant species.
If~$\hatH$ is a flat constructible species dominating~$\bc H$, and $\,\bc U$ is defined by~$\,\bc U=\hatH({\mathcal Z},\bc U)$, then $\,\bc U$ is flat and values of its generating series can be computed by the iteration for exponential generating series. This value is then used to bound truncation orders for subspecies of~$\bc S$ thanks to the following result.
\begin{lemma}\label{tail-dominant}
	Let $\bs U(z)=\bs u_0+\bs u_1z+\dotsb$ and $\bs V(z)=\bs v_0+\bs v_1z+\dotsb$ be the generating series of two species $\,\bc U$ and~$\bc V$, such that~$\bc V\lhd\bc U$. Assume that the series~$\bs U$ converges at ${r}\ge0$. Then for any~$\alpha$ with~$|\alpha|\le{r}$, \[\|\bs U({r})-(\bs u_0+\bs u_1{r}+\dots+\bs u_N{r}^N)\|<\epsilon\Rightarrow
	\|\bs V(\alpha)-(\bs v_0+\bs v_1\alpha+\dots+\bs v_N\alpha^N)\|<\epsilon.\]
\end{lemma}
\begin{proof}
The coefficients of the species and its subspecies are related by~$0\le \bs v_n\le \bs u_n$. Therefore,
\[\|\sum_{n>N}{\bs v_n\alpha^n}\|\le\|\sum_{n>N}{\bs v_n{r}^n}\|\le
\|\sum_{n>N}{\bs u_n{r}^n}\|\le\epsilon. \qedhere\] 
\end{proof}

This technique yields the following truncation algorithm.

\medskip
\noindent{\bf Algorithm:} Truncation order\\
{\bf Input:} $\widehat{\bc H}$, $\rho$, $\epsilon$, with $\rho\ge0$ inside the disk of convergence of~$\bs U$\\
{\bf Output:} A bound on the truncation orders for generating series of subspecies of $\bc S$ for~$|\alpha|\le\rho$
\begin{enumerate}
	\item Compute $\bs R:=\bs U(\rho)$ by Theorem~\ref{th:newtonnum}.
\item $i:=0;$
\item {\bf Repeat:}
\begin{itemize}
	\item[] $i:=i+1;$ Compute $[z^i]\bs U$ by Theorem~\ref{th:newt_sg};
	\item[] $\bs R:=\bs R-[z^i]\bs U|\rho|^i;$
\end{itemize}
\item {\bf Until} $\|\bs R\|<\epsilon$
\item {\bf Return} ${i}$.
\end{enumerate}

\begin{example} By Example~\ref{ex:dom2}, the species of Cayley trees (nonordered trees) is dominated by the species of Catalan trees. For $\rho=1/10$, Example~\ref{ex:catalan_num} gives $\bs U(\rho)\simeq0.112701665379258311482$, while Example~\ref{ex:catalan_series} gives the first terms of the series, showing that 9 terms are sufficient to guarantee a precision $\epsilon=10^{-6}$ for the computation, since
	\[\bs U(\rho)-(\rho+\rho^2+2\rho^3+5\rho^4+14\rho^5+42\rho^6+132\rho^7+429\rho^8+1430\rho^9)\simeq 0.74\cdot 10^{-6}.\]
It follows that the computation of~$\tilde{C}$ at precision~$\epsilon$ at any~$\alpha\le1/10$ can be performed using~$M=10$ (9~terms) in the hybrid method.
\end{example}

If Newton's iteration is used, the computation of $\bs[z^i]{U}$ computes several coefficients at once. Obviously, they are all used in the next step to improve the precision of $\bs R=\bs U(\rho)-(\bs U_0+\dots+[z^i]\bs U\rho^i)$.
The termination of this algorithm follows from the convergence of the series $\bs U$ at $\rho$ which implies that $\|\bs R\|\rightarrow0$ as $i\rightarrow\infty$.

If $N$ is the bound computed by this algorithm, 
our hybrid method can be used with~$k$ chosen in such a way that $\alpha^{k+1}\le\rho$ and $M=\min(2m,N)$. By Lemma~\ref{tail-dominant}, all the evaluations of the series~$\bs Y^{[n]}(z)$ at~$\alpha^i$ with $i>k$ then have an error bounded by~$\epsilon$. 

\begin{example} To compute the numerical values of the series $\tilde{G}$ of Cayley trees at $\alpha=3/10$, we use their domination by Catalan trees.  While $3/10$ is larger than the radius $1/4$ of convergence of the generating series of Catalan trees, the computation of~$\tilde{G}(3/10)$ can be done by the hybrid method, as soon as~$K>1$. With~$K=2$ and a precision~$10^{-20}$ we get $M=16$ while $K=4$ leads to a smaller $M=7$. With these values the numerical computation of the Pólya operator for sets uses at most 36~terms of the sum in Lemma~\ref{numer_polya} and leads to the values given in Example~\ref{numer_Co}. 
\end{example}

\paragraph{Optimized Newton iteration} The preceding methods for numerical evaluation also apply in the case of the optimized Newton iteration.
We just give some examples.

\begin{example} The exponential generating series of Cayley trees is evaluated in Example~\ref{ex:num-cayley} by the simple Newton iteration:
	\[g^{[n+1]}=g^{[n]}+\frac{\alpha e^{g^{[n]}}-g{[n]}}{1-\alpha e^{g^{[n]}}}.\]
The optimized iteration is
\[t^{[n+1]}=t^{[n]}+u^{[n+1]}(\alpha e^{t^{[n]}}-t^{[n]}),\quad u^{[n+1]}=u^{[n]}+u^{[n]}(\alpha e^{t^{[n]}}u^{[n]}-u^{[n]}+1),\]
both with initial value~$t^{[n]}=0$, and $u^{[0]}=1$; the latter one trades one division for one multiplication. For $\alpha=1/10$, the first few values are: \begin{align*}
&\mathbf{0.11}000000000000000000,\\
&\mathbf{0.11183}139629772798888,\\
&\mathbf{0.111832559158}22437627,\\
&\mathbf{0.11183255915896296483}. 
\end{align*}
The convergence is mildly slower, but still quadratic and the complexity is slightly smaller. The real gain of this method is for systems, where inverses of matrices can become expensive. 
\end{example}     

\begin{example} For series-parallel graphs, the optimized iteration is given in the introduction to this article.	
\end{example}

\section{Extensions}\label{sec:extend}
In this section, we discuss two extensions of our work.

%
In the first part, we generalize and extend our results to combinatorial systems including integrals. We define well-founded integral systems, that possess a unique linear species solution and we adapt Newton's iteration to their resolution, with quadratic convergence. From there, we compute the generating series with quasi-optimal arithmetic complexity. This framework also provides an efficient way to compute with equations involving $\Set $ and $\Cyc$, by expressing them in terms of simple differential equations (without exponentials or logarithms).

The second extension deals with the construction \emph{Powerset}, which is not a species \emph{per se}, but is amenable to Newton iteration nonetheless.

\subsection{Linear Species and Differential Equations}
Combinatorial differential equations can be stated for classical species, but such a simple equation as ${\mathcal Y}'={\mathcal Z}\Set({\mathcal Z})$, ${\mathcal Y}(0)=0$ does not have a solution, and others may have an infinite number of solutions  (this is proved by using a decomposition in molecular species~\cite{Labelle86b}). A nicer framework is provided by \emph{linear species}, that make use of a linear order on the underlying set, and it is possible in this framework to define structures such as alternating permutations or increasing rooted trees.       

A concise presentation of linear species is given in the last chapter of~\cite{BeLaLe98}. We summarize here the salient points.
Compared to Definition~\ref{def:species}, the differences are that the finite sets~$U$ are endowed with a total order and the bijections preserve orders. The automorphisms become trivial and there are no unlabeled structures. Only exponential generating series are used and two linear species are isomorphic if and only if their exponential generating series are identical. Any classical species gives rise to a linear species by restricting it to totally ordered sets and increasing bijections, and we keep the same name and notation for sets, cycles,\dots. Composition, derivative, as well as sum, product and more general multisort linear species are defined by paying attention to the order. With the new definitions for these operations, the exponential generating series still satisfy the same equations as before.

A new important operation  on linear species is \emph{combinatorial integration}, defined by
\[\left(\int_0^{{\mathcal Z}}{{\mathcal F}}({\mathcal T})d{\mathcal T}\right)[U]=\begin{cases}\eset\quad&\text{if $U=\eset$,}\\ {{\mathcal F}}[U\setminus\{\min(U)\}],\quad&\text{otherwise.}\end{cases}\]
A graphical representation of an $\int{{\mathcal F}}$-structure is given by Figure~\ref{fig:def_int}. The exponential generating series of~$\int_0^{{\mathcal Z}}{\mathcal F}({\mathcal T})d{\mathcal T}$ is $\int_0^zF(t)\,dt$. Indeed, the number of structures of size $n+1$ (including the minimum) in the combinatorial integration is $f_n$; thus the exponential generating function is $\sum  f_n \frac{z^{n+1}}{(n+1)!}= \int \sum f_n \frac{z^n}{n!}$.

\begin{figure}[t]
        \centering
                \includegraphics[width=0.2\textwidth]{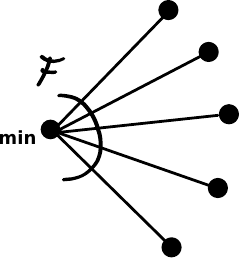}
        \caption{Representation of an $\int{{\mathcal F}}$-structure.}
        \label{fig:def_int}
\end{figure}

\begin{example} Alternating permutations. Any permutation~$\sigma$ can be written as a word whose $i$th letter is $\sigma(i)$. One can then associate a tree to $\sigma$, with $\min(\sigma)$ at the root and the left and right subtrees constructed from the factors on the left and on the right of $\min(\sigma)$. In this bijection, alternating permutations, i.e., of the form $\sigma(1)>\sigma(2)<\sigma(3)>\sigma(4)<\dotsb$ are in correspondence with binary increasing trees (binary labeled trees whose labels increase along each branch). These trees satisfy the equation~${{\mathcal B}}={{\mathcal Z}}+\int{{{\mathcal B}}^2}$. Their exponential generating series satisfies~$B'=1+B^2$ and $B(0)=0$, so that $B(z)=\tan(z)$. 
\end{example}

The  following two examples show how to express $\Cyc$ and  $\Set$ in terms of differential equations. These expressions are used, in particular, to get the complexity of constructible species in a self-contained way, in Proposition~\ref{prop:constructible}.

\begin{example} Cycles. 
        The equation for $\Cyc(\bc A)$ is obtained  as the integral of its derivative (see Table~\ref{tab:deriv}):
\[\Cyc(\bc A({\mathcal Z}))=\int_0^{{\mathcal Z}}{\Seq(\bc A({\mathcal T}))\bc A'({\mathcal T})\,d{\mathcal T}}.\]
Thus the computation of the generating series is obtained simply by expanding
\[\int_0^{z}{\frac{A'(t)}{1-A(t)}\,dt}.\]
This is a special case of the classical way of computing the power series expansion of $\log(u)$ as the integral of its derivative.
\end{example}

\begin{example} Sets. When viewed as linear species, sets satisfy an integral equation, again obtained as the integral of their derivative. If $\bc A$ is a species such that $\bc A(0)=0$, then the species $\Set(\bc A)$ satisfies
	\[\Set(\bc A({\mathcal Z}))= 1+\int_0^{{\mathcal Z}}{\Set(\bc A({\mathcal T}))\bc A'({\mathcal T})\,d{\mathcal T}}.\]
The algorithms developped in this section solve such equations and thus provide fast ways of computing the ``$\exp$'' needed in the computation of the exponential generating series for $\Set$, and we will show more generally how to compute the composition $\exp(u)$ where $u$ is a power series, so that the ordinary generating series of $\Set$ can be computed efficiently too.
\end{example}

\medskip
The rest of this section is devoted to the resolution of integral systems.
The first step is to  define general combinatorial  integral systems over linear species, and find  suitable conditions
 ensuring that they have a solution which is unique up to isomorphism : this is done in \ref{sec:wfis}, relying on previous work by Leroux and Viennot. Then in \ref{sec:niis}, generalizing a result by Labelle, we define Newton's Iteration for integral system and show that the solution of the linearized system doubles the precision. Newton's Iteration reduces a non linear problem to a linear integral system. In order to solve linear integral system, we propose in \ref{subseq:varconst} a  Variation of the Constant method, which applies to homogeneous  and non homogeneous, and finally results in an efficient algorithm for the resolution of general well-founded integral systems by Newton's Iteration in \ref{subseq:isni}. Some examples are given in  \ref{subseq:ex}, showing that specifications containing $\Set $ and $\Cyc$ can be transformed into integral systems with rapid enumeration algorithms.

\subsubsection{Well-founded Integral Systems}\label{sec:wfis}
With the combinatorial integration operator, integral systems over linear species can be defined. It turns out that they are all ``well founded'' in the sense that they define unique linear species, and we show a general implicit theorem for integral systems in \ref{th:well_founded_integral_system}. Our result relies on a proposition due to Leroux and Viennot, concerning purely integral systems:
\begin{property}{\cite{LerouxViennot1986}}\label{prop:LV86}
For arbitrary linear species ${\mathcal W}_1, \dots, {\mathcal W}_m$, the system
\begin{equation}\label{diff_sys}
{\mathcal Y}_i({\mathcal Z},\bc X)={\mathcal X}_i+\int_0^{{\mathcal Z}}{{\mathcal W}_i({\mathcal T},\bc Y({\mathcal T},\bc X))d{\mathcal T}},\quad i=1,\dots,m
\end{equation}
has a linear species solution that is unique up to isomorphism.
\end{property}
The solution is called a $\bc W$-enriched increasing rooted tree.
Many other systems, including higher-order differential or integral equations, can be reduced to the form of the theorem by adding new variables. 
This property shows that, thanks to the integral, all systems of type~\eqref{diff_sys} are unconditionally well founded. 
 
We now extend this result in two directions: first we generalize these integral systems to systems admitting a non-integral part, so that the implicit species theorem is also a special case; next, we consider well-founded systems of that type.
\begin{proposition} Let~$\bc H_{1:m}$ and~$\bc G_{1:m}$ be vectors of $m+1$-sort linear species such that
$\bc H(0,\bs{0})=\bs{0}$ and $\bs{\partial}\bc H/\bs{\partial}\bc Y(0,\bs{0})$ is nilpotent. Then the system
\[\bc Y({\mathcal Z})=\bc H({\mathcal Z},\bc Y({\mathcal Z}))+\int_0^{{\mathcal Z}}{\bc G({\mathcal T},\bc Y({\mathcal T}))\,d{\mathcal T}}\]
has a linear species solution~$\bc S$ such that~$\bc S(0)=\bs{0}$, that is unique up to isomorphism.
\end{proposition}

Instead of proving this result, we proceed directly to the following more general one that encompasses the general implicit species theorem as well. Apart from technical conditions ensuring that all compositions with species that are not~0 at~0 are possible, the idea consists in reducing to an integral system of the type~\eqref{diff_sys}.
\begin{theorem}[Well-founded integral systems]\label{th:well_founded_integral_system} Let~$\bc H_{1:m}$ and~$\bc G_{1:m}$ be vectors of $m+1$-sort linear species. The system
\begin{equation}\label{eq:integral_system}
\bc Y({\mathcal Z})=\bc H({\mathcal Z},\bc Y({\mathcal Z}))+\int_0^{{\mathcal Z}}{\bc G({\mathcal T},\bc Y({\mathcal T}))d{\mathcal T}}
\end{equation}
admits a linear species solution~$\bc S$ such that~$\bc S(0)=\bc H^m(0,\bs{0})$, that is unique up to isomorphism when the following two conditions hold:
\begin{enumerate}
	\item $\bs{\partial}\bc H/\bs{\partial}\bc Y(0,\bs{0})$ is nilpotent;
	\item if $\bc H(0,\bs{0})\neq\bs{0}$ then
	\begin{enumerate}
 		\item $\bc H(0,\bc Y)$ is polynomial and $\bs{\partial}\bc H/\bs{\partial}\bc Y(0,\bc Y)$ is nilpotent;
		\item if~$\bc R$ is the solution of~$\bc Y=\bc H(0,\bc Y)$ then for all~$i\in\{1,\dots,m\}$, either ${{\mathcal R}}_i=0$ or the solution of~\eqref{diff_sys} is polynomial in~${{\mathcal X}}_i$, with $\bc W$ solution of the implicit system
\begin{equation}\label{def_W}
\bc Q=\frac{\bs{\partial}\bc H}{\bs{\partial}\bc Y}({\mathcal Z},\bc Y)\cdot\bc Q+\frac{\bs{\partial}\bc H}{\partial{\mathcal Z}}({\mathcal Z},\bc Y)+\bc G({\mathcal Z},\bc Y).
\end{equation}
	\end{enumerate}
\end{enumerate}
A system satisfying these conditions will be called \emph{well-founded}.
\end{theorem}
\begin{proof}
Differentiating~\eqref{eq:integral_system} (see~\eqref{eq:diffcomp}) gives
\[\bc Y'({\mathcal Z})=\frac{\bs{\partial}\bc H}{\bs{\partial}\bc Y}({\mathcal Z},\bc Y)\cdot\bc Y'(\mathcal Z)+\frac{\bs{\partial}\bc H}{\partial{\mathcal Z}}({\mathcal Z},\bc Y)+\bc G({\mathcal Z},\bc Y),\]
which shows that the derivative of the solution satisfies~\eqref{def_W}.
Since~$\bs{\partial}\bc H/\bs{\partial}\bc Y(0,\bs{0})$ is nilpotent, the implicit species theorem~\ref{th:IST} implies that the related linear system
\begin{equation}\label{eqU}
\bc U(\mathcal Z,\bc X,\bc Y)={\bs{\partial}\bc H}/{\bs{\partial}\bc Y}({\mathcal Z},\bc Y)\cdot\bc U(\mathcal Z,\bc X,\bc Y)+\bc X,
\end{equation}
with $\bc X$ a vector of sorts, is well founded. Its solution is polynomial (even linear!) in~$\bc X$, so that the composition with $\bc X=\frac{\bs{\partial}\bc H}{\partial{\mathcal Z}}({\mathcal Z},\bc Y)+\bc G({\mathcal Z},\bc Y)$ is possible, giving a solution~$\bc W(\mathcal Z,\bc Y)$ to~\eqref{def_W}. Moreover, such an equation has a unique solution, since Eq.~\eqref{eqU} with~$\bc X=\bs{0}$ is well founded. Thus~$\bc Y'=\bc W(\mathcal Z,\bc Y)$ and
the system~\eqref{eq:integral_system} rewrites
\[\bc Y({\mathcal Z})=\bc Y(0)+\int_0^{{\mathcal Z}}{\bc W({\mathcal T},\bc Y({\mathcal T}))\,d{\mathcal T}}.\]

By Property~\ref{prop:LV86}, the same system with the sorts~$\bc X_{1:m}$ in place of~$\bc Y_{1:m}(0)$ has a linear species solution~$\bc V({\mathcal Z},\bc X)$.
Now, if~$\bc S({\mathcal Z})$ is a solution of~\eqref{eq:integral_system}, then~$\bc S(0)$ satisfies the equation~$\bc Y=\bc H(0,\bc Y)$. This equation is well founded as a consequence of condition \emph{2.(a)} and Theorem~\ref{th:WF}. Thus~$\bc S(0)$ is its solution, up to isomorphism.
Finally, condition \emph{2.(b)} ensures that~$\bc V(\mathcal Z,\bc X)$ can be composed with~$\bc X=\bc S(0)$, giving the solution~$\bc S({\mathcal Z})=\bc V({\mathcal Z},\bc S(0))$. 

Uniqueness up to isomorphism is a consequence of the uniqueness at each stage of the construction.
\end{proof}
                                                                                                           
\subsubsection{Newton's Iteration}\label{sec:niis}
We now consider Newton's Iteration by linearizing integral systems.
Again we start with the purely integral system~\eqref{diff_sys}. Newton's iteration has been lifted to this combinatorial level  by Labelle:
\begin{property}[\cite{Labelle86b}]\label{prop:newtondiff}          
Let $\bc S$ be the solution of~\eqref{diff_sys}. If a subspecies~$\bc A\subset\bc S$ has contact of order~$k$ with~$\bc S$, then the solution of the linearized system
\begin{equation}\label{newton_diff}
\bc B=\left(\int_0^{{\mathcal Z}}{\frac{\bs\partial\bc W}{\bs\partial\bc Y}(\bc A)\cdot \bc B}\right)+\bc X+\left(\int_0^{{\mathcal Z}}{\bc W(\bc A)}\right)-\bc A
\end{equation}
is such that $\bc A+\bc B$ has contact of order~$2k$ with~$\bc S$. 
\end{property}

The following result is a generalization of this property to systems that are well founded in the sense of Theorem~\ref{th:well_founded_integral_system}.
\begin{theorem}[Newton's iteration for integral systems]\label{th:newt_diff}
Let $\bc S$ be the solution of the well-founded system
\begin{equation}\tag{\ref{eq:integral_system}}
\bc Y({\mathcal Z})=\bc H({\mathcal Z},\bc Y({\mathcal Z}))+\int_0^{{\mathcal Z}}{\bc G({\mathcal T},\bc Y({\mathcal T}))d{\mathcal T}}.
\end{equation}
If a subspecies~$\bc A\subset\bc S$ has contact of order~$k$ with~$\bc S$, then the solution of the linearized system
\begin{equation}\label{eq:lin_int_sys}
	\bc B({\mathcal Z},\bc X)=\bc X+\int_0^{{\mathcal Z}}{\frac{\bs\partial\bc G}{\bs\partial\bc Y}({\mathcal T},\bc A({\mathcal T}))\left(\Id-\frac{\bs\partial\bc H}{\bs\partial\bc Y}({\mathcal T},\bc A({\mathcal T}))\right)^{-1}\bc B({\mathcal T},\bc X) \,d{{\mathcal T}}}
\end{equation}
is such that $\bc A({\mathcal Z})+
\left(\Id-\frac{\bs\partial\bc H}{\bs\partial\bc Y}({\mathcal Z},\bc A({\mathcal Z}))\right)^{-1}
\bc B\left({\mathcal Z},\bc H({\mathcal Z},\bc A)-\bc A+\int{\bc G(\bc A)}\right)$ has contact of order $2k$ with~$\bc S$.
\end{theorem}
Note that setting~$\bc H=\bc X$ in this theorem yields Property~\ref{prop:newtondiff}, while setting $\bc G=0$ produces Theorem~\ref{th:newton}.
\begin{proof}
The proof follows the same reasoning as that of Lemma~\ref{lem:quad-cvg}. Instead of decomposing an $\bc S$-structure of size at most~$2k$ that is not an $\bc A$-structure, we write directly contacts between species:
\[\bc S=_{2k}\bc H({\mathcal Z},\bc A)+\frac{\bs\partial\bc H}{\bs\partial\bc Y}({\mathcal Z},\bc A)\cdot(\bc S-\bc A)
+\int_0^{{\mathcal Z}}{\bc G({\mathcal T},\bc A({\mathcal T}))\,d{\mathcal T}}+\int_0^{{\mathcal Z}}{\frac{\bs\partial\bc G}{\bs\partial\bc Y}({\mathcal T}, \bc A({\mathcal T}))\cdot(\bc S({\mathcal T})-\bc A({\mathcal T}))\,d{\mathcal T}}.\]
Indeed further terms in the Taylor expansions of $\bc H$ and $\bc G$ are supported only by structures of size at least~$2k+2$, which reduces to $2k+1$ by integration. Now, as in Lemma~\ref{lem:quad-cvg}, a structure in $\bc S-\bc A$ starts with a string of $\bs\partial\bc H/\bs\partial\bc Y(\bc A)$. Thus~$\bc S-\bc A$ decomposes as $(\Id-\bs\partial\bc H/\bs\partial\bc Y(\bc A))^{-1}\cdot\bc V$ for a species $\bc V$ that satisfies
\[\bc V=_{2k}\bc H({\mathcal Z},\bc A)-\bc A+\int_0^{{\mathcal Z}}{\bc G({\mathcal T},\bc A)\,d{\mathcal T}}+
\int_0^{{\mathcal Z}}{\frac{\bs\partial\bc G}{\bs\partial\bc Y}({\mathcal T}, \bc A({\mathcal T}))(\Id-\frac{\bs\partial\bc H}{\bs\partial\bc Y}({\mathcal Z},\bc A))^{-1}\bc V({\mathcal T})\,d{\mathcal T}}.\]
This is exactly~\eqref{eq:lin_int_sys}, with~$\bc X$ replaced by $\bc H({\mathcal Z},\bc A)-\bc A+\int{\bc G(\bc A)}$, and the result follows from the uniqueness of Prop.~\ref{prop:LV86}.
\end{proof}

\subsubsection{Virtual Species and Variation of the Constant}\label{subseq:varconst}
Newton's iteration reduces the nonlinear problem~\eqref{diff_sys} 
to the linear integral system~\eqref{newton_diff}.
In the case of \emph{one} equation, this linear integral equation takes the form
\[{{\mathcal B}}={{\mathcal Q}}+\int_0^{{\mathcal Z}}{{\mathcal P}{\mathcal B}}.\]
Leroux and Viennot~\cite{LerouxViennot1986} have shown that its resolution can be reduced further to the solution~${\mathcal V}$ of the homogeneous equation ${{\mathcal V}}'={{\mathcal P}{\mathcal V}}$ by a combinatorial lifting of Lagrange's method of variation of the constant:
\begin{equation}\label{varcte}
{{\mathcal B}}=\Set\left(\int_0^{{\mathcal Z}}{\mathcal P}\right)\cdot\left(\int_0^{{\mathcal Z}}{\Set\left(-\int_0^{{\mathcal T}}{{\mathcal P}}\right)\cdot{\mathcal Q}}\right),
\end{equation}
using $\Set$ in place of the analytical $\exp$ and virtual linear species to account for the minus sign in the inner integral. Virtual species are presented in~\cite[\S2.5]{BeLaLe98}; they are to species what the relative integers are to the natural integers; they let one define the opposite of a species and the reciprocal of a species $\bc F$, the latter when $\bc F(0)=1$; their generating series obey the same rules as those of the usual species. 

We now apply variation of the constant to systems of equations. In this setting, the linear homogeneous system is no longer solved by an exponential (or combinatorially, by a $\Set$) in general. However, fast power series expansion via Newton's iteration is still possible in this context as shown in~\cite{BostanChyzakOllivierSalvySchostSedoglavic2007}. 
We lift the idea of~\cite{BostanChyzakOllivierSalvySchostSedoglavic2007} combinatorially, using virtual species. The key idea is to compute not only a vector solution of the homogeneous equation~$\bc V'=\bc P\cdot\bc V$ used for variation of the constant, but a square matrix solution~$\bc V$ together with its inverse.
The result is summarized in the following statement, where we define the \emph{valuation} of a species $\bc F$ as that of its generating series and denote it $\val(\bc F)$.
\begin{proposition}[Variation of the constant]\label{prop:var_cte}
Let $\bc M$ be a matrix species solution of $\bc Y'=\bc A\bc Y$, with $\bc Y(0)=\Id$ and let $\bc W\subset\bc M$ have contact of order $k$ with $\bc M$. Let $\bc B$ be a species with valuation at least~$\ell$. Then the species~$\bc W\int\bc W^{-1}\bc B$ has contact of order $k+\ell+1$ with the solution~$\bc S$ of~$\bc Y'=\bc A\bc Y+\bc B$.
\end{proposition}
\begin{proof}
First, we observe that the valuation of $\bc A\bc W-\bc W'$ is at least $k$. Indeed,  the difference $\bc M-\bc W$ has valuation at least $k+1$ by definition of contact and valuation. This valuation does not decrease when multiplying by~$\bc A$ on the left, while it decreases by 1 by differentiation.

Next, since $\bc W(0)=\Id$, the valuation of $\bc W$ is~0 and moreover, a virtual species $\,\bc U$ is defined by $\bc S=\bc W\bc U$. Injecting into the differential equation yields
\begin{align*}
\bc W'\bc U+\bc W\bc U'&=\bc A\bc W\bc U+\bc B,\\
\bc U'&=\bc W^{-1}\bc B+\bc W^{-1}(\bc A\bc W-\bc W')\bc U,\\
\bc U&=\int\bc W^{-1}\bc B+\int\bc W^{-1}(\bc A\bc W-\bc W')\bc U.
\end{align*}
The valuation of the second integrand is at least~$0+k+\val(\,\bc U)$, the first term coming from~$\bc W^{-1}$, the next one from $\bc A\bc W-\bc W'$. Consequently, the second integral has valuation at least~$1+k+\val(\,\bc U)$, while $\val(\,\bc U)$ is given by the first integral as being at least $\ell+1$. In summary, 
\[\bc U=_{k+\ell+1}\int\bc W^{-1}\bc B,\]
which implies $\bc S=_{k+\ell+1}\bc W\int\bc W^{-1}\bc B$, concluding the proof.
\end{proof}
As a first application, we get a quadratically convergent iteration to the solution of the homogeneous system itself.
\begin{corollary}\label{cor:homdiffsys}
Let $\bc M$ be a matrix species solution of $\bc Y'=\bc A\bc Y$, with $\bc Y(0)=\Id$ and let $\bc W\subset\bc M$ have contact of order $k$ with $\bc M$. Then 
\[\bc W+\bc W\int{\bc W^{-1}(\bc A\bc W-\bc W')}\]
has contact of order~$2k+1$ with $\bc M$.
\end{corollary}
\begin{proof}
	Set~$\bc M:=\bc W+\bc U$, then $\,\bc U$ satisfies $\,\bc U'=\bc A\bc U+(\bc A\bc W-\bc W')$ and the result follows from the previous proposition.
\end{proof}

\subsubsection{Resolution of Integral Systems by Newton's Iteration} \label{subseq:isni}
At this stage, our aim is to solve~\eqref{eq:integral_system}. We proceed as in the optimized Newton iteration of~\S\ref{subsec:optimized_Newton} by setting up an iteration that converges to its solution~$\bc S$, but also to related quantities. 

The result is the following generalization of both Prop.~\ref{prop:newt_opt} (in the case of linear species) and Prop.~\ref{prop:newtondiff}.
\begin{theorem}[Resolution of Integral Systems]\label{newtdiffsys}
Let $\bc S$ be the solution of the well-founded integral system
\begin{equation}\tag{\ref{eq:integral_system}}
\bc Y({\mathcal Z})=\bc H({\mathcal Z},\bc Y({\mathcal Z}))+\int_0^{{\mathcal Z}}{\bc G({\mathcal T},\bc Y({\mathcal T}))d{\mathcal T}}.
\end{equation}
Let also $\,\bc U$ be the species $(\Id-\bs\partial\bc H/\bs\partial\bc Y({\mathcal Z},\bc S))^{-1}$, $\bc A$ be the species $\frac{\bs\partial\bc G}{\bs\partial\bc Y}({\mathcal Z},\bc S)\bc U$ and 
$\bc M$ be the virtual species solution of $\bc M'=\bc A\bc M$ with~$\bc M(0)=\Id$.
For a positive integer $k$, assume that $\bs u, \bs m, \overline{\bs m}$ have contact of order $\lfloor k/2\rfloor$ with $\,\bc U, \bc M, \bc M^{-1}$ and that $\bs s$ has contact of order $k$ with~$\bc S$. Then
\begin{itemize}
	\item $\tilde{\bs u}:=\bs u+\bs u(\frac{\bs\partial\bc H}{\bs\partial\bc Y}(\bs s)\bs u+\Id-\bs u)$ has contact of order $k$ with $\,\bc U$;
	\item $\tilde{\bs m}:=\bs m+\bs m\int{\overline{\bs m}(\frac{\bs\partial\bc G}{\bs\partial\bc Y}(\bs s)\tilde{\bs u}\bs m-\bs m')}$ has contact of order $k$ with $\bc M$;
	\item $\tilde{\overline{\bs m}}:=\overline{\bs m}+\overline{\bs m}(\Id-\tilde{\bs m}\overline{\bs m})$ has contact of order $k$ with $\bc M^{-1}$;
	\item $\bs s+\tilde{\bs u}\tilde{\bs m}\int{\tilde{\overline{\bs m}}(\frac{\partial\bc H}{\partial{\mathcal Z}}(\bs s)+\frac{\bs\partial\bc H}{\bs\partial\bc Y}(\bs s)\bs s'-\bs s'+\bc G(\bs s))}$ has contact of order $2k$ with $\bc S$.
\end{itemize}
\end{theorem}
\begin{proof}
The result for $\tilde{u}$ is the familiar Newton iteration for matrix inverses of~\eqref{eq:U}. The contact follows from the fact that $\bs s$ and $\bc S$ having contact of order $k$ implies that so do $\frac{\bs\partial\bc H}{\bs\partial\bc Y}(\bs s)$ and $\frac{\bs\partial\bc H}{\bs\partial\bc Y}(\bc S)$.
The result for $\tilde{m}$ is a consequence of Corollary~\ref{cor:homdiffsys}, where again, the use of $\frac{\bs\partial\bc G}{\bs\partial\bc Y}(\bs s)\tilde{\bs u}$ instead of $\frac{\bs\partial\bc G}{\bs\partial\bc Y}(\bc S){\bc U}$ is made possible by their contact being of order~$k$.
The result for $\tilde{\overline{\bs m}}$ is again Newton's iteration for matrix inversion.

The last formula is where we exploit the work done up to now.
Newton's iteration for integral systems (Thm.~\ref{th:newt_diff}) implies that $\bc S-\bs s$ has contact of order~$2k$ with the product $(\Id-\frac{\bs\partial\bc H}{\bs\partial\bc Y}(\bs s))^{-1}\cdot \bs b$ where $\bs b$ is the solution of
\[\bs b'=\frac{\bs\partial\bc G}{\bs\partial\bc Y}(\bs s)(\Id-\frac{\bs\partial\bc H}{\bs\partial\bc Y}(\bs s))^{-1}\bs b+\bs q\quad\text{with}\quad \bs q=\frac{\partial\bc H}{\partial{\mathcal Z}}(\bs s)+\frac{\bs\partial\bc H}{\bs\partial\bc Y}(\bs s)\bs s'-\bs s'+\bc G(\bs s).\]
Due to differentiation, $\bs q$ has contact $k-1$ with its counterpart in~$\bc S$, which is~0. Its valuation is therefore at least $k$. Then by variation of the constant (Prop.~\ref{prop:var_cte}), $\tilde{\bs m}\int{\tilde{\overline{\bs m}}\bs q}$ has contact of order $2k+1$ with $\bs b$. Now, by integration, $\bs b$ has valuation at least $k+1$ and thus $\tilde{\bs u}\bs b$ has contact of order $2k$  with $(\Id-\frac{\bs\partial\bc H}{\bs\partial\bc Y}(\bs s))^{-1}\cdot \bs b$ and therefore with $\bc S-\bs s$, which concludes the proof.
\end{proof}

This theorem translates into Algorithm~\ref{algo:newtonSeriesIntegralSystem}, an analogue of Algorithm~\ref{algo:newtonSeries} p.~\pageref{algo:newtonSeries} for generating series.

\SetKwFunction{sK}{\large sK}
\SetKwFunction{sL}{\large sL}
\SetKwFunction{sG}{\large sG}
\begin{Algorithm}[tb]{0.85\textwidth}   
\SetAlgoRefName{newtonSeriesIntegralSystem}
\caption{Computation of generating series with a given precision\label{algo:newtonSeriesIntegralSystem}}
\DontPrintSemicolon
\Input{Two vectors of species $\bc H$ and $\bc G$, such that $\bc Y({\mathcal Z})=\bc H({\mathcal Z},\bc Y({\mathcal Z}))+\int_0^{{\mathcal Z}}{\bc G({\mathcal Z},\bc Y({\mathcal T}))\,d{\mathcal T}}$ is well founded} 
\Input{An integer $N$}
\Output{The first $N$ terms of the generating series $\bs S(z)$ of the linear species solution}
\SetAlgoNoLine
\Begin{
\SetAlgoVlined
                Set up a procedure \sG: $(\bs Y(z), N)\mapsto\bs G(z,\bs Y(z))\bmod z^N$\;
                Set up a procedure \sJ: $(\bs Y(z), N)\mapsto{\partial\bs H}/{\partial\bs Y}(z,\bs Y(z))\bmod z^N$\;
                Set up a procedure \sK: $(\bs Y(z), N)\mapsto{\partial\bs G}/{\partial\bs Y}(z,\bs Y(z))\bmod z^N$\;
                Set up a procedure \sL: $(\bs Y(z), N)\mapsto{\partial\bs H}/{\partial z}(z,\bs Y(z))\bmod z^N$\;
                \Comment*[r]{ $\bs G, \bs H$ are the generating series of $\bc G,\bc H$}
        
        $\bs U$, $\bs M$, $\overline{\bs M}$, $\bs Y$:= \recSeries($N$)\;
        Return $\bs Y$
        }
 
\medskip

\Function(\recSeries){
\Input{An integer $N$}
\Output{$\bs U$: $\left(\Id-{\bd\partial\bc H}/{\bd\partial\bc Y}(z,\bs S(z))\right)^{-1}\bmod z^{\lceil N/2\rceil}$}
\Output{$\bs M$: solution of $\bs Y'(z)={\bd\partial\bc G}/{\bd\partial\bc Y}(z,\bs S(z))\bs U(z)\bs Y(z)$ with $\bs Y(0)=\Id$, $\mod z^{\lceil N/2\rceil}$}
\Output{$\overline{\bs M}$: $\bs M^{-1}\bmod z^{\lceil N/2\rceil}$}
\Output{$\bs Y$: $\bs S(z) \bmod z^N$}
\SetAlgoNoLine
\Begin{
\SetAlgoVlined
       \lIf{$N = 1$}{$\bs M:=\overline{\bs M}:=\Id;\bs Y:= \bs H^m(0,\bs{0});\bs U:=(\Id-\bs J(0,\bs Y))^{-1}$}\\
        \Else{
                $\bs U,\bs M, \overline{\bs M}, \bs Y:=$ \recSeries($\lceil N/2\rceil$)\;
                $\bs U:= \bs U + \bs U\cdot($\sJ($\bs Y,{\lceil N/2\rceil}$) $\cdot \bs U +\Id-\bs U)) \bmod z^{\lceil N/2\rceil}$\;
		$\bs M:=\bs M+\bs M\int_0^z{\overline{\bs M}(\sK(\bs Y,{\lceil N/2\rceil})\bs U\bs M-\bs M')}\bmod z^{\lceil N/2\rceil}$\;
		$\overline{\bs M}:=\overline{\bs M}+\overline{\bs M}(\Id-\bs M\overline{\bs M})\bmod z^{\lceil N/2\rceil}$\;
                $\bs Y:= \bs Y+ \bs U\bs M\int_0^z{\overline{\bs M}\left(\sL(\bs Y,N)+\sJ(\bs Y,N)\cdot\bs Y'-\bs Y'+\sG(\bs Y,N)\right)} \bmod z^N$ \;
                }
                \Return $\bs U,\bs M, \overline{\bs M}, \bs Y$ 
        }
}
\end{Algorithm}

In terms of complexity, we get the following result.
\begin{proposition}Let $\bc Y({\mathcal Z})=\bc H({\mathcal Z},\bc Y)+\int_0^{{\mathcal Z}}{\bc G({\mathcal T},\bc Y({\mathcal T}))\,d{\mathcal T}}$ be a well-founded system such that $\bc G$, $\partial\bc H/\partial{\mathcal Z}$, $\bs\partial\bc H/\bs\partial\bc Y$, and $\bs\partial\bc G/\bs\partial\bc Y$ are linear species with arithmetic complexity~$C(N)$. Then the generating series of the linear species~$\bc S$ solution of the system can be computed in arithmetic complexity~$O(C(N)+\M(N))$.
\end{proposition}
Note that this is not a generalization of Theorem~\ref{prop:complexityFixedPoint}: while we obtain the desired complexity for the generating series of~$\bc S$, we do not claim that for any linear species~$\bc A$, the generating series of the composition $\bc S(\bc A)$ can be obtained within the same complexity bound.
\begin{proof}
The recursion of Algorithm~\ref{algo:newtonSeriesIntegralSystem} is correct by the previous Theorem. Its complexity~$T(N)$ obeys the recurrence
\[T(N)=T(N/2)+2C(N/2)+2C(N)+K\M(N)+O(N),\]
where $K$ counts the number of products involved. Since $C(N/2)=O(C(N))$, this behaves as in the proof of Theorem~\ref{prop:complexityFixedPoint} and leads to the result.
\end{proof}

The special case of interest is that of constructible species:
\begin{theorem}\label{th:constructible_linear_species}
If $\bc H$ and $\bc G$ are constructible species such that the system~\eqref{eq:integral_system} is well founded, then the solution of this system has a generating series that can be computed in~$O(\M(N))$ arithmetic operations.	
\end{theorem}        
\begin{proof}
This is a consequence of the previous result, the observation that the derivatives of constructible species are constructible and the arithmetic complexity of constructible species (Cor.~\ref{prop:combSystem}).
\end{proof}

\subsubsection{Examples} \label{subseq:ex}
\paragraph{Exponential}
The equation for $\Set(\bc A)$ is
\[\bc Y({\mathcal Z})=1+\int_0^{{\mathcal Z}}{\bc A'({\mathcal T})\bc Y({\mathcal T})\,d{\mathcal T}}.\]
The specialization of Algorithm~\ref{algo:newtonSeriesIntegralSystem} retrieves a recent algorithm for the computation of the power series $\exp(A(z))$ given $A(z)$~\cite{HanrotZimmermann2002}, see also~\cite{Bernstein2004}. In that case no matrices are involved (the dimension is~1), $\bs U=1$ and we are left with the following recursive part:
\begin{equation}\label{newton_exp}
m:=m+m\int_0^z{\overline{m}(A'(t)m-m')\,dt},\quad
\overline{m}:=\overline{m}+\overline{m}(1-m\overline{m}).
\end{equation}
In this case, as in all the linear cases (i.e., when $\bc H$ is a constant), it is not necessary to compute the iteration for~$y$: this iteration reads $y:=y+m\int_0^z{\overline{m}(A'(t)y-y')\,dt}$ so that an easy induction shows that $y=my(0)$.

\paragraph{Cayley Trees}
The generating series of Cayley trees (${\mathcal Y}={\mathcal Z}\Set({\mathcal Y})$) has been computed in Example~\ref{ex:newton_cayley} appealing to an external "$\exp$" for power series. This can be achieved by using the iteration of~\eqref{newton_exp}. Another approach, without exponential, is to use the following system:
\[{\mathcal Y}_1={\mathcal Z}{\mathcal Y}_2,\quad {\mathcal Y}_2=1+\int{{\mathcal Y}_1'{\mathcal Y}_2}.\]
This is not strictly in the format required, since the integral involves a derivative of ${\mathcal Y}_1$. Introducing ${\mathcal Y}_3$ to represent ${\mathcal Y}_1'$, we obtain the following system
\[{\mathcal Y}_1={\mathcal Z}{\mathcal Y}_2, \quad {\mathcal Y}_2=1+\int{{\mathcal Y}_3{\mathcal Y}_2},\quad{\mathcal Y}_3={\mathcal Y}_2+{\mathcal Z}{\mathcal Y}_2{\mathcal Y}_3.\]
It is thus only necessary to compute ${\mathcal Y}_2$ and ${\mathcal Y}_3$ recursively, since ${\mathcal Y}_1$ is then simply recovered by multiplication by~${\mathcal Z}$.
In the notations of Theorem~\ref{newtdiffsys}, we have
\[\bc Y=\begin{pmatrix}{\mathcal Y}_2\\ {\mathcal Y}_3\end{pmatrix},\quad \bc H=\begin{pmatrix}1\\ {\mathcal Y}_2+{\mathcal Z}{\mathcal Y}_2{\mathcal Y}_3\end{pmatrix},\quad \bc G=\begin{pmatrix}{\mathcal Y}_2{\mathcal Y}_3\\ 0\end{pmatrix}.\]
The iteration for generating series then reads
\begin{align*}
U&:=U+U\cdot\left(\begin{pmatrix}0&0\\1+zY_3&zY_2\end{pmatrix}\cdot U+\Id-U\right),\\
M&:=M+M\cdot \int{\overline{M}\cdot \begin{pmatrix}Y_3&Y_2\\0&0\end{pmatrix}\cdot U\cdot M-M'},\\
\overline{M}&:=\overline{M}+\overline{M}\cdot(\Id-M\cdot \overline{M}),\\
Y&:=Y+U\cdot M\cdot\int{\overline{M}\cdot\begin{pmatrix}Y_2Y_3-Y_2'\\ Y_2Y_3+(1+zY_3)Y_2'+(zY_2-1)Y_3'\end{pmatrix}}.
\end{align*}
The only costly operations needed are products of power series, the other ones (addition, differentiation, integration) having linear complexity.
Several of the entries of the Jacobian matrices are~0, so that not that many products are actually used in the iteration.

\paragraph{Mobiles}
These ``trees'' are defined by~${\mathcal Y}={\mathcal Z}+\int{\Cyc({\mathcal Y})}$: to the root is attached a cycle of similar trees, the labels increasing along the ``branches''. These were studied in~\cite{BergeronFlajoletSalvy1992} from the asymptotic point of view. The generating series does not appear to have a nice closed form. It satisfies the differential equation
\[y(z)=z+\int{\log\frac{1}{1-y(t)}\,dt}.\]
The species equation is reduced to purely elementary operations by introducing new species for $\Cyc(\bc Y)$ and using the fact that $\bc Y'=1+\Cyc(\bc Y)$, so that we consider
\[{\mathcal Y}_1={\mathcal Z}+\int{{\mathcal Y}_2},\quad {\mathcal Y}_2=\int{\Seq({\mathcal Y}_1)(1+{\mathcal Y}_2)}.\]
Now, in the notations of Theorem~\ref{newtdiffsys}, we have
\[\bc H=\begin{pmatrix}{\mathcal Z}\\ 0\end{pmatrix},\quad \bc G=\begin{pmatrix}{\mathcal Y}_2\\ \Seq({\mathcal Y}_1)(1+{\mathcal Y}_2)\end{pmatrix}.\]
This is very similar to the previous example.
The general recursion simplifies: since $\bc H$ does not depend on~$\bc Y$ the matrix $\bs U$ is $\Id$. Thus we are left with the recursion
\begin{align*}
M&:=M+M\cdot\int{\overline{M}\cdot\begin{pmatrix}0&1\\ \frac{1+Y_2}{(1-Y_1)^2}&\frac{1}{1-Y_1}\end{pmatrix}},\\
\overline{M}&:=\overline{M}+\overline{M}\cdot(\Id-M\cdot \overline{M}),\\
Y&:=Y+M\cdot\int{\overline{M}\cdot\begin{pmatrix}1+Y_2-Y_1'\\ \frac{1+Y_2}{1-Y_1}-Y_2'\end{pmatrix}}.
\end{align*}
Starting with initial value~$Y=0$, we get successively
\begin{align*}
Y_1^{[1]}&=z,\\
Y_1^{[2]}&=z+\frac{1}{2}z^2,\\
Y_1^{[3]}&=z+\frac{1}{2}z^2+\frac{1}{3}z^3+\frac{7}{24}z^4,\\
Y_1^{[4]}&=z+\frac{1}{2}z^2+\frac{1}{3}z^3+\frac{7}{24}z^4+\frac{3}{10}z^5+\frac{49}{144}z^6+\frac{173}{420}z^7+\frac{21059}{40320}z^8.\\
\end{align*}
More terms are easily obtained: each iteration has the cost of a few multiplication of power series.
A small constant factor could further be saved by introducing another variable for $\Seq({\mathcal Y}_1)$ so as to avoid the computations of the inverses of $1-Y_1$.

\subsection{Powersets}\label{sec:powerset}
The constructible classes of~\cite{FlSe09} also involve  a construction called \emph{PowerSet}, which is not a species \emph{per~se}. We show in this section that Newton's iteration can still be applied to species defined with it.
\subsubsection{Definition}
When counting unlabeled structures, the species~$\Set$ lets multiple identical unlabelled structures be counted. The principle of the \emph{PowerSet} construction, abbreviated $\PSet$, is to remove duplicates. 
An alternative construction, inspired by ``Vallée's identity''~\cite[p.~30]{FlSe09} is as follows. Given 
a species~$\mathcal F$, the equation
\[\Set(\mathcal{F}(\mathcal Z))=\mathcal{P}(\mathcal Z)\cdot\Set(\mathcal{F}(\mathcal Z^2))\]
defines a virtual species~$\mathcal{P}$ that we denote~$\PSet(\mathcal F)$. The intuition is to decompose the groups of duplicates according to their parity.
Note that $\PSet$ is not a species  (this is shown by examples below), but rather a mapping sending species to virtual species.

\subsubsection{Generating Series}
By Property~\ref{prop:plethysm}, the cycle index series of~$\mathcal{P}$ satisfies
\[Z_{\Set(\mathcal F)}(z_1,z_2,z_3,\dotsc)=Z_{\mathcal P}(z_1,z_2,z_3,\dotsc)Z_{\Set(\mathcal F)}(z_1^2,z_2^2,z_3^2,\dotsc),\]
so that 
\[Z_{\mathcal P}(z_1,z_2,z_3,\dotsc)=\exp\left(\sum_{k>0}{\frac1kZ_{\mathcal F}(z_k,z_{2k},z_{3k},\dotsc)}\right)\cdot \exp\left(-\sum_{k>0}{\frac1kZ_{\mathcal F}(z_k^2,z_{2k}^2,z_{3k}^2,\dotsc)}\right).\]
Setting~$z_k=z^k$ yields the ordinary generating series
\[\tilde{P}(z)=\exp\left(\sum_{k>0}{\frac1k\tilde{F}(z^k)}-\sum_{k>0}{\frac1k\tilde{F}(z^{2k})}\right),\]
which simplifies to
\[\tilde{P}(z)=\exp\left(\sum_{k>0}{\frac{(-1)^{k+1}}k\tilde{F}(z^k)}\right).\]
Thus we recover the classical formula used by~\cite{FlSe09}. A simple consequence is the following.
\begin{lemma}If the species~$\mathcal F$ has ordinary arithmetic complexity~$C_F$, then the species~$\PSet(\mathcal F)$ has ordinary arithmetic complexity~$C_F+O(\M)$.
\end{lemma}
\begin{proof}
The argument is the same as for sets and cycles in the proof of Proposition~\ref{prop:constructible}.
\end{proof}

Exponential generating series can be computed as well. Substituting~$z_1=z$ and~$z_k=0$ for~$k>1$ in the cycle index series gives
\[P(z)=\exp(F(z)-F(z^2)).\]
For instance, the special case~$\mathcal{F}=\mathcal{Z}$ leads to~$\exp(z-z^2)=1+z-z^2/2+O(z^3)$. The presence of a minus sign shows that~$\PSet(\mathcal Z)$ is a virtual species that is not a real species. Also, the fact that $P(z)\neq\exp(F(z)-F(z)^2)$ in general shows that $\PSet$ is not a species, for otherwise exponential generating series would compose.

\subsubsection{Newton's Iteration}
Since~$\PSet$ is not a species, its derivative is not defined, so that another approach is needed for Newton's iteration. The idea is to perform a first order Taylor expansion of~$\Set$ on both sides of
\[\Set(\mathcal F(\mathcal Z)+\mathcal U(\mathcal Z))=\PSet(\mathcal F+\mathcal U)\cdot\Set(\mathcal F(\mathcal Z^2)+\mathcal U(\mathcal Z^2)).\]
If~$\,\mathcal U=_k0$ then $\mathcal U(\mathcal Z^2)=_{2k+1}\mathcal U(\mathcal Z)^2=_{2k+1}0$, whence
\[\Set(\mathcal F(\mathcal Z))\cdot(1+\mathcal U(\mathcal Z))=_{2k+1}\PSet(\mathcal F+\mathcal U)\cdot\Set(\mathcal F(\mathcal Z^2))\]
and thus
\[\PSet(\mathcal F+\mathcal U)=_{2k+1}\PSet(\mathcal F)\cdot (1+\mathcal U).\]  
This is the key identity in the proof of the quadratic convergence of Newton's iteration. Thus, $\PSet$ behaves as if it had itself for derivative, exactly like~$\Set$, and the same iteration as for $\Set$ converges quadratically.

\subsubsection{Example}When~$\mathcal F$ is a flat species (see Section~\ref{subsec:flat}), then $\PSet(\mathcal F)$ is also equal to the flat part of $\Set(\mathcal F)$. Thus in this case, we recover species that can also be enumerated by the method of \emph{asymmetry index series}, for which we refer to~\cite[\S4.4]{BeLaLe98} and~\cite{Labelle1992a}.

An example of this type is provided by asymmetric rooted trees, with equation~$\mathcal A=\mathcal Z\cdot \PSet(\mathcal A)$. By application of the general rule outlined above, Newton's iteration is the same as for Cayley trees (Ex.~\ref{ex:unorderedrooted_series}), with the species $\Set$ replaced by the mapping $\PSet$:
\[\mathcal N(\mathcal Y)=\mathcal Y+\Seq(\mathcal B)\cdot(\mathcal B-\mathcal Y)\qquad\text{with}\qquad
\mathcal B=\mathcal Z\cdot\PSet(\mathcal Y).\]
The associated Pólya operator is given by
\[\Phi_{{\mathcal N}_{\mathcal H}}:Y(z)\mapsto Y(z)+\frac{B(z)-Y(z)}{1-B(z)}\quad\text{with}\quad B(z)=z\exp\!\left(Y(z)-\frac12Y(z^2)+\dotsb\right).\]
The first few steps give:
\begin{align*}
\tilde{G}^{[1]}&=\mathbf{z+z^2+z^3}+z^4+z^5+z^6+z^7+z^8+z^9+z^{10}+z^{11}+z^{12}+\dotsb,\\ 
\tilde{G}^{[2]}&=\mathbf{z+z^2+z^3+2z^4+3z^5+6z^6+12z^7+25z^8+52z^9}+111z^{10}+237z^{11}+507z^{12}+\dotsb,\\
\tilde{G}^{[3]}&=\mathbf{z+z^2+z^3+2z^4+3z^5+6z^6+12z^7+25z^8+52z^9+113z^{10}+247z^{11}+548z^{12}}+\dotsb.\\     
\end{align*}
Thus we recover a classical sequence~\cite[p.~330]{BeLaLe98} for which we obtain the first terms in quasi-optimal complexity.

%
%

\section*{Acknowledgements}
Gilbert Labelle gave us many encouragements and useful comments on an early version of this work.

%
\def\cprime{$'$} \def\gathen#1{{#1}}


\end{document}